\documentclass[10pt]{article}

\usepackage{amsmath,amssymb}

\usepackage{amsthm,amsfonts,amssymb,euscript}
\usepackage{latexsym, multicol, fancybox}
\usepackage{graphicx}
\usepackage{color}
\usepackage{amsmath, amsthm, amssymb,bm}
\usepackage{epstopdf}
\usepackage{caption}
\usepackage{psfrag}
\usepackage{comment}
\usepackage[]{upgreek}
\usepackage[normalem]{ulem}
\usepackage{cancel}

\usepackage{amsmath, amssymb}
\usepackage{pgfplots}

\usepackage{mathrsfs}
\usepackage{xcolor}
\usepackage[citebordercolor={green}]{hyperref}
\usepackage{tikz}
\usepackage{bbm}
\usepackage{dsfont}

\numberwithin{equation}{section}

\newtheorem{theorem}{Theorem}[section]
\newtheorem{lemma}[theorem]{Lemma}
\newtheorem{proposition}[theorem]{Proposition}
\newtheorem{corollary}[theorem]{Corollary}
\newtheorem{definition}[theorem]{Definition}
\newtheorem{remark}[theorem]{Remark}

\newcommand{\bea}{\begin{eqnarray}}
	\newcommand{\eea}{\end{eqnarray}}
\newcommand{\beq}{\begin{equation}}
	\newcommand{\eeq}{\end{equation}}
\def\beaa{\begin{eqnarray*}}
	\def\eeaa{\end{eqnarray*}}
\def\ba{\begin{array}}
	\def\ea{\end{array}}

\DeclareFontFamily{U}{mathx}{\hyphenchar\font45}
\DeclareFontShape{U}{mathx}{m}{n}{
	<5> <6> <7> <8> <9> <10>
	<10.95> <12> <14.4> <17.28> <20.74> <24.88>
	mathx10
}{}
\DeclareSymbolFont{mathx}{U}{mathx}{m}{n}
\DeclareFontSubstitution{U}{mathx}{m}{n}
\DeclareMathAccent{\widecheck}{0}{mathx}{"71}

\newcommand{\Romanupper}[1]
{\MakeUppercase{\romannumeral #1}}

\def\a{{\alpha}}

\def\b{{\beta}}
\def\be{{\beta}}

\def\Ga{\Gamma}
\def\de{\delta}
\def\De{\Delta}
\def\ep{\epsilon}

\def\la{\lambda}

\def\si{\sigma}
\def\Si{\Sigma}
\def\om{\omega}

\def\varep{\varepsilon}
\def\vphi{{\varphi}}

\def\th{\theta}

\def\ze{\zeta}

\def\nab{\nabla}

\def\pr{{\partial}}
\def\les{\lesssim}
\def\c{\cdot}

\def\BB{{\mathcal B}}

\def\NN{{\mathcal N}}

\def\II{{\mathcal I}}

\def\FF{{\mathcal F}}
\def\EE{{\mathcal E}}
\def\HH{{\mathcal H}}

\def\TT{{\mathcal T}}

\def\OO{{\mathcal O}}

\def\KK{{\mathcal K}}

\def\DD{{\mathcal D}}

\def\RR{{\mathcal R}}
\def\QQ{{\mathcal Q}}

\def\HH{{\mathcal H}}

\def\D{{\bf D}}

\def\h{{\bf h}}
\def\H{{\bf H}}

\def\g{{\bf g}}

\def\T{T}
\def\Z{Z}

\def\SSS{{\Bbb S}}
\def\RRR{{\Bbb R}}

\def\CCC{{\Bbb C}}
\def\ZZZ{{\Bbb Z}}
\def\f12{{\frac 1 2}}

\def\Wk{\mathfrak{W}}

\def\Lk{\mathfrak{L}}

\def\sk{\mathfrak{s}}

\def\dk{\mathfrak{d}}

\def\Div{{\mbox{ \bf Div} \hspace{- 0.2pt}}}

\def\ov{\overline}

\def\f12{\frac 1 2}
\def\lab{\label}

\def\bsplit{\begin{split}}

	\def\Rhat{{{\widehat R}}}	
	
	\def\rhat{{\widehat{r}}}

	\def\gt{\widetilde{\g}}
	\def\tt{{\widetilde{t}}}
	  \def\hbt{{\widetilde{\h}}}
	  \def\Hbt{{\widetilde{\H}}}
	  
	  \def\KKt{\widetilde{\KK}}

	\def\ntrap{trap\mkern-18 mu\big/\,}

	\def\DDntrap{{\, ^{(\ntrap)}\DD}}

	\def\ntrap{trap\mkern-18 mu\big/\,}

	\def\BE{B\hspace{-2.5pt} E}

	\def\BP{B\hspace{-2.5pt} P}
	\def\BPE{B\hspace{-2.5pt}P \hspace{-2.5pt} E}

	\def\psit{\widetilde{\psi}}

	\def\Fdeg{\,^{(deg)} F}
	
	\def\wt{\widetilde}

	\def\Rt{\widetilde{R}}
	\def\rt{{\widetilde{r}}}
	\def\ut{{\widetilde{u}}}
	\def\vt{{\widetilde{v}}}
	\def\wt{{\widetilde{w}}}

	\def\qt{{\widetilde{q}}}
	\def\gt{{\widetilde{\g}}}
	\def\at{{\widetilde{\a}}}
	\def\bt{{\widetilde{\b}}}
	\def\vphit{{\widetilde{\vphi}}}
	\def\tht{\widetilde{\th}}

	\def\Nt{\widetilde{N}}

\def\DDtrap{\,^{(trap)} \DD}

\def\phim{{\phi_{-}}}
\def\phip{{\phi_{+}}}
\def\square{\Box}

\def\ghat{\widehat{\g}}	

\def\psihat{\widehat{\psi}}
\def\Sit{\widetilde{\Si}}

\def\DDt{\widetilde{\DD}}

\def\DDt{{\widetilde{\DD}}}

\def\Lk{\mathfrak{L}}
\def\Lks{\Lk^{[\sk]}}
\def\Lkshat{\widehat{\Lk}^{[\sk]}}

\def\as{\upalpha^{[\sk]}} 
\def\As{A^{[\sk]}}
\def\Bs{B^{[\sk]}}
\def\Vs{V^{[\sk]}}
\def\psits{\psit^{[\sk]} }

\def\Lkst{\widetilde{\Lk}^{[\sk]}}

\def\psihat{\widehat{\psi}}
\def\HHt{{\widetilde{\HH}}}
\def\DDt{{\widetilde{\DD}}}
\def\IIt{\widetilde{\II}}
\def\rhot{\widetilde{\rho}}
\def\taut{\widetilde{\tau}}

  \def\ntrap{trap\mkern-18 mu\big/\,}
 \def \tauh{\widehat{\tau}}
  
\begin{document}

		\title{Generalized Whiting transform and           Flux  Estimates  for   Wave Equations in  Subextremal Kerr 
		 }
		\author{Lili He  and Sergiu Klainerman }

		\maketitle
		
		\begin{abstract}
		All   recent   advances  
		\cite{DRS}, \cite{S-Rita1}, \cite{S-Rita2}, \cite{Millet}, \cite{MaS}, \cite{MaS2}  on the  boundedness and decay for  wave  equations  on subextremal  Kerr  spacetimes  $\KK(a,m)$,  
		 with  non-small  angular  rotation,   rely in an essential way on  Whiting's  groundbreaking    discovery of a powerful
		   integro-differential  transformation    with   the help of   which  he       ruled out exponentially growing modes   in \cite{W} (see also \cite{Yacov}, \cite{AMPW} and \cite{Rita} for further developments).     We give  a purely   \textit{physical space}, generalized version,    of the transformation  which takes general solutions of {spin-$\sk$} wave equations  in $\KK=\KK(a,m)$  to solutions  of  a  secondary  wave equation on a new  metric  manifold $\KKt=\KKt(a,m)$ with two stationary, asymptotically flat,  ends. Moreover,   for a certain  subdomain  $\DDt\subset\DDt_-$,  which can be identified as   the exterior  of a black hole region,  the new metric $\gt$  of $\KKt(a, m)$  is Lorentzian   and the time  translation $\T  $ is timelike.  Using the  geometric properties  of the  two  ends  one  can then deduce,  by   classical    asymptotic Fourier  analysis techniques,      the desired  bound  for  the    flux of  $\psi$ at the future event horizon.
		  This estimate played  a crucial role   in  our   companion paper  \cite{He-K1}.
		   
			\end{abstract}

		\tableofcontents

		
		
		\section{Introduction}

		
		\subsection{Broad perspective}
		The analysis of solutions to wave equations in the exterior domain
		$\DD=\{r>r_+\}$ of a Kerr spacetime $\KK(a,m)$, $|a|<m$, is
		greatly hampered by the presence of the ergoregion, a subdomain of
		$\DD$ in which the vectorfield $\T$ is spacelike.
		Consequently, the $\T$-energy fails to be coercive, leaving open, in
		principle, the possibility of rapidly growing mode solutions which
		would render the Kerr solution unstable.  In his seminal paper
		\cite{W}, Whiting produced a remarkable  integro-differential transformation which associates
		to a mode solution of the form
		\[
		\psi=e^{-i\om t}e^{in\phi}S(\th)R(r),
		\]
		separated relative to the Boyer--Lindquist (BL) coordinates, a corresponding
		mode solution of the form
		$e^{-i\om t}e^{in\phi}S'(\th)R'(r)$ of a similar wave equation
		corresponding to a new Lorentzian metric $\g'$, possessing the same
		 $ (T, Z)$ symmetries   as the original Kerr metric \(\g\), but for which,
		miraculously, no ergoregion\footnote{In fact the time translation
		$\T$ for the Whiting metric is everywhere null in the domain of outer
		communication.} is present.  Using his
		transformation\footnote{In \cite{W}, Whiting makes in fact two
		transformations: an integral transformation from $R(r)$ to $R'(r)$,
		and a differential transformation from $S(\th)$ to $S'(\th)$.},
		Whiting was able to show that there are no exponentially growing modes
		(i.e. with $\Im(\om)>0$) for all spin-$\sk$ wave equations in Kerr.

Whiting's result was extended by  Shlapentokh-Rothman
\cite{Yacov}, for zero spin $\sk=0$, and by \cite{AMPW}, for arbitrary
spin, to exclude nonzero modes in the entire upper half-plane.  Moreover
\cite{Yacov} made the mode stability statement quantitative and  derived a bound for the future horizon flux\footnote{Later extended to the
Teukolsky equation in \cite{Rita}.}  restricted to
bounded mode frequencies.  Together with    the separation between superradiant and trapping
frequencies in the  high-frequency region,  this  bound played a  crucial role in  \cite{DRS}  to derive decay for scalar wave equations
 in the full sub-extremal case.
  The analogous
frequency-space analysis for the spin-2 Teukolsky equation was carried out
in \cite{S-Rita1}, \cite{S-Rita2}, while \cite{Millet} gives a substantially
simpler proof of the same boundedness and decay result by Hintz--Vasy
microlocal techniques.\footnote{Note that \cite{Millet} does not contain
energy bounds, a feature typical of the Hintz--Vasy method.}

	       These above  results rely in an essential way on frequency
analysis  and as such  they appeared difficult to adapt directly to realistic perturbations
of Kerr, of the type expected in a proof of nonlinear stability. This
difficulty has recently been overcome by Ma and Szeftel in \cite{MaS} and
\cite{MaS2}, where they derive full energy--Morawetz estimates for both the
scalar wave equation and the Teukolsky equation on realistic perturbations of
subextremal Kerr, using the results of \cite{DRS} and \cite{Millet} as black
boxes.

  Broadly speaking, there are two  approaches to the  nonlinear   black hole stability problem. The first treats the
 linearized Einstein equations, after a suitable choice of  a wave map gauge,
 as a coupled system of wave equations on metric perturbations.  This approach,  followed most
 prominently by Hintz--Vasy,  relies on  microlocal and spectral methods to obtain estimates for  general systems of linear wave
 equations, modulo lower order terms.  A crucial part  of their  argument is the microlocal treatment,
 by semiclassical methods, of the normally hyperbolic trapping set. 
   To remove  the dependence on the   lower order  terms, and thus obtain an invertibility estimate for the  linearized 
   Einstein equations, they  appeal to mode stability.\footnote{   Whiting mode stability  result allows them to  control  the bounded frequencies. The treatment of the  very low frequencies requires  additional spectral-theoretic techniques.}  It is important to note   that in this part of the argument they need  to  make use of the specific algebraic  structure  of the linearized  Einstein equations because they rely   on Whiting's   mode stability result for the    Teukolsky wave  equation.\footnote{Mode stability for the linearized Einstein equation on Kerr, for metric perturbations, is obtained by reducing it to the mode stability of Teukolsky equations, see \cite{AHW}.  It is notable that the Hintz-Vasy technique requires  mode stability even  for small angular momentum. Examples of   other  linear   results  based on mode stability  are  \cite{HHV21},  
 \cite{He26} and \cite{Millet}.}
      The first important success  of that approach  is  
		 the proof of the nonlinear stability theorem for the Kerr--de Sitter family
	\cite{HV18b}, based on  a specific version  of mode stability on 
		Schwarzschild--de Sitter.  An  more sophisticated   adaptation  of this strategy was
recently used by Hintz \cite{Hintz} to  derive   a  version of the  nonlinear stability of Kerr
in the full subextremal range.\footnote{ The result holds for a restricted set of initial conditions.   See the introduction of
\cite{He-K1} for a comparison between Hintz \cite{Hintz} and Ma--Szeftel
\cite{MaS2}.} Finally, it is important to note that, in both  nonlinear  results, the Hintz-Vasy technique  leads to
  a loss of derivatives in the linear estimates. This  forces them to appeal to  a Nash-Moser  argument.

	  The second approach  uses  more closely the  special  algebraic properties  of the Kerr  background to
 pass first to decoupled Teukolsky equations for the  gauge-invariant curvature
 components.  To  recover the remaining curvature  components, as well as the   metric perturbation components, it relies on a  dynamically  dependent   gauge  condition.
  Its central  achievement  is
 the nonlinear stability of Kerr for small angular momentum, contained in 
 \cite{KS-GCM1}, \cite{KS-GCM2}, \cite{KS:Kerr}, \cite{GKS}, and
 \cite{Shen}.  This result, based entirely on physical space methods, 
  does not require  mode stability.  It was  however widely expected  that 
 the full subextremal range is not  treatable by the existing
 purely physical-space methods alone, requiring  instead   delicate microlocal
 analysis techniques, centered on some version of  Whiting's mode stability
 result.\footnote{For another  recent result, in the second approach
 framework, using mode stability  as an essential ingredient, see
 \cite{ABBM}.}
 The recent results of Ma--Szeftel appear  to confirm this expectation.  As
 noted in the introduction of \cite{He-K1}, their result in \cite{MaS2},
 combined with the results and methods in \cite{KS-GCM1}, \cite{KS-GCM2},
 \cite{KS:Kerr}, \cite{GKS}, and \cite{Shen}, settle the Kerr stability
 problem in the physically realistic situation of perturbations of Kerr
 initial data\footnote{This is the same type of data used in the proof of the
 nonlinear stability of the Minkowski space \cite{CK}, compatible with
 non-smooth scri.} which decay like $r^{-3/2-\de}$ as $r\to \infty$.  Finally, these stability  results   are based on  bootstrap arguments 
 and as such no losses of derivatives are allowed.

 In  \cite{He-K1}  we have initiated a program  of extending the physical space methods used in \cite{GKS} to the full subextremal case.
  An essential ingredient of that program  is  the development of a physical space version of the Whiting transform,   based on which we can  control the horizon flux of  solutions.

       In this work we deliver on that promise. To start with, we   introduce  a   class of \textit{physical space}, integral,  transformations $  \Wk  $
               that take   solutions $\psi$  of a   spin $\sk$  wave equations in $\KK(a, m)$     to solutions $\psit=\Wk[\psi]$  of a  similar    equation  for a new  metric   $\gt$    defined on a   manifold   $\KKt=\KKt(a,m)$  with two asymptotically flat branches  $\KKt=\DDt_{-}\cup \DDt_+$.   The metric   preserves     \textit{all  the  symmetries} of the Kerr metric and,      for a certain  subdomain $\DDt\subset\DDt_-$  which can be identified as   the exterior  of a black hole region in $\KKt$,  the metric $\gt$ is Lorentzian   and the time  translation $\T  $ is timelike.

Using these remarkable  geometric properties of $\gt$ on $\DDt$    we can derive  a  coercive  bound  for the   flux of  $\psit=\Wk[\psi]$ at the future  null infinity of  $\DDt$.     We can then deduce,  by   classical    asymptotic Fourier  analysis techniques,        a bound  for  the frequency localized  flux of the original  solution $\psi$ at the future event horizon.  More precisely, with the notation  used  in our previous paper  \cite{He-K1},   we  are  able to control  $\Fdeg[\T^2\psi]$,   the degenerate  flux of $\T^2\psi$,  when the frequencies of $\psi$ with respect to $(\T, \Z)$  are either    bounded\footnote{Thus recovering the  main result of  \cite{Yacov}.}   or non-superradiant (in the terminology of section~3.5.2 of \cite{He-K1}).
                
                To derive information about the remaining  unbounded superradiant
                frequencies, we need to consider  the second branch
                $\DDt_+$ of $\KKt$.  Although $\gt$ is not Lorentzian
                throughout $\DDt_+$, its  far region  is Lorentzian
                and possesses a well-defined future null infinity.
                Remarkably, using the special  geometric properties of
                $(\DDt_+,\gt)$, we can still derive a coercive bound for the
                flux of $\psit$ at this second future  null infinity and use it, by
                a similar asymptotic analysis, to control
                $\Fdeg[\T^2\psi]$ in the superradiant frequency regime.

                To make this strategy work, and to make sense of the integral
                expressions defining $\Wk$, we localize $\psi$ with
                respect to a suitably chosen time function.  This
                localization generates source terms in the transformed
                equation.  To estimate them, we  rely on  integrated bilinear
                identities for the transformed fields, see Proposition
                \ref{prop:bilinear-outgoing-general}.  The
                resulting $L^2$ estimates allow us
	                to transfer the  source terms back to the original
	                Kerr spacetime, where they can be controlled by the energy,
	                $r^p$, and Morawetz norms used in \cite{He-K1}.

		             A slightly modified version of our integral transform also
		             applies to higher spin wave equations.  In particular, for the
		             spin $\sk=\pm2$ Teukolsky equations, the transformed equations
		             \eqref{eq:Lkst-psits} have a strikingly simple structure: the
		             full transformed operator $\mathfrak{P}_{\gt}^{[s]}$, defined in
		             \eqref{eq:def-Pgt-spin}, consists of the wave operator of the
		             transformed metric $\gt$, together with explicit lower order
		             spin-dependent terms which preserve the $(T,Z)$
		             symmetries, has a Lagrangian structure  which,  allows one to derive coercive energy estimates in $\DDt$.
		             Remarkably the  form of  the operator   $\mathfrak{P}_{\gt}^{[s]}$ is identical to that of the  main part of the 
		             Chandrasekhar  transformed spin $\sk$ wave equation  in Kerr, according to Ma's formulas in \cite{MaLG}. In that sense 
		              $\mathfrak{P}_{\gt}^{[s]}$  is considerably better since, in addition to the main (Lagrangian)   part,  the Chandrasekhar-Ma transformed equation contains  a non-trivial  coupling  with the original  Teukolsky variables,  see Remark \ref{rem:Ma-comparison}.\footnote{The analogy decribed   here is purely formal; our
		             transformation is integral rather than differential (as is the case of the Chandraseckhar-Ma  transformation)  and is
		             adapted to  a new  geometry.}

			               Based on the coercive   estimates available for  the  transformed operator   $\mathfrak{P}_{\gt}^{[s]}$
						       the     proof of the horizon-flux estimates  follows  as in the spin $0$ case.   For simplicity  we 
						        describe the main details  of the proof of Theorem           \ref{thm:cutoff-horizon-flux-estimate} only in the spin $\sk=0$  case; the same argument extends to the higher spin equations.

In summary, our new  physical-space version of the  Whiting transform can be regarded as a
mirror correspondence between two different metric structures, extending,  in  a striking  way, the internal symmetries
 of the Kerr family.  It converts
horizon fluxes for the original solution into radiation fluxes at the
asymptotic ends of the transformed geometry.  The two relevant branches of
$\KKt$ then act as complementary mirrors, reflecting  two different  frequency regimes
which together make up the full horizon flux.

One can reasonably ask whether this one-sided mirror symmetry can be
extended to a full duality.  The enemy in  this regard  is  presence of the $r$-Hilbert transform
on the  right hand side of the  integrated bilinear identities  of Proposition  \ref{prop:bilinear-outgoing-general}. 
To eliminate it  would require    some control   for  $\psit$ on the entire manifold $\KKt$.  Though this is  improbable,
 one can show that a weaker form of duality  is valid, based on the fact that we  can derive  coercive  energy  bounds 
 for $\psit$ on  $\KKt\setminus\{ X\in [-\de, 0]\}$, see  Remark
\ref{rem:bil.smallangle}.

In our  forthcoming   paper  \cite{He-K3}  we plan   to derive Morawetz--energy estimates directly  for  the transformed  $\psit$ and, together  with   the partial duality   mentioned above, use them      to vastly simplify the results of \cite{He-K1}.  More importantly, since the transformed nonzero-spin equations  have  the  favorable geometric structure mentioned above, this strategy
should also extend to derive Morawetz-energy  estimates  for  the more important  Teukolsky equations. Finally, given the  physical space nature of our technique, our results   can be  easily extended  to realistic perturbations of     Kerr.

	       \subsection{Statement  of   the   Main Theorem and  strategy of  the proof}
	       		
\begin{theorem}
\lab{thm:mainthm}
Let $\psi$ be a solution of $\square_\g\psi=N$. For every $\epsilon>0$, we have for    $p>0$\footnote{See  \eqref{eq:defofnorms} for  the definition of the norms used here. Note that the $E_p$ norms  differs slightly from  the  corresponding norms in \cite{He-K1}. }  and     $r_0>0$   sufficiently large.
\beq\lab{eq:Thm-mainintro}
\bsplit
\int_{\HH^+(0,\tauh)}
\big|T_+T^2\psi\big|^2\,dv\,d\omega
&\leq \epsilon \BPE_p^2[\psi](0,\tauh)+C(\epsilon)\int_{\widehat{\tau}}^{\widehat{\tau}+1}
E_{\leq r_0}^2[\psi](\widehat{s})\,d\widehat{s}\\
&\quad+C(\epsilon)\Big(E_p^2[\psi](0)+\NN_p^2[\psi, N](0,\tauh)\Big).
\end{split}
\eeq
\end{theorem}		

The proof of the Theorem \ref{thm:mainthm}  is based 	on the following ingredients:
\begin{itemize}
\item A new  notion of integral transformation, which we call  generalized Whiting transform,   and the manifold structure $\KKt$ it generates. 
\item The metric $\gt $  induced     on $\KKt$    and the relation  between $\square_\g \psi=N$  and $\square_\gt \psit=\Nt$.
   \item       Use the  remarkable  geometric properties of $\gt$ to derive   flux estimates  for  $\psi$.
\end{itemize} 
		

		\subsubsection{Generalized Whiting transform  and   transformed manifold } 
		
		
		 We  give the precise form of the physical-space transform used in
		the proof, in the simpler case of the scalar wave equation
		$\square_\g\psi=0$.  We use the  outgoing Eddington--Finkelstein
		coordinates $(u,r,\theta,\phim)$, write
		$\Delta=(r-r_+)(r-r_-), r_-<r_+$, and recall that
		$|q|^2\g^{-1}=\h+O$, with
		\[
		\h^{rr}=\Delta,\qquad
		\h^{ur}=-(r^2+a^2),\qquad
		\h^{r\phim}=-a,\qquad
		\h^{uu}=\h^{u\phim}=0,
		\]
		the remaining angular part being denoted by $O$.  In the canonical
		representation\footnote{This is the canonical outgoing specialization
		of the general family of transforms introduced in Definition
		\ref{Def:WhitingTr}.} the general
		transform $\Wk$ takes the form
		\[
		\Wk[\psi](\vt,\rt,\tht,\vphit)
		=\int_{r_+}^{\infty}
		\psi\left(\vt+X(r-r_-),r,\tht,\vphit\right)\,dr,
		\qquad
		X=-2\frac{\rt-r_-}{r_+-r_-}.
			\]
			The parameter $X$ is allowed to range over all of $\RRR$.  Thus
			the transform defines a new function on a new two-ended manifold $\KKt$,
			with transformed time coordinate $\vt$ and radial coordinate $\rt$. 	
			
			 We apply this transform after localizing $\psi$ relative to an appropriate 
			time function. We also  assume
			that $\psi$ has compact support in the original radial variable
			$r$.  These assumptions make the integral defining $\Wk[\psi]$
			well defined and also remove the endpoint contribution at
			$r=r_+$ in Lemma \ref{Lemma:Whiting-metric-general}.		
			
		\begin{remark}
 We note that our  generalized Whiting transform defined in \eqref{eq:notationoftransgeneral}  is closely connected   to the  John's transform \cite{John}  except  that, as in the case of the Radon transform,  we restrict to  rays in  the two dimensional  space spanned by  $u, r$.
   \end{remark}

			\subsubsection{Transformed metric}
			\lab{section;1.2.2}

			The transform intertwines the Kerr wave operator with a wave operator
			for a new metric $\gt$.  More precisely,
		\[
		\Wk[|q|^2\square_\g\psi]
		=
		|\qt|^2\square_\gt\Wk[\psi].
		\]
		The rescaled inverse metric has the same  splitting  structure as Kerr, $|\qt|^2\gt^{-1}=\hbt+O$, with the nonzero components of $\hbt$, 
		\[
		\hbt^{\rt\rt}=\Delta(\rt),\quad
		\hbt^{\vt\rt}=\rt^2-r_-^2,\quad
		\hbt^{\vt\vt}=-\frac{8mr_-(\rt-r_-)}{r_+-r_-},\quad
		\hbt^{\vt\vphit}=-\frac{2a(\rt-r_-)}{r_+-r_-},
		\]
		with the symmetric components understood and with the angular
		operator $O$ unchanged.

			The transformed geometry has two distinct exterior branches.
			The branch $X<-2$ is Lorentzian, asymptotically flat, and is the
			exterior of a regular subextremal black hole with event horizon given by $X=-2$. Moreover on this branch the
			Killing field $\T$ is timelike in the exterior and null at
		$X=-2$.   The branch
		$X>0$ contains a second asymptotically flat Lorentzian end, beyond
		the degeneracy hypersurface $X=X_+(\th)$ (see Proposition
		\ref{Prop:outgoing-Snegative-regions}), and \(\T\) is timelike in that
			end.  On this positive branch there is no proper horizon; instead
			the branch terminates at \(X=0\).  Between \(X=0\) and the Lorentzian
			end $X\to \infty$,  $\gt$ changes character, and \(X=0\) is a genuine
			singular hypersurface. 
			
			The following very approximate  figure  on the left  below  may      be  helpful 
			 to convey the main features   structure  of  $(\KKt, \gt)$.   It is also helpful to compare it to the case of Schwarzschild, 
			  in which case the transformed metric  $\gt$  coincides  with the  original Schwarzschild  metric.
		\begin{figure}[!ht]
			\includegraphics[width=7.0in]{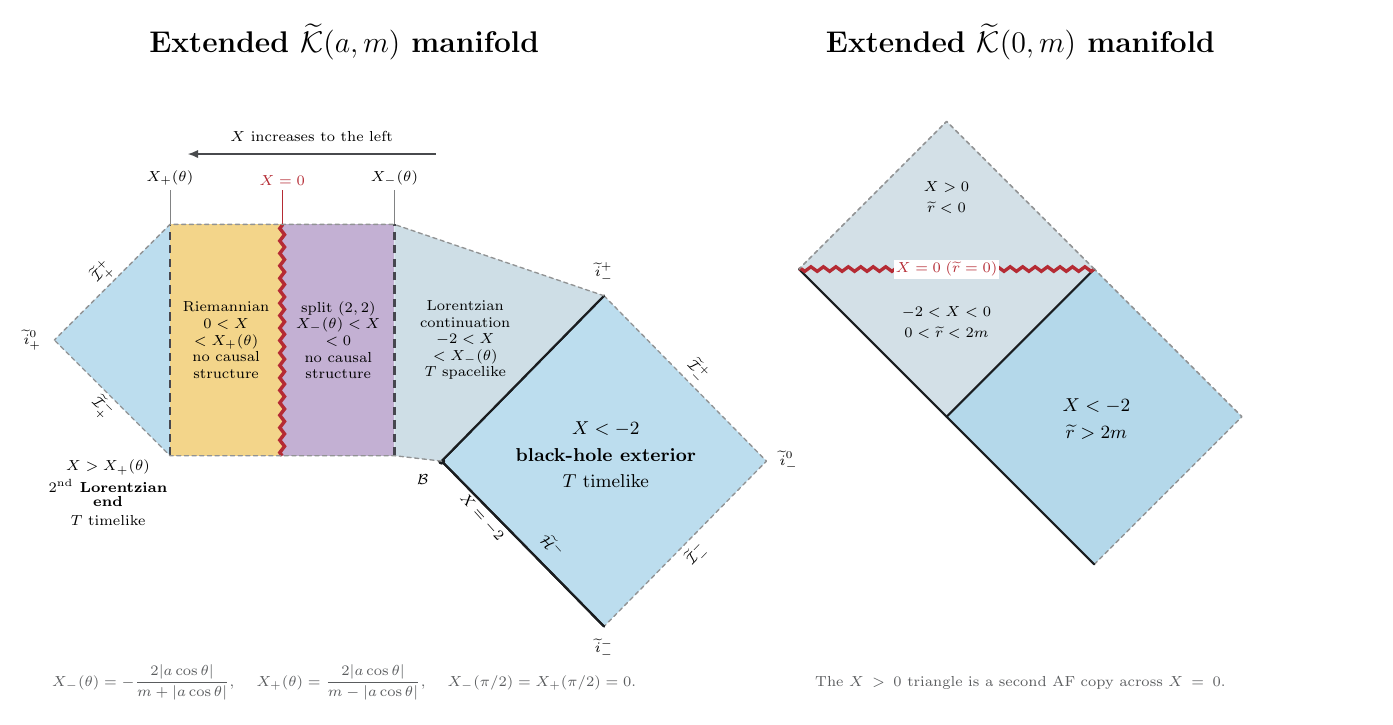}  
			\caption{Extended $\KKt(a,m)$ and $\KKt(0,m)$ manifolds.}
			\lab{fig:KKt}
						\end{figure}

		\subsubsection{Main  estimates  in the proof  of Theorem \ref{thm:mainthm}}

		Using conformal energy estimates ((Proposition \ref{Prop:conf-en(sk-0)}), we control the flux through the future null infinity on the negative branch $X<-2$. Using weighted conformal energy estimate (Proposition \ref{Prop:conf-en(sk-0)Xpos}) and localized elliptic estimate (Proposition \ref{prop:positive-localized-elliptic-estimate}), we control the flux through the future null infinity on the positive branch $X>0$.					
			
			Based on  these  fluxes, the two branches control complementary frequency ranges of the
		original horizon flux.  The radiation field at $\IIt^+(X<-2)$
		controls the non-superradiant frequencies, together with the bounded
		superradiant frequencies, see \eqref{eq:controlled-frequency-region} and
		Proposition \ref{prop:psi-psit-flux-strong-Xminus}.  The radiation field at $\IIt^+(X>0)$
		controls the remaining large superradiant frequencies, see Proposition \ref{prop:positive-branch-large-superradiant-flux}.

Combining these we derive	 the following full horizon flux estimate (see Proposition \ref{prop:combined-horizon-flux}) 	
	\beq
	\lab{eq:fullhorizonintro}
	\int_{\HH^+}
|T_+T^2(\chi\psi)|^2\les \text{ inhomegeneous terms }|\qt|^2\Nt +\text{ lower order terms}.
	\eeq
where $|\qt|^2\Nt=\Wk[|q|^2N]$ and $N$ is the source term of the wave equation verified by the $\chi$ cut-off of $\psi$.
We note that,  although frequency decompositions\footnote{The frequency decomposition required for  each of the two branches, in Proposition
		\ref{prop:psi-psit-flux-strong-Xminus} for the negative branch and
		Proposition \ref{prop:positive-branch-large-superradiant-flux} for the
		positive branch,   is internal to the proof. The  term  on
		the left is the full, geometrically defined,   horizon flux. }   are central to the proof, the final result in the main theorem  is a purely physical-space estimate.
		
				  Theorem \ref{thm:cutoff-horizon-flux-estimate} states
		that, for an appropriate  time cut-off  $\chi$, 
		 the right side of \eqref{eq:fullhorizonintro} and thus the  flux integral  
		 \[
		 \int_{\HH^+(0,\tauh)}
		\big|T_+T^2\psi\big|^2
		\]
		can be  estimated   by the type of  norms  used in the
		Morawetz--energy  estimates of   \cite{He-K1}.  More precisely, it can be estimated by a small multiple of the main spacetime norms,
		namely the $\BE^2[\psi]$ norm plus  the trapping norm $P^2[\psi]$, together with a large multiple of initial energy and compact energy norm on final time slab,
		and $r_0^{-p}$ multiple of $r^p$ weighted energy norm.  The latter error
		is made absorbable by choosing $r_0$ large.

	
	\section{Preliminaries}\lab{section:preliminaries-chapter5}


\subsection{Basic facts about the Kerr metric}
\lab{section:Kerr-metric}
The  Kerr metric  $\g=\g_{a,m}$ in  Boyer-Lindquist  (BL) coordinates 
is  given by 
\beq
\lab{eq:coordsBL}
   -\frac{\left(\Delta-a^2\sin^2\theta\right)}{|q|^2}dt^2-
\frac{ 4 amr }{|q|^2}   \sin^2\theta     dt  d\phi+\frac{|q|^2}{\Delta}dr^2+ |q|^2 d\theta^2+
\frac{ \Si^2}{|q|^2}\sin^2\theta
d\phi^2
\eeq
where $q=r+ i a \cos\th,$ \,$\De= r^2-2mr+a^2$ and  $\Sigma^2 = (r^2+a^2)^2-a^2(\sin\theta)^2\Delta$ 
and the  volume form is $ d\mu=|q|^2 \sin\th dt dr d\th d\phi$. The roots of  the quadratic  polynomial $\De  $       are   $r_{-}<  r_+$ with $r=r_+=m+\sqrt{m^2-a^2}$
denoting the event horizon  $\HH$.

The inverse  metric is given by
\beq
\lab{eq:inversemetric-Kerr}
-\frac{\Si^2}{|q|^2 \De}\pr_t^2-\frac{4 a mr }{|q|^2\De}\pr_t\pr_\phi+\frac{\De- a^2 \sin^2\th}{|q|^2 \De \sin^2 \th}\pr_\phi^2 
+\frac{\De}{|q|^2}\pr_r^2+\frac{1}{|q|^2}\pr_\th^2.
\eeq

The inverse metric  can also be written in the form
\beq\lab{inverse-metric}
|q|^2 \g^{\a\b}=\h^{\a\b}+O^{\a\b}
\eeq
Where
\beq
\lab{eq:h-metricBL}
\h^{rr}=\De, \quad \h^{tt}=-\frac{(r^2+a^2)^2}{\De}, \quad \h^{t\phi} =-\frac{a(r^2+a^2)}{\De}, \quad \h^{\phi\phi}=-\frac{a^2}{\De}
\eeq
and 
\beq
\lab{equation:def-O}
\bsplit
O^{\a\b}&= \pr_\th^\a \pr_\th^\b +\frac{1}{\sin^2\th} \pr_\phi^\a \pr_\phi^\b +  2 a \pr_t^{(\a} \pr_\phi^{\b)}+ a^2\sin^2 \th \pr_t^\a \pr_t^\b\\
&=\pr_\th^\a \pr_\th^\b + \frac{1}{\sin^2\th}\big(a\sin^2\th \pr_t +\pr_\phi\big)^{\a}  \big(a\sin^2\th \pr_t +\pr_\phi\big)^{\b}
\end{split}
\eeq

Using the above   decomposition  the scalar, spin $0$,  wave operator takes the form
\beq
\lab{eq:decompose-square}
|q|^2 \square_{a,m} =\pr_r\big(\De \pr_r \big) - \frac{1}{\De}\Big( (r^2+a^2)\pr_t + a\pr_\phi\Big)^2+\OO
\eeq
with  $\OO$ the second order operator
\beq
\lab{eq:RR-OO}
\OO=\frac{1}{\sin \th}\pr_\a(\sin \th O^{\a\b}\pr_\b )=\frac{1}{\sin\th}\pr_\th\big(\sin \th \pr_\th)+\frac{1}{\sin^2\th} \Big(a\sin^2\th \pr_t +\pr_\phi\Big)^2.
\eeq

\subsection{Eddington-Finkelstein  (EF)  coordinates}

\subsubsection{Ingoing PN frame and adapted EF(+) coordinates}
\lab{section:ingoingPNframe}


The ingoing PN  frame (with $\D_3 e_3=0$), regular toward the future  for all $r>0$,  is given by the null pair
\beq
\lab{eq:null-pair-in}
e^{(in)}_4=\frac{r^2+a^2}{|q|^2} \pr_t +\frac{\De}{|q|^2} \pr_r +\frac{a}{|q|^2} \pr_\phi, \qquad
e^{(in)}_3   =\frac{r^2+a^2}{\De} \pr_t -\pr_r +\frac{a}{\De} \pr_\phi.
\eeq
Given a fixed    $r_0> r_+$ we introduce  the  adapted ingoing Eddington-Finkelstein,   \textit{advanced time},  function $v$ defined by
\[
v =    t+r_*(r), \qquad  r_*(r)=     \int_{r_0}^r\frac{{r'}^2+a^2}{\Delta(r')}dr'.
\]

  The associated, ingoing,  Eddington-Finkelstein coordinates  $(v, r, \th, \phip)$   are given by   
\[
v= t+ r_*, \quad r_*'(r)=\frac{r^2+ a^2}{ \De}, \qquad \phip:= \phi +\phi_*(r), \quad \phi_*'(r)=\frac{a}{\Delta},
\]
such that,
\[
e^{(in)}_3(r)=-1, \qquad  e^{(in)}_3(v)=e^{(in)}_3(\th)=e^{(in)}_3(\phip)=0. 
\]
\begin{lemma}
\lab{Lemma:EFin-coords}
In ingoing coordinates  the inverse metric  takes the form   $|q|^2 \g^{\a\b}=\h^{\a\b}+O^{\a\b}$   where 
the   non-vanishing   components of $\h $  are given by
\beq
\lab{eq:inverseEF-}
\h^{rr}=\De,\qquad  \h^{vr}=\h^{rv}=r^2+a^2, \qquad  \h^{r\phip}=\h^{\phip r}=a
\eeq
and 
\beq
O^{\a\b}=\pr_\th^\a \pr_\th^\b + \frac{1}{\sin^2\th}\big(a\sin^2\th \pr_v +\pr_\phip\big)^{\a}  \big(a\sin^2\th \pr_v +\pr_\phip\big)^{\b},
\eeq
 The Killing vectorfields are  $\T=\pr_v$  and $\Z=\pr_\phip$ and  the  canonical null pair is given by \eqref{eq:null-pair-in}.
 The  scalar wave operator  is  given by the formula
\begin{align*}
|q|^2 \square_{a,m}  \psi&=  \partial_{r}\!\big(\Delta \,\partial_{r} \psi\big)+\partial_v\big((r^2+a^2) \partial_{r} \psi\big) +\partial_r\big((r^2+a^2) \partial_{v} \psi\big)+ 2a\,\partial_{r} \partial_{\phip} \psi+\OO\\
\OO\psi&=  \frac{1}{\sin\theta} \partial_{\theta}(\sin\theta \, \partial_{\theta} \psi) + \frac{1}{\sin^2\theta}\,\partial_{\phip}^2 \psi + 2a\,\partial_v \partial_{\phip} \psi +a^2 \sin^2\theta \, \partial_v^2 \psi.
\end{align*}
Finally
 the covariant form   of the metric is
\begin{align*}
&-\left(1-\frac{2mr}{|q|^2}\right)dv^2
+2\,dv\,dr
+|q|^2\,d\theta^2
+\left(r^2+a^2+\frac{2ma^2 r\sin^2\theta}{|q|^2}\right)\sin^2\theta\,d\phip^2\\
&-2a\sin^2\theta\,dr\, d \phip
-\frac{4mar\sin^2\theta}{|q|^2}\,dv\,d\phip.
\end{align*}
\end{lemma}


\subsubsection{Outgoing PN frame and  adapted EF(-)  coordinates}


The  outgoing PN  frame  (with $ \D_4 e_4=0$),  regular toward the future  for all $r>r_+$,   is given by  the null pair
 $ e^{(out)}_4=\frac{|q|^2}{\De} e_4^{(in)}, \,  e^{(out)}_3=\frac{\De}{|q|^2} e_3^{(in)}$, i.e. 
\beq
\lab{eq:Out.PGdirections-Kerr}
 e^{(out)}_4=\frac{r^2+a^2}{\Delta}\pr_t+\pr_r+\frac{a}{\Delta}\pr_\phi,\quad 
  e^{(out)}_3=\frac{r^2+a^2}{|q|^2}\pr_t-\frac{\Delta}{|q|^2}\pr_r+\frac{a}{|q|^2}\pr_\phi,
\eeq
 The  outgoing  Eddington-Finkelstein, \textit{retarded time},   function     $u$ is  defined by
\[
u := t- r_*, \qquad r_*=\int_{r_0}^r\frac{{r'}^2+a^2}{\Delta(r')}dr'.
\]
   The associated, outgoing,  Eddington-Finkelstein coordinates  $(u, r, \th, \phim)$   are given by   
\[
u:= t- r_*, \quad r_*'=\frac{r^2+ a^2}{ \De}, \qquad \phim:= \phi -\phi_*(r), \quad \phi_*'(r)=\frac{a}{\Delta},
\]
such that,
\[
e^{(out)}_4(r)=1, \qquad  e^{(out)}_4(u)=e^{(out)}_4(\th)=e^{(out)}_4(\phim)=0. 
\]

\begin{lemma} 
\lab{Lemma:EFout-coords}
In outgoing coordinates  the inverse metric  takes the form   $|q|^2 \g^{\a\b}=\h^{\a\b}+O^{\a\b}$   where 
the   non-vanishing   components of $\h $  are given by
\beq
\lab{eq:Kmertic-outgoing}
\h^{rr}=\De, \quad \h^{ur}=\h^{ru}=-(r^2+a^2), \quad  \h^{r\phim}=\h^{\phim r}=-a
\eeq
and 
\beq
O^{\a\b}=\pr_\th^\a \pr_\th^\b + \frac{1}{\sin^2\th}\big(a\sin^2\th \pr_u +\pr_\phim\big)^{\a}  \big(a\sin^2\th \pr_u +\pr_\phim\big)^{\b},
\eeq
 The Killing vectorfields  are  $\T=\pr_u$  and $\Z=\pr_\phim$ and the  outgoing null frame is given by   \eqref{eq:Out.PGdirections-Kerr}.
 The wave operator takes the form
\begin{align*}
|q|^2 \square_{a,m}  \psi&= \partial_{r}\!\big(\Delta \,\partial_{r} \psi\big)-\partial_{u}\big(( r^2+a^2)  \partial_{r} \psi\big) -\partial_{r}\big(( r^2+a^2)  \partial_{u} \psi\big)- 2a\,\partial_{r} \partial_{\phim} \psi+\OO\\
\OO\psi&=  \frac{1}{\sin\theta} \partial_{\theta}(\sin\theta \, \partial_{\theta} \psi) + \frac{1}{\sin^2\theta}\,\partial_{\phim}^2 \psi + 2a\,\partial_{u} \partial_{\phim} \psi +a^2 \sin^2\theta \, \partial_{u}^2 \psi.
\end{align*}
The covariant form  of the metric is 
\begin{align*}
&-\left(1-\frac{2mr}{|q|^2}\right)du^2
-2\,du\,dr
+|q|^2\,d\theta^2
+\left(r^2+a^2+\frac{2ma^2 r\sin^2\theta}{|q|^2}\right)\sin^2\theta\,d\phim^2\\
&+2a\sin^2\theta\,dr\,d\phim
-\frac{4mar\sin^2\theta}{|q|^2}\,du\,d\phim.
\end{align*}
\end{lemma}






\subsection{Modified  time  function} 
 We recall that  in  \cite{He-K1}   we made use of  the time function
 \[
\tau=t+ f(r), \qquad f'(r)= \big(\frac{r^2+a^2}{\Delta}-\frac{m^2}{r^2} \big) \chi
\]
with $\chi\ge 0$ a smooth cut-off function of $r$,   with $\chi=1$    for $r\le  4m $ and $\chi=0$ for $r\ge  8m $.   In particular  the level hypersurfaces   of $\tau$ are spacelike, transversal to the   horizon and    $\tau=t$  for $r$ sufficiently large. In this paper  we rely on a modified   time function defined as follows.
\beq
\lab{eq:defoftaut}
\widehat{\tau}=\begin{cases}
\tau\quad r\leq 4m\\
u+m\ln r\quad r\geq 5m
\end{cases}.
\eeq
with a smooth interpolation in the intermediate region $4m<r<5m$.  In the
exterior region $r\geq 5m$, the azimuthal coordinate $\phi$  is the  corresponding BL  one.
  
\begin{lemma} 
\lab{Lemma:tauhat-coords}
In coordinates  $(\widehat{\tau}, r, \th, \phi)$, the inverse metric  takes the form   $|q|^2 \g^{\a\b}=\h^{\a\b}+O^{\a\b}$   where 
the    components of $\h $ in the region $r\geq 5m$ are given by
\beq
\lab{eq:Kmertic-tauhat}
\bsplit
\h^{rr}&=\De, \quad \h^{\widehat{\tau}r}=\h^{r\widehat{\tau}}=-(r^2-mr+2m^2+a^2-\frac{ma^2}{r}), \quad  \h^{r\phi}=\h^{\phi r}=0, \quad \\
\h^{\widehat{\tau}\widehat{\tau}}&=-m(2r-m+\frac{2m^2+2a^2}{r}-\frac{ma^2}{r^2}),\quad \h^{\widehat{\tau}\phi} =\h^{\phi\widehat{\tau}}=-a\frac{r^2+a^2}{\De},\quad \h^{\phi\phi}=-\frac{a^2}{\De}
\end{split} 
\eeq
and 
\beq
\lab{eq:Kmertic-Ohat}
O^{\a\b}=\pr_\th^\a \pr_\th^\b
+ \frac{1}{\sin^2\th}\big(a\sin\th \pr_{\widehat{\tau}} +\pr_\phi\big)^{\a}
  \big(a\sin\th \pr_{\widehat{\tau}}+\pr_\phi\big)^{\b}.
\eeq
 The wave operator takes the form
\begin{align*}
|q|^2 \square_{a,m}  \psi&= \partial_{r}\!\big(\Delta \,\partial_{r} \psi\big)+\partial_{\widehat{\tau}}\big(\h^{\widehat{\tau}r}\partial_{r} \psi\big) +\partial_{r}\big(\h^{r\widehat{\tau}}  \partial_{\widehat{\tau}} \psi\big)+\h^{\widehat{\tau}\widehat{\tau}}\pr_{\widehat{\tau}}^2\psi+2\h^{\widehat{\tau}\phi}\pr_{\widehat{\tau}}\pr_{\phi}\psi+\h^{\phi\phi}\pr_{\phi}^2\psi+\OO\psi\\
\OO\psi&=  \frac{1}{\sin\theta} \partial_{\theta}(\sin\theta \, \partial_{\theta} \psi) + \frac{1}{\sin^2\theta}\,\partial_{\phi}^2 \psi + 2a\,\partial_{\widehat{\tau}} \partial_{\phi} \psi +a^2 \sin^2\theta \, \partial_{\widehat{\tau}}^2 \psi.
\end{align*}
The covariant form  of the metric is 
\begin{align*}
&-\left(1-\frac{2mr}{|q|^2}\right)d\widehat{\tau}^2-\frac{4mar\sin^2\theta}{|q|^2}\,d\widehat{\tau}\,d\phi+|q|^2\,d\theta^2
+\left(r^2+a^2+\frac{2ma^2 r\sin^2\theta}{|q|^2}\right)\sin^2\theta\,d\phi^2\\
&+\g_{rr}dr^2+2\g_{\widehat{\tau}r}\,d\widehat{\tau}\,dr+2\g_{r\phi}\,dr\,d\phi.
\end{align*}
where for $r\geq 5m$
\[
\g_{rr}\sim mr^{-1},\quad \g_{\widehat{\tau}r}\sim -1,\quad \g_{r\phi}\sim -ar^{-1}\sin^2\th.
\]
Finally,  the hypersurfaces
$\{\widehat{\tau}=const\}$ are spacelike.
\end{lemma}

\begin{proof}
In the region $r\geq 5m$ we have
\[
\widehat{\tau}=t+F(r),\qquad
F'(r)=\frac mr-\frac{r^2+a^2}{\Delta}
\]
while $r,\theta,\phi$ are unchanged. Put $A=r^2+a^2$.  In BL coordinates the rescaled inverse
metric has the split form
\[
|q|^2\g^{\alpha\beta}=\h^{\alpha\beta}+O^{\alpha\beta},
\]
where 
\[
\h^{rr}=\Delta,\qquad
\h^{tt}=-\frac{A^2}{\Delta},\qquad
\h^{t\phi}=-a\frac{A}{\Delta},\qquad
\h^{\phi\phi}=-\frac{a^2}{\Delta},
\]
and $\h^{tr}=\h^{r\phi}=0$.  Since $\widehat{\tau}=t+F$, we obtain
\[
\h^{\widehat{\tau}r}
=F'\h^{rr}
=(\frac mr-\frac{A}{\Delta})\Delta
=-(r^2-mr+2m^2+a^2-\frac{ma^2}{r}),
\]
and $\h^{r\phi}=0$.  Similarly,
\begin{align*}
\h^{\widehat{\tau}\widehat{\tau}}
&=\h^{tt}+(F')^2\h^{rr}  =-\frac{A^2}{\Delta}
+(\frac mr-\frac{A}{\Delta})^2\Delta  =-\frac{2mA}{r}+\frac{m^2\Delta}{r^2}  \\
&=-m(2r-m+\frac{2m^2+2a^2}{r}-\frac{ma^2}{r^2}).
\end{align*}
The remaining components of $\h$ and $O$ are unchanged. This proves \eqref{eq:Kmertic-tauhat} and \eqref{eq:Kmertic-Ohat}.  The covariant metric compomemts follow by applying
$d\widehat{\tau}=dt+F'(r)dr$ to the covariant metric in BL coordinates and
keeping only the leading powers of $r$.

Since
\[
\g(\D\widehat{\tau}, \D\widehat{\tau})
=\g^{\widehat{\tau}\widehat{\tau}}
=-\frac{2m}{r}+O(r^{-2})<0
\]
for $r$ sufficiently large, the hypersurfaces
$\{\widehat{\tau}=const\}$ are spacelike.
\end{proof}

\begin{remark}
All the results  derived in \cite{He-K1}    remain true  with appropriate modifications of the   $E$ and $ E_p$ norms (see \eqref{eq:defofnorms}) if we replace  the time function $\tau$ by $\tauh$.
\end{remark}

We make use of the following commutator  lemma.
\begin{lemma}
\lab{Le:comm-tauh}
For a given smooth cut-off function $\chi=\chi(\tauh)$  we have for sufficiently large $r$
\beq
\lab{eq:Comm-tauh1}
\big[|q|^2\square_{\g}, \chi(\tauh)\big]\psi= -2\chi'r^2\c\frac{1}{r}e_4(r\psi)-2mr\chi'T\psi-2mr\chi''\psi+
O(1)(|\chi'|+|\chi''|)\dk^{\leq 1}\psi
\eeq
and
\beq
\lab{eq:Comm-tauh2}
e_4\left(\big[|q|^2\square_{\g}, \chi(\tauh)\big]\psi\right)=
O(1)(|\chi'|+|\chi''|+|\chi'''|)(\dk^{\leq 1}\psi+re_4\dk^{\leq 1}\psi).
\eeq
\end{lemma} 
\begin{proof}
For $r\geq 5m$, using
\[
e^{(out)}_4=\pr_r+\frac mr T+\frac a\Delta Z,\qquad
\pr_r=e^{(out)}_4-\frac mr T-\frac a\Delta Z,
\]
where $\pr_r$ is taken at fixed $(\tauh,\theta,\phi)$, we compute
\beq
\lab{eq:Comm-tauh1detail}
\bsplit
\big[|q|^2\square_{\g}, \chi(\tauh)\big]\psi
&=2|q|^2\g^{\a\b}\pr_\a\chi\pr_\b\psi+(|q|^2\square_{\g}\chi)\psi\\
&=2(\h^{\widehat{\tau}\widehat{\tau}}+O^{\widehat{\tau}\widehat{\tau}})\chi' T\psi
  +2\h^{\widehat{\tau}r}\chi' \pr_r\psi
  +2(\h^{\widehat{\tau}\phi}+O^{\widehat{\tau}\phi})\chi' Z\psi\\
&\quad +(\h^{\widehat{\tau}\widehat{\tau}}+O^{\widehat{\tau}\widehat{\tau}})\chi''\psi
  +\pr_r\h^{\widehat{\tau}r}\chi'\psi\\
&=(2\h^{\widehat{\tau}r}\chi' \pr_r\psi+\pr_r\h^{\widehat{\tau}r}\chi'\psi)
  +\big(-4mr+O(1)\big)\chi'T\psi\\
&\quad +\big(-2mr+O(1)\big)\chi''\psi+O(1)\chi'Z\psi\\
&= -2\chi'r^2\c\frac{1}{r}e^{(out)}_4(r\psi)-2mr\chi'T\psi-2mr\chi''\psi
   +O(1)(|\chi'|+|\chi''|)\dk^{\leq 1}\psi.
   \end{split}
\eeq
This proves \eqref{eq:Comm-tauh1}.  Applying $e_4$ to \eqref{eq:Comm-tauh1} and using
\[
e^{(out)}_4(r)=1,\qquad e^{(out)}_4(\tauh)=\frac{m}{r},\qquad e^{(out)}_4(\chi^{(j)}(\tauh))=\frac{m}{r}
\chi^{(j+1)}(\tauh),\qquad e^{(out)}_4(a/\Delta)=O(r^{-3})
\]
give \eqref{eq:Comm-tauh2}.
\end{proof}

\begin{remark}
The choice of the modified time function is used only in the asymptotic
region.  If, instead, we used the original time function, which agrees with
$t$ for $r$ large, the same calculation would give
\[
\big[|q|^2\square_{\g},\chi(t)\big]\psi
=\big(-2r^2+O(r)\big)\chi'T\psi
 +\big(-r^2+O(r)\big)\chi''\psi
 +O(1)\chi'Z\psi,
\]
which contains the unfavorable terms
\(r^2\chi'T\psi\) and \(r^2\chi''\psi\). More importantly, applying $e_4$ to the commutator does not improve the decay in $r$ (because $e_4(\chi(t))=O(1)$). 

At the opposite extreme, if one uses $u$ itself, then 
\begin{align*}
\big[|q|^2\square_{\g},\chi(u)\big]\psi
&=2\h^{ur}\chi'e_4\psi+\partial_r\h^{ur}\chi'\psi+O(1)\chi'\dk^{\leq1}\psi\\
&=-2\chi' r^2\frac1r e_4(r\psi)
  +O(1)\chi'(\dk^{\leq1}\psi+e_4\psi),
\end{align*}
which has the desired (radiation-field) structure. However, with compactly supported initial data, cutting off $\psi$ in a bounded $u$  region does not imply that $\psi$ has sufficient decay in $r$, which is necessary for justifying the definition of generalized Whiting transform.
\end{remark}

\subsection{Time cutoff}
\lab{subsubsection:energy-time-cutoff-bookkeeping}

Let $\chi=\chi(\widehat\tau)$ be a smooth time cutoff satisfying
\beq
\lab{eq:cut-off-chi}
\chi=0\quad\hbox{for }\widehat\tau\notin(0,\widehat\tau^*),
\qquad
\chi=1\quad\hbox{for }1<\widehat\tau<\widehat\tau^*-1.
\eeq
Thus $\chi'$ and $\chi''$ are supported in the initial and final
transition slabs
\[
0<\widehat\tau<1,
\qquad
\widehat\tau^*-1<\widehat\tau<\widehat\tau^*.
\]
Recall
\[
\QQ_{\a\b } [\psi] =\D_\a \psi \D_\b \psi -\frac12 \g_{\a\b }\D^\mu \psi \D_\mu \psi.
\]
We also introduce 
\[
\QQ_{\a\b }[\phi, \psi]=\frac 1 2 \big( \D_\a \phi \D_\b\psi+ \D_\b \phi \D_\a \psi- \g_{\a\b}\D^\mu \phi \D_\mu \psi\big)
\]
\begin{lemma}
We have the following identity
\[
\QQ_{\a\b }[\phi\c \psi]= 
\phi^2 \QQ_{\a\b }[\psi]
+
\psi^2 \QQ_{\a\b }[\phi]
+
2\phi\psi\, \QQ_{\a\b}[\phi,\psi].
\]
\end{lemma}
\begin{proof}
We have,
\[
\D_\a(\phi\psi)=\phi\D_\a\psi+\psi\D_\a\phi.
\]
Thus
\begin{align*}
\D_\a(\phi\psi)\D_\b(\phi\psi)
&=
\phi^2\D_\a\psi\D_\b\psi
+
\psi^2\D_\a\phi\D_\b\phi  \\
&
+
\phi\psi
\left(
\D_\a\phi\D_\b\psi
+
\D_\a\psi\D_\b\phi
\right),
\end{align*}
and similarly
\begin{align*}
\D^\mu(\phi\psi)\D_\mu(\phi\psi)
&=
\phi^2\D^\mu\psi\D_\mu\psi
+
\psi^2\D^\mu\phi\D_\mu\phi  \\
&
+
2\phi\psi\,\D^\mu\phi\D_\mu\psi .
\end{align*}
Substituting these two identities into the definition of
$\QQ_{\a\b}[\phi\psi]$, we obtain
\begin{align*}
\QQ_{\a\b}[\phi\psi]
&=
\phi^2
(
\D_\a\psi\D_\b\psi
-\frac12\g_{\a\b}\D^\mu\psi\D_\mu\psi
)\\
&+
\psi^2
(
\D_\a\phi\D_\b\phi
-\frac12\g_{\a\b}\D^\mu\phi\D_\mu\phi
)\\
&+
\phi\psi
(
\D_\a\phi\D_\b\psi
+
\D_\a\psi\D_\b\phi
-
\g_{\a\b}\D^\mu\phi\D_\mu\psi
).
\end{align*}
By the definition of the polarized tensor,
\[
\QQ_{\a\b}[\phi,\psi]
=
\frac12
\left(
\D_\a\phi\D_\b\psi
+
\D_\a\psi\D_\b\phi
-
\g_{\a\b}\D^\mu\phi\D_\mu\psi
\right),
\]
the last line is precisely
\[
2\phi\psi\,\QQ_{\a\b}[\phi,\psi].
\]
Therefore
\[
\QQ_{\a\b }[\phi\psi]
=
\phi^2 \QQ_{\a\b }[\psi]
+
\psi^2 \QQ_{\a\b }[\phi]
+
2\phi\psi\, \QQ_{\a\b}[\phi,\psi],
\]
as desired.
\end{proof}
We apply the Lemma  to solutions $\psi $ of $\square \psi=N$ and  with $\phi$ a smooth cut-off function $\chi$.
We deduce
\begin{align*}
 \D^\b \QQ_{\a\b }[\chi \psi]&= \chi^2\D^\b  \QQ_{\a\b }[\psi] +\D^\b (\chi^2)  \QQ_{\a\b }[\psi]  + \D^\b \Big(\psi^2 \QQ_{\a\b }[\chi]+2\chi \psi\, \QQ_{\a\b}[\chi, \psi]\Big)\\
 &= \chi^2  \D_\a \psi  N+ \D^\b (\chi^2)  \QQ_{\a\b }[\psi]  + \D^\b \Big(\psi^2 \QQ_{\a\b }[\chi]+2\chi\psi\, \QQ_{\a\b}[\chi,\psi]\Big)
\end{align*}
Multiplying by $\T$ we further deduce, if $\T$ is Killing,
\[
 \D^\b \big( \T^\a \QQ_{\a\b }[\chi \psi]\big)= \QQ_{\a\b }[\psi]\T^\a  D^\b (\chi^2) +  \D^\b \Big(\psi^2\T^\a  \QQ_{\a\b }[\chi]+2\chi\psi\, \T^\a \QQ_{\a\b}[\chi,\psi]\Big) +\chi^2  \T \psi  N
\]
\begin{lemma}
In the particular  case when $\chi$ is a function of $\tauh$,
\beq
\lab{eq:cutoffdivergence}
\D^\b \big( \T^\a \QQ_{\a\b }[\chi \psi]\big)=   2 \chi \chi'  \QQ[\psi] (T, \D\tau)+  \D^\b \Big(\psi^2\T^\a  \QQ_{\a\b }[\chi]+2\chi\psi\, \T^\a \QQ_{\a\b}[\chi,\psi]\Big) +\chi^2  \T \psi\c N.
\eeq
Note that  the divergence  term on the  right    vanishes  in regions where  $\chi'=0$.
\end{lemma}

\section{Generalized $\Wk$--transform} 
\lab{sec:GenWhitingTr}


\begin{definition}
\lab{Def:WhitingTr}
We consider the Kerr metric given in either incoming or outgoing coordinates $w, r, \th, \vphi$, where 
$w=u$, $\vphi=\phim$  in outgoing coordinates and  $w=v$, $\vphi=\phip$ in ingoing coordinates.
We define the generalized Whiting transform  $\psi\to \psit= \Wk[\psi] $  as follows:
  \beq\lab{eq:notationoftransgeneral}
\psit(\wt, \rt,  \tht, \vphit ):=\int_{r_+}^\infty \psi\big(\wt+B(\rt,r), r, \tht,  \vphit\big)\,dr.
\eeq	
where
\beq
 B(\rt,r):=S(\rt-\Rt)(r-R)
 \eeq
 with $\Rt, R\in\RRR$ and $S<0$\footnote{The sign of $S$ is of little consequence.}, and suppressing the variables\footnote{We will  need to take them into consideration  when  deriving integral estimates.  } $(\tht, \vphit) $ we rewrite in the simplified form 
 \beq
 \lab{eq:notationoftrans'}
 \psit(\wt, \rt)=\int_{r_+}^\infty \psi\big(\wt+B(\rt,r), r\big)\,dr.
 \eeq
 \end{definition}
 Throughout this section, for simplicity, we use the same letter for the
 two one-variable factors in the phase,
 \[
 B(\rt):=S(\rt-\Rt),\quad B(r):=S(r-R).
  \]
 Thus 
 \[
 B(\rt,r)=B(\rt)(r-R)=B(r)(\rt-\Rt).
 \]
 We shall also use the notation
 \[
 X:=B(\rt)=S(\rt-\Rt).
 \]
 
 \begin{remark}
 We note that our  generalized Whiting transform defined in \eqref{eq:notationoftransgeneral}  is closely connected   to the  John's transform \cite{John}  except  that, as in the case of the Radon transform,  we restrict to two dimensional rays in $u, r$.
   \end{remark}
   
   \begin{definition}
    \lab{def:acceptablesol}
   We say that $\psi$ is an acceptable solution of $\square_{\g}\psi=N$ if, relative to $\tauh$, 
   \beq
    \lab{eq:acceptablesol}
   \psi\in C_c^\infty
\bigl(\RRR_{\tauh}\times[r_+,\infty)_r\times\SSS^2\bigr).  
\eeq
In particular, if $\psi$ solves $\square_{\g}\psi=N$ with compactly supported initial data at $\tauh=c$ and $\psi\in C_c^\infty(\RRR_{\tauh})$, then $\psi$ is acceptable.
   \end{definition}
   As immediate results of the definition, we have the following conclusions.
    \begin{lemma}
   \lab{lem:acceptablesol}
   The followings hold true for an acceptable solution\footnote{Acceptable solutions of $\square \psi=N'$, with a modified inhomogeneous term $N'$,    are generated by  
   cutting  off solutions  $\psi$ of $\square \psi=N$ in $\tauh$.}   $\psi$.
   \begin{enumerate}
   \item The generalized  Whiting transform $\psit=\Wk[\psi]$ is well defined for any $\wt, X\in\RRR$.
   \item The Fourier transform of $\psi$ in $\tauh$ and of $\psit$ in $\wt$ are well defined.
   \item Let $w=u$ in the definition of generalized  Whiting transform. Let 
 \[
    \EE[\psit]:=X^2|T\psit|^2+X^2|\pr_{\rt}\psit|^2+|\pr_{\tht}\psit|^2+\frac{1}{\sin^2\tht}|Z\psit|^2+|\psit|^2.
\]
Then there exist sequences $s_i\to   \pm \infty$ such that 
\beq\lab{eq:positive-final-slice-vanishing}
\lim_{i\to \infty}
\int_{\widetilde\Sigma(s_i)}
 \EE[TT^{1/2}\psit]\, d\rt d\omega=0
\eeq
where $\widetilde\Sigma(s)=\{\taut=s\}\cap\{-\infty<X<\infty\}$.
      \end{enumerate}
   \end{lemma}
   \begin{proof}
   The proof of the first two statements is straightforward.  The proof of   the last statement requires however   notations and 
   results which are proved  later in this section.   As there is no danger of circularity,   we  present the proof here. To prove the last statement 
   it suffices  to show that
  \[
  \int_{\RRR}\int_{\widetilde\Sigma(s)}
 \EE[TT^{1/2}\psit]\, d\rt d\omega ds<\infty.
   \]
   Using Lemma \ref{lem:Rk-commutation} and Corollary  \ref{cor:filtered-L2-outgoing-transform}, we derive
   \begin{align*}
   & \int_{\RRR}\int_{\widetilde\Sigma(s)}
 \EE[TT^{1/2}\psit]\, d\rt d\omega ds\\
   &\quad= \int\Big(X^2|TTT^{1/2}\psit|^2+X^2|\T \pr_{\rt}T^{1/2}\psit|^2+|\pr_{\tht}TT^{1/2}\psit|^2+\frac{1}{\sin^2\tht}|ZTT^{1/2}\psit|^2+|TT^{1/2}\psit|^2\Big)\, d\taut d\rt d\omega \\
   &\quad\les \int_{\DD}\Big(|e_4T\psi|^2+| e_4(B(r)T\psi)|^2\Big)\, |q|^{-2}dv_{\g}+ \int_{\DD}\Big(|\pr_{\th}T\psi|^2+\frac{1}{\sin^2\th}|ZT\psi|^2+|T\psi|^2\Big)\, |q|^{-2}dv_{\g} \\
   &\quad<\infty
      \end{align*} 
      where we use the   $ L^2 $ estimates of the Corollary \ref{cor:filtered-L2-outgoing-transform},   and in particular,  the estimate    \eqref{eq:e_4trick}  to deal with the first two terms.   In the last step we  have also used the        fact $ \psi\in C_c^\infty
\bigl(\RRR_{\tauh}\times[r_+,\infty)_r\times\SSS^2\bigr)$  
     \end{proof}

  
\subsection{Transformed metrics}
 \subsubsection{Commutation properties of the $\Wk$ transform}
 \lab{subsubsection:Commutations-general}
 

 \begin{lemma}\lab{lem:Rk-commutation}
	Let
\[
\psi_+(\wt,\rt):=
\psi\big(\wt+B(\rt,r_+),r_+\big).
\]
Then the operator $\Wk$ verifies the following commutation properties.
\begin{enumerate}
\item
\[
\pr_\wt \Wk[\psi]=\Wk[\pr_w \psi],\qquad
\pr_\vphit \Wk[\psi]=\Wk[\pr_\vphi \psi], \qquad
\pr_{\tht}\Wk[\psi]=\Wk[\pr_\th \psi].
\]
\item For the radial derivative and an arbitrary function $f(r)$,
	 \beq
	 \lab{eq:Rk-commutation}
	 \bsplit
	 \pr_{\rt }\Wk[\psi](\wt,\rt)
	 &=\pr_{\wt}\Wk[B(r)\psi](\wt,\rt)
	 =\Wk[B(r)\pr_w\psi](\wt,\rt).\\
	  \Wk[f\pr_r\psi](\wt,\rt)
	 &= -B(\rt)\pr_\wt\Wk[f\psi](\wt,\rt)-\Wk[\pr_rf\psi](\wt,\rt)
	 -f(r_+)\psi_+(\wt,\rt).
 \end{split}
 \eeq
\end{enumerate}
In particular, if the horizon trace
\(\psi(\wt+B(\rt,r_+),r_+)\) vanishes, then
\[
\Wk[\pr_r\psi]= -B(\rt)\pr_\wt\Wk[\psi].
\]
\end{lemma}
 \begin{proof}
	 The commutation relations with $\pr_w,\pr_\vphi,\pr_\th$ are immediate, since
	 the kernel only shifts the $w$ variable.  
	 Since $\pr_r B(\rt,r)=B(\rt)$,
	 \[
	 \frac{d}{dr}(f\psi)\big(\wt+B(\rt,r),r\big)
	 =f\pr_r\psi+(\pr_rf)\psi+B(\rt)\pr_w(f\psi).
	 \]
	 Integrating from $r_+$ to $\infty$ gives
	 \[
	 \Wk[f\pr_r\psi]
	 =
	 -B(\rt)\pr_\wt\Wk[f\psi]
		 -\Wk[\pr_rf\psi](\wt,\rt)
	 -f(r_+)\psi\big(\wt+B(\rt,r_+),r_+\big),
	 \]
	 where the upper boundary term at $r=\infty$ vanishes because $\psi$ has compact support in $r$.  This proves the first identity in \eqref{eq:Rk-commutation}, including
 the lower boundary term at the horizon.  Finally, since
 $\pr_\rt B(\rt,r)=B(r)$,
 \[
 \pr_\rt\Wk[\psi]=\Wk[B(r)\pr_w\psi]
 =\pr_\wt\Wk[B(r)\psi].
 \]
 \end{proof}

\subsubsection{Specific  form of the transformed metric}
\begin{proposition}
\lab{Lemma:Whiting-metric-general}
There exists a metric $\gt$, of the same split form as the Kerr metric $\g$,
whose rescaled inverse metric is
 \beq
|\qt|^2  \gt^{\at\bt}=\hbt^{\at\bt}+O^{\at\bt},
 \eeq
 such that
 \beq
 \lab{eq:Whiting-metric-general}
\Wk[|q|^2 \square_\g\psi]
=
|\qt|^2  \square_\gt \Wk[ \psi] +\mathcal E_+[\psi]
 \eeq
where
\beq
\lab{eq:defofqtsquaregt}
|\qt|^2\square_{\gt}\psit:=\pr_{\at}\big(\hbt^{\at\bt}\pr_{\bt}\psit\big)+\frac{1}{\sin\tht}\pr_{\at}\big(\sin\tht O^{\at\bt}\pr_{\bt}\psit\big)
\eeq
and the horizon boundary term is (using the fact that $\h^{rr}(r_+)=\Delta(r_+)=0$)
\beq
\lab{eq:Whiting-boundary-error}
\mathcal E_+[\psi]
=
-2\h^{wr}(r_+)(\pr_w\psi)_+
-2\h^{r\vphi}(r_+)(\pr_\vphi\psi)_+.
\eeq
Moreover the
non-vanishing components of $\hbt$ are given by (using $\varep=1$ or $\varep=-1$ for  the  ingoing or outgoing coordinates choice in the definition of generalized Whiting transform)
 \beq
 \lab{eq:hbt-gencomponents'}
 \bsplit
\hbt^{\rt\rt}
&=S^{-2}X(X-2\varepsilon),\\
\hbt^{\wt\rt}
&=S^{-1}\left(\frac12\Delta'_R X^2-2\varepsilon RX\right)
=\hbt^{\rt\wt},\\
\hbt^{\wt\wt}
&=\Delta_R X^2-2\varepsilon A_R X,\\
\hbt^{\wt\vphit}
&=-\varepsilon aX=\hbt^{\vphit\wt}.
\end{split}
\eeq
where 
\beq
\lab{eq:radial-multiplier-XA_RDe_RDeprime_R}	
X:=B(\rt)=S(\rt-\Rt),\quad
A_R:=R^2+a^2,\quad
\Delta_R:=\Delta(R),\quad
\Delta'_R:=2(R-m).
\eeq
 \end{proposition}

\begin{proof}
Since the angular operator $\OO$ only contains derivatives in \(w,\theta,\vphi\),
it commutes with $\Wk$.  It is therefore enough to prove
\[
 \Wk \big[\pr_\a (\h^{\a\b} \pr_\b\psi)\big]
 =
 \pr_\at(\hbt^{\at\bt}\pr_\bt\Wk[\psi])
 +\mathcal E_+[\psi].
\]
We write the radial part of the wave operator in the unified form
\beq
\lab{eq:wave-hpart-general}
\bsplit
\pr_\a (\h^{\a\b} \pr_\b\psi)
&= \pr_r\big(\h^{rr}\pr_r\psi\big)
+\pr_w\big(\h^{wr}\pr_r\psi\big)
+\pr_r\big(\h^{wr}\pr_w\psi\big)\\
&\quad
+\pr_r\big(\h^{r\vphi}\pr_\vphi\psi\big)
+\pr_\vphi\big(\h^{\vphi r}\pr_r\psi\big).
\end{split}
\eeq
Here $\h^{rr}$ and $\h^{wr}$ are quadratic polynomials in $r$, while
$\h^{r\vphi}$ is constant in $r$.  Thus
\begin{align*}
\h^{rr}(r)&=\frac 12\pr_r^2\h^{rr}(R)S^{-2}B(r)^2
+\pr_r\h^{rr}(R)S^{-1}B(r)+\h^{rr}(R),\\
\pr_r\h^{rr}(r)&=\pr_r^2\h^{rr}(R)S^{-1}B(r)+\pr_r\h^{rr}(R),\\
\h^{wr}(r)&=\frac 12\pr_r^2\h^{wr}(R)S^{-2}B(r)^2
+\pr_r\h^{wr}(R)S^{-1}B(r)+\h^{wr}(R),\\
\pr_r\h^{wr}(r)&=\pr_r^2\h^{wr}(R)S^{-1}B(r)+\pr_r\h^{wr}(R).
\end{align*}
Using Lemma \ref{lem:Rk-commutation},
\begin{align*}
 \Wk[f\pr_r\psi](\wt,\rt)
	 &= -B(\rt)\pr_\wt\Wk[f\psi](\wt,\rt)-\Wk[\pr_rf\psi](\wt,\rt)
	 -f(r_+)\psi_+(\wt,\rt),\\
	 \Wk[B(r)\pr_w\psi]&=\pr_\rt\Wk[\psi].
\end{align*}
and $\h^{rr}(r_+)=0$, we compute
\begin{align*}
\Wk[\pr_r(\h^{rr}\pr_r\psi)]
&=
B(\rt)^2\pr_\wt^2\Wk[\h^{rr}\psi]
 +B(\rt)\pr_\wt\Wk[\pr_r\h^{rr}\psi]\\
&=
B(\rt)^2
\Big(
\frac12\pr_r^2\h^{rr}(R)S^{-2}\pr_\rt^2
+\pr_r\h^{rr}(R)S^{-1}\pr_\rt\pr_\wt
+\h^{rr}(R)\pr_\wt^2
\Big)\Wk[\psi]\\
&\quad+
B(\rt)
\Big(
\pr_r^2\h^{rr}(R)S^{-1}\pr_\rt
+\pr_r\h^{rr}(R)\pr_\wt
\Big)\Wk[\psi],
\end{align*}
and
\begin{align*}
\Wk[\pr_w(\h^{wr}\pr_r\psi)]
&=
-B(\rt)\pr_\wt^2\Wk[\h^{wr}\psi]
-\pr_\wt\Wk[\pr_r\h^{wr}\psi]
-\h^{wr}(r_+)(\pr_w\psi)_+\\
&=
-B(\rt)
\Big(
\frac12\pr_r^2\h^{wr}(R)S^{-2}\pr_\rt^2
+\pr_r\h^{wr}(R)S^{-1}\pr_\rt\pr_\wt
+\h^{wr}(R)\pr_\wt^2
\Big)\Wk[\psi]\\
&\quad-
\Big(
\pr_r^2\h^{wr}(R)S^{-1}\pr_\rt
+\pr_r\h^{wr}(R)\pr_\wt
\Big)\Wk[\psi]
-\h^{wr}(r_+)(\pr_w\psi)_+,
\end{align*}
while
\begin{align*}
\Wk[\pr_r(\h^{wr}\pr_w\psi)]
&=
-B(\rt)\pr_\wt^2\Wk[\h^{wr}\psi]-\h^{wr}(r_+)(\pr_w\psi)_+\\
&=
-B(\rt)
\Big(
\frac12\pr_r^2\h^{wr}(R)S^{-2}\pr_\rt^2
+\pr_r\h^{wr}(R)S^{-1}\pr_\rt\pr_\wt
+\h^{wr}(R)\pr_\wt^2
\Big)\Wk[\psi]\\
&\quad-\h^{wr}(r_+)(\pr_w\psi)_+.
\end{align*}
Denoting the sum of these three terms by $\II[\psi]$, we obtain
\begin{align*}
\II[\psi]&=
\frac12 B(\rt)
\Big(B(\rt)\pr_r^2\h^{rr}(R)-2\pr_r^2\h^{wr}(R)\Big)
S^{-2}\pr_\rt^2\Wk[\psi]\\
&\quad+
B(\rt)
\Big(B(\rt)\pr_r\h^{rr}(R)-2\pr_r\h^{wr}(R)\Big)
S^{-1}\pr_\rt\pr_\wt\Wk[\psi]\\
&\quad+
B(\rt)
\Big(B(\rt)\h^{rr}(R)-2\h^{wr}(R)\Big)
\pr_\wt^2\Wk[\psi]\\
&\quad+
\Big(B(\rt)\pr_r^2\h^{rr}(R)-\pr_r^2\h^{wr}(R)\Big)
S^{-1}\pr_\rt\Wk[\psi]\\
&\quad+
\Big(B(\rt)\pr_r\h^{rr}(R)-\pr_r\h^{wr}(R)\Big)
\pr_\wt\Wk[\psi]\\
&\quad-2\h^{wr}(r_+)(\pr_w\psi)_+.
\end{align*}
With $\hbt^{\rt\rt}, \hbt^{\rt\wt}$, and $\hbt^{\wt\wt}$ as below
 \beq
 \lab{eq:hbt-gencomponents}
 \bsplit
 \hbt^{\rt\rt }&= \frac 1 2 S^{-2} \Big(  B^2(\rt) \pr_r^2 \h^{rr}(R) - 2  B(\rt)   \pr_r^2 \h^{wr}(R)  \Big), \\
   \hbt^{\wt\rt}&=\frac 1 2 S^{-1} \Big(  B(\rt)^2 \pr_r \h^{rr}(R) - 2 B(\rt) \pr_r \h^{wr}(R)  \Big)= \hbt^{\rt\wt},\\
   \hbt^{\wt\wt}&= B(\rt)^2 \h^{rr}(R)-2 B(\rt)\h^{wr}(R),
 \end{split}
 \eeq
 and
using $\pr_\rt B(\rt)=S$, we have
\begin{align*}
\pr_\rt\hbt^{\rt\rt}
&=
S^{-1}\Big(B(\rt)\pr_r^2\h^{rr}(R)-\pr_r^2\h^{wr}(R)\Big),\\
\pr_\rt\hbt^{\rt\wt}
&=
B(\rt)\pr_r\h^{rr}(R)-\pr_r\h^{wr}(R).
\end{align*}
Therefore
\[
\II[\psi]
=
\pr_\rt(\hbt^{\rt\rt}\pr_\rt\Wk[\psi])
+\pr_\rt(\hbt^{\rt\wt}\pr_\wt\Wk[\psi])
+\pr_\wt(\hbt^{\wt\rt}\pr_\rt\Wk[\psi])
+\pr_\wt(\hbt^{\wt\wt}\pr_\wt\Wk[\psi])
+\mathcal E^{(rw)}_+[\psi],
\]
where
\[
\mathcal E^{(rw)}_+[\psi]
=-2\h^{wr}(r_+)(\pr_w\psi)_+.
\]
It remains to transform the $r\vphi$ components in
\eqref{eq:wave-hpart-general}.  Since $\h^{r\vphi}=\h^{\vphi r}$ is constant
in $r$, Lemma \ref{lem:Rk-commutation} gives
\begin{align*}
\Wk[\pr_r(\h^{r\vphi}\pr_\vphi\psi)]
+\Wk[\pr_\vphi(\h^{\vphi r}\pr_r\psi)]
&=
-2B(\rt)\h^{r\vphi}(R)\pr_\wt\pr_{\vphit}\Wk[\psi]
-2\h^{r\vphi}(r_+)(\pr_\vphi\psi)_+.
\end{align*}
Thus setting
\beq
 \lab{eq:hbt-gencomponentsextra}
 \hbt^{\wt\vphit}=\hbt^{\vphit\wt}
=-B(\rt)\h^{r\vphi}(R)
\eeq
gives
\[
 \Wk \big[\pr_\a (\h^{\a\b} \pr_\b\psi)\big]
 =
 \pr_\at(\hbt^{\at\bt}\pr_\bt\Wk[\psi])
 +\mathcal E_+[\psi],
\]
with $\mathcal E_+[\psi]$ given by
\eqref{eq:Whiting-boundary-error}.  
Adding back the unchanged angular operator $\OO$ gives
\[
	\Wk[|q|^2\square_\g\psi]
	=|\qt|^2\square_\gt\Wk[\psi]+\mathcal E_+[\psi],
	\]
	as desired.

	It remains to prove	\eqref{eq:hbt-gencomponents'}.  With the convention
	\[
	\varepsilon=-1\quad\hbox{in outgoing coordinates }(w=u),
	\qquad
	\varepsilon=+1\quad\hbox{in ingoing coordinates }(w=v),
	\]
	the rescaled inverse Kerr metric has the split form
	\[
	|q|^2\g^{\a\b}=\h^{\a\b}+O^{\a\b},
	\]
	with non-vanishing $\h$-components
	\[
	\h^{rr}=\Delta,\qquad
	\h^{wr}=\h^{rw}=\varepsilon(r^2+a^2),\qquad
	\h^{r\vphi}=\h^{\vphi r}=\varepsilon a.
	\]
	Therefore, at $r=R$
	\[
	\pr_r^2\h^{rr}(R)=2,\qquad
	\pr_r\h^{rr}(R)=\Delta'_R=2(R-m),\qquad
	\h^{rr}(R)=\Delta_R,
	\]
	and
	\[
	\pr_r^2\h^{wr}(R)=2\varepsilon,\qquad
	\pr_r\h^{wr}(R)=2\varepsilon R,\qquad
	\h^{wr}(R)=\varepsilon A_R,\qquad
	\h^{r\vphi}(R)=\varepsilon a .
	\]
	Substituting these expressions in \eqref{eq:hbt-gencomponents} and \eqref{eq:hbt-gencomponentsextra}, and writing
	$X=B(\rt)$, yield
	\begin{align*}
	\hbt^{\rt\rt}
	&=\frac12S^{-2}(2X^2-4\varepsilon X)
	=S^{-2}X(X-2\varepsilon),\\	
	\hbt^{\wt\rt}
	&=\frac12S^{-1}\big(2(R-m)X^2-4\varepsilon RX\big)
	=S^{-1}\left(\frac12\Delta'_R X^2-2\varepsilon RX\right),\\
	\hbt^{\wt\wt}
	&=\Delta_R X^2-2\varepsilon A_RX,\qquad
	\hbt^{\wt\vphit}=-\varepsilon aX.
	\end{align*}
	This proves \eqref{eq:hbt-gencomponents'}.
	\end{proof}

\begin{remark}
 In outgoing EF coordinates $w=u$,
we have $\h^{ur}(r_+)=-(r_+^2+a^2), \h^{r\phim}(r_+)=-a$, and thus
\[
\mathcal E_+[\psi]
=2\Big((r_+^2+a^2)(\pr_u\psi)_+
+a(\pr_\phim\psi)_+\Big).
\]
In ingoing EF coordinates $w=v$, we have $\h^{vr}(r_+)=r_+^2+a^2, \h^{r\phip}(r_+)=a$, and thus
\[
\mathcal E_+[\psi]
=-2\Big((r_+^2+a^2)(\pr_v\psi)_+
+a(\pr_\phip\psi)_+\Big).
\]
\end{remark}
\begin{remark}  
\lab{rem:noincomingrad}
   In outgoing coordinates  the boundary $r=r_+$  corresponds to the past
   event horizon. For solutions of $\square_\g\psi=N$ supported in a bounded region of $\tauh$\footnote{More generally, for solutions supported to the
   future of a spacelike hypersurface.}, the boundary term \(\mathcal E_+\) in
   \eqref{eq:Whiting-metric-general} vanishes. In the ingoing case, however,
   $r=r_+$ is the future event horizon and the corresponding boundary term
   cannot be neglected.
\end{remark}

\subsubsection{Geometric character of the  transformed metric}
We continue to  treat simultaneously the outgoing and ingoing choices  for the Kerr metric. 
\[
\varepsilon=-1\quad\hbox{for }w=u,\qquad
\varepsilon=+1\quad\hbox{for }w=v.
\]
We write the rescaled inverse metric  $\hbt+O$ in coordinates $(\wt, \rt, \vphit, \tht )$ in the following matrix form 
\beq\lab{eq:Hbt-rescaled-inverse-matrix}
\Hbt=\begin{pmatrix}
\hbt^{\wt\wt}+a^2\sin^2\tht & \hbt^{\wt\rt} & a(1-\varepsilon X) &0\\
\hbt^{\wt\rt} & \hbt^{\rt\rt} & 0&0\\
a(1-\varepsilon X) & 0 & \csc^2\tht&0\\
0&0&0 &1
\end{pmatrix}.
\eeq

\begin{lemma}
\lab{lemma:det-Hbt-Pepsilon}
We have
\beq\lab{eq:det-Hbt-Pepsilon}
\bsplit
\det \Hbt
&=S^{-2}X^2\csc^2\tht\,
P_\varepsilon(X,\tht),\\
P_\varepsilon(X,\tht)
&=-(m^2-a^2\cos^2\tht)X^2
-4\varepsilon a^2\cos^2\tht\,X
+4a^2\cos^2\tht.
\end{split}
\eeq
\end{lemma}
\begin{proof}
The matrix   can be written in the form.
\[
\Hbt=
\begin{pmatrix}
A&b&c&0\\
b&d&0&0\\
c&0&e&0\\
0&0&0&1
\end{pmatrix},
\]
where
\[
A=\hbt^{\wt\wt}+a^2\sin^2\tht,\qquad
b=\hbt^{\wt\rt},\qquad
c=a(1-\varepsilon X),\qquad
d=\hbt^{\rt\rt},\qquad
e=\csc^2\tht .
\]
The angular block $\operatorname{diag}(e,1)$ is positive.  Hence, by the
Schur complement criterion (see Lemma \ref{lemma:Schur-complement-positivity})
$\Hbt$ is positive definite if and only if
\[
K:=
\begin{pmatrix}
A-\frac{c^2}{e}&b\\
b&d
\end{pmatrix}
=
\begin{pmatrix}
\hbt^{\wt\wt}
+a^2\sin^2\theta\big(1-(1-\varepsilon X)^2\big)&\hbt^{\wt\rt}\\
\hbt^{\wt\rt}&\hbt^{\rt\rt}
\end{pmatrix}
\]
is positive definite.  We compute
\begin{align*}
A-\frac{c^2}{e}&=\hbt^{\wt\wt}+a^2\sin^2\tht\big(1-(1-\varepsilon X)^2\big)\\
&=\Delta_R X^2-2\varepsilon A_R X
+a^2\sin^2\tht(2\varepsilon X-X^2)\\
&=(R^2-2mR+a^2\cos^2\tht)X^2
-2\varepsilon(R^2+a^2\cos^2\tht)X,\\
b&=S^{-1}X\big((R-m)X-2\varepsilon R\big),\\
d&=\hbt^{\rt\rt}=S^{-2}X(X-2\varepsilon).
\end{align*}
Therefore, writing $Y=a^2\cos^2\tht$, we have
\[
\begin{split}
\det K
&=d\Big(A-\frac{c^2}{e}\Big)-b^2\\
&=S^{-2}X^2\Big[
(X-2\varepsilon)\big((R^2-2mR+Y)X
-2\varepsilon(R^2+Y)\big)
-\big((R-m)X-2\varepsilon R\big)^2
\Big]\\
&=S^{-2}X^2\Big[
\big(R^2-2mR+Y-(R-m)^2\big)X^2\\
&\hspace{2.4cm}+
\big(-2\varepsilon(2R^2-2mR+2Y)
+4\varepsilon R(R-m)\big)X
+4(R^2+Y)-4R^2
\Big]\\
&=S^{-2}X^2\Big[
-(m^2-a^2\cos^2\tht)X^2
-4\varepsilon a^2\cos^2\tht\,X
+4a^2\cos^2\tht
\Big]\\
&=S^{-2}X^2\,P_\varepsilon(X,\tht).
\end{split}
\]
Equivalently, since $e=\csc^2\tht$, this proves
$\det\Hbt=e\det K=\csc^2\tht\det K$.
\end{proof}

We then have the following result.
\begin{proposition}
\lab{Prop:Lorentzian}
The rescaled metric $|\qt|^2\gt$ of Proposition \ref{Lemma:Whiting-metric-general} is Lorentzian  in the region where 
\beq
P_\varepsilon(X,\tht)
=-(m^2-a^2\cos^2\tht)X^2
-4\varepsilon a^2\cos^2\tht\,X
+4a^2\cos^2\tht<0,
\eeq
i.e. in the complement of the interval between  the roots of  $P_\varepsilon(X,\tht)$
\begin{align*}
X_-^{(\varepsilon)}(\tht)
&=-\frac{2|a\cos\tht|}{m-\varepsilon|a\cos\tht|}, \qquad 
X_+^{(\varepsilon)}(\tht)
=\frac{2|a\cos\tht|}{m+\varepsilon|a\cos\tht|}.
\end{align*}
The vectorfield $\T=\pr_\wt$  is timelike  in the region where the conformal metric $|\qt|^2\gt$ is Lorentzian  and  
\beq
X(X-2\varepsilon)>0.
\eeq
\end{proposition}
\begin{proof}
According to Lemma \ref{lemma:det-Hbt-Pepsilon},  the rescaled metric $|\qt|^2\gt$ is Lorentzian if and only if $ P_\varepsilon(X,\tht)<0$. That is, $|\qt|^2\gt$ is Lorentzian in the complement of the interval between  the roots of  $P_\varepsilon(X,\tht)$
\begin{align*}
X_-^{(\varepsilon)}(\tht)
&=-\frac{2|a\cos\tht|}{m-\varepsilon|a\cos\tht|}, \qquad 
X_+^{(\varepsilon)}(\tht)
=\frac{2|a\cos\tht|}{m+\varepsilon|a\cos\tht|}.
\end{align*}
 We next observe that
\[
 |\qt|^{-2}\gt_{\wt\wt}=\frac{\hbt^{\rt\rt}\csc^2\theta}{\det \Hbt}=\frac{\hbt^{\rt\rt}}{\det K}.
\]
Therefore,  the vectorfield $\pr_\wt= T$ is timelike precisely  in the region where  the conformal metric $|\qt|^2\gt$ is Lorentzian and $\hbt^{\rt\rt}=S^{-2}X(X-2\varepsilon)>0$.
\end{proof}
 
 We now discuss the geometric character of $|\qt|^2\gt$.
\begin{proposition}
\lab{Prop:general-signature-classification}
The rescaled metric $|\qt|^2\gt$ is degenerate on $X=0$ and on the hypersurfaces $X=X_{\pm}^{(\varepsilon)}(\tht)$ where $P_\varepsilon=0$.  

At the points where $X\neq0,X_{\pm}^{(\varepsilon)}(\tht)$, the signature of 
$|\qt|^2\gt$ is determined by the signs of
\[
P_\varepsilon(X,\tht),\qquad X(X-2\varepsilon).
\]
More precisely,
\[
\begin{array}{c|c|c}
\hbox{Signs} & \hbox{Character of }||\qt|^2\gt & \hbox{Character of }\T
\\ \hline
P_\varepsilon<0,\ X(X-2\varepsilon)>0
& \hbox{Lorentzian} & \hbox{timelike} 
\\
P_\varepsilon<0,\ X(X-2\varepsilon)=0
& \hbox{Lorentzian} & \hbox{null} 
\\
P_\varepsilon<0,\ X(X-2\varepsilon)<0
& \hbox{Lorentzian} & \hbox{spacelike} 
\\
P_\varepsilon>0,\ X(X-2\varepsilon)>0
& \hbox{Riemannian} & \hbox{--} 
\\
P_\varepsilon>0,\ X(X-2\varepsilon)<0
& \hbox{split }(2,2) & \hbox{--} 
\end{array}
\]
In the Lorentzian region, the hypersurface $X=2\varepsilon$ is null.
\end{proposition}

\begin{proof}
By Lemma \ref{lemma:det-Hbt-Pepsilon}, $\det\Hbt=S^{-2}X^2\csc^2\tht\,P_\varepsilon$.  Hence the metric
is degenerate when $X=0$ or when $P_\varepsilon=0$, i.e. at
$X=X_\pm^{(\varepsilon)}(\tht)$.

We recall the matrix $\Hbt$ introduced in \eqref{eq:Hbt-rescaled-inverse-matrix}. The
$(\vphit,\tht)$-angular Schur block is positive, so the signature of
$\Hbt$ is determined by the reduced $2\times2$ Schur complement
\[
K=
\begin{pmatrix}
\hbt^{\wt\wt}
+a^2\sin^2\theta\big(1-(1-\varepsilon X)^2\big)&\hbt^{\wt\rt}\\
\hbt^{\wt\rt}&\hbt^{\rt\rt}
\end{pmatrix}.
\]
Since $S^{-2}>0, X\neq 0$, the sign of $\det K=S^{-2}X^2P_{\varepsilon}$ is the same as that of $P_{\varepsilon}$.

If $P_\varepsilon<0$, then $K$ has one positive and one negative
eigenvalue.  Adding the positive angular block gives Lorentzian signature.
In this Lorentzian region,
\[
|\qt|^{-2}\gt_{\wt\wt}
=\frac{\hbt^{\rt\rt}\csc^2\tht}{\det\Hbt}=\frac{\hbt^{\rt\rt}}{\det K},
\]
and $\det K<0$.  Hence $\T=\pr_\wt$ is timelike, null, or spacelike
accordingly as $\hbt^{\rt\rt}=S^{-2}X(X-2\varepsilon)$, equivalently $X(X-2\varepsilon)$, is
positive, zero, or negative.

If $P_\varepsilon>0$, then $K$ is definite.  It is positive definite when
$\hbt^{\rt\rt}>0$, i.e. when $X(X-2\varepsilon)>0$, and negative definite
when $\hbt^{\rt\rt}<0$, i.e. when $X(X-2\varepsilon)<0$.  Adding the positive
angular block gives respectively a Riemannian metric or a split-signature
$(2,2)$ metric.

The hypersurface $X=2\varepsilon$ is
null in the Lorentzian region because its normal is proportional to $d\rt$
and
\[
|\qt|^2\gt^{\alpha\beta}\pr_\alpha \rt\pr_\beta \rt
=\hbt^{\rt\rt}=S^{-2}X(X-2\varepsilon)
\]
which vanishes there.  This completes the proof.
\end{proof}


\subsubsection{Relation between the cases
$\varepsilon=-1$ and $\varepsilon=1$}
\lab{subsection:precise-relation-epsilon-pm}

There are two relations.  Firstly, let $X$ be the outgoing ($\varepsilon=-1$)
slope and $Y$ the ingoing ($\varepsilon=1$) slope.  The algebraic reflection
\beq\lab{eq:epsilon-signature-reflection}
Y=-X
\eeq
gives
\[
P_1(Y,\tht)=P_{-1}(X,\tht),\qquad Y(Y-2)=X(X+2),
\]
and therefore identifies the complete signature tables in Proposition \ref{Prop:general-signature-classification}, in reversed order.
In particular, $Y_-^{(1)}=-X_+^{(-1)}, Y_+^{(1)}=-X_-^{(-1)}$.

Secondly, the outgoing transform is adapted to the past horizon,
whereas the ingoing transform carries the future-horizon endpoint term which
has to be kept in the identity.

Therefore, modulo the horizon endpoint distinction
just mentioned, the two cases $\varep=-1$ and $\varep=+1$ provide equivalent geometric character of the rescaled metric $|\qt|^2\gt$. Consequently, in what follows, {\bf we will treat
only the case $\varep=-1$}. We keep track of the  quantity $X=S(\rt-\Rt)$   which, in principle, 
can take any value  in  $(-\infty, +\infty)$.

\subsection{Geometric properties of the transformed metric}

As mentioned above  we restrict our attention to the case $\varep=-1$.  Hence in Definition \ref{Def:WhitingTr}   $w$ is to be replaced by the 
 outgoing coordinate $u$ and $\vphi=\phim$.  We shall also replace $\wt$ by $\vt$ rather than $\ut$. The reason for this, as  we shall see below, is that   in the case  $X<-2$,     $\vt$  corresponds in fact to  an ingoing variable. Therefore
 \[
\psit(\vt, \rt,  \tht, \vphit ):=\int_{r_+}^\infty \psi\big(\vt+B(\rt,r), r, \tht,  \vphit\big)\,dr, \qquad B(\rt,r):=S(\rt-\Rt)(r-R).
\]
\begin{definition}[Canonical representation]
\lab{def:canonical-representation}
We define the
canonical representation by
\beq\lab{eq:canonicalrepr}
S=-\frac{2}{r_+-r_-},\qquad \Rt=R=r_-,\qquad
X=-2\frac{\rt-r_-}{r_+-r_-}.
\eeq
Equivalently,
\[
\rt=r_--\frac{r_+-r_-}{2}X.
\]
 In particular,
\[
X<-2\Longleftrightarrow \rt>r_+,\qquad
-2<X<0\Longleftrightarrow r_-<\rt<r_+,\qquad
X>0\Longleftrightarrow \rt<r_-,
\]
while $X=-2$ corresponds to $\rt=r_+$ and $X=0$ corresponds to
$\rt=r_-$.  

The canonical representation of Whiting transform is
\beq\lab{eq:canonical-Whiting-transform-exterior}
\psit(\vt,\rt,\tht,\vphit)
=\int_{r_+}^{\infty}
\psi\left(\vt-2\frac{\rt-r_-}{r_+-r_-}(r-r_-),r,\tht,\vphit\right)\,dr .dence 
\eeq
\end{definition}

\begin{figure}[!ht]
\qquad \qquad  \includegraphics[width=5.4in]{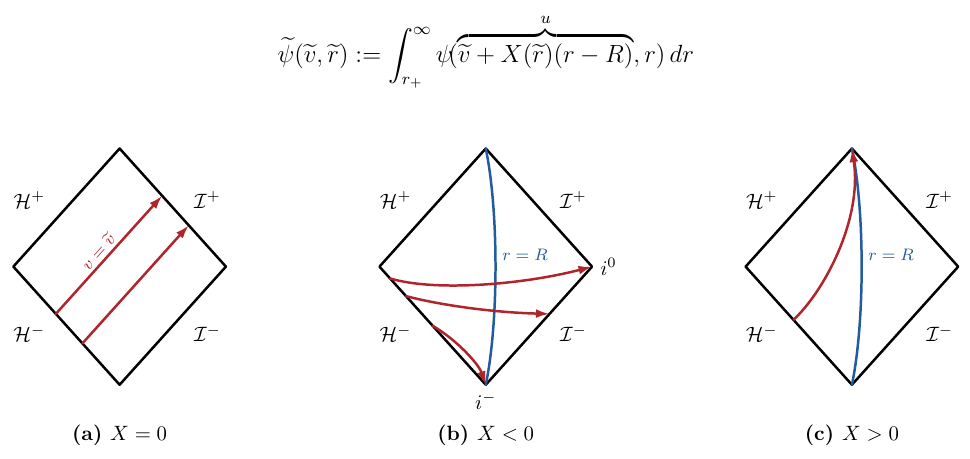}  
\caption{ Integration curves for  the Whiting transform  depending  on the sign of $X$}
\label{fig:Wtransform}
\end{figure}

Expressed relative to this representation the metric $\hbt$ takes the form 
 \beq
 \lab{eqhbt-canonical}
 \bsplit
 \hbt^{\rt\rt}&=\De(\rt),\qquad
 \hbt^{\vt\rt}=\hbt^{\rt\vt}=\rt^2-r_-^2,\qquad
 \hbt^{\vt\vt}=-\frac{8mr_-}{r_+-r_-}(\rt-r_-)\\
 \hbt^{\vt\vphit}&=\hbt^{\vphit\vt}=-\frac{2a(\rt-r_-)}{r_+-r_-}.
 \end{split}
 \eeq
  \begin{remark}
 In the particular case of the Schwarzschild  $\g$, i.e.  $a=0$,  the metric $\gt$ coincides (with the choice $|\qt|^2=\rt^2$) with  the original metric  $\g$ in ingoing coordinates. 
\end{remark}
\begin{proof}
Indeed  in  that case $r_-=0$,  $r_+ = 2m$,  therefore  in  the canonical  representation,  
\[
\hbt^{\rt\rt} = \De(\rt), \qquad  \hbt^{\vt\rt} =\hbt^{\rt\vt}= \rt^2,\qquad  \hbt^{\vt\vt}=
 \hbt^{\vt\vphit}=\hbt^{\vphit\vt}=0.
\]
which corresponds precisely to the Schwarzschild metric in ingoing coordinates.
\end{proof} 
 
\begin{remark}
The canonical representation\footnote{This is also  closer in spirit with the original Whiting  mode based integral   transformation.} is only useful to make comparisons with the original Kerr  or Schwarzschild spacetime, such as in the remark above.   Throughout  our  work, however, we will use the general representation.
\end{remark}


\subsubsection{Geometric character of the transformed metric and vectorfield $\T$}

  Denoting
	  \[
	  P(X, \tht)= P_{-1} (X, \tht),  \qquad
	  X_{-}(\tht)=X_-^{(-1)}(\tht), \quad
	  X_{+}(\tht)=X_+^{(-1)}(\tht),
	  \]
	  we have
	  \begin{align*}
X_-(\theta)
&=-\frac{2|a\cos\tht|}{m+|a\cos\tht|}, \qquad 
X_+(\theta)
=\frac{2|a\cos\tht|}{m-|a\cos\tht|}.
\end{align*}
Note  that $-2 <X_-(\tht) \leq 0\leq  X_+(\tht)$. Away from the equator $\cos\tht=0$ these inequalities are strict; at $\cos\tht=0$, the two roots collapse to $X_-(\tht)=X_+(\tht)=0$.

 The following proposition describes the geometric character 
 of the rescaled metric  $|\qt|^2\gt$ and the causality properties  of the vectorfield $\T$  depending on the  range  of values of $X$.
 
\begin{proposition}
\lab{Prop:outgoing-Snegative-regions}
Away from the equatorial set
$\cos\tht=0$, the following holds true.
\[
\begin{array}{c|c|c}
\hbox{Region} & \hbox{Character of }|\qt|^2\gt & \hbox{Character of }\T
\\ \hline
X<-2 & \hbox{Lorentzian} & \hbox{timelike} 
\\
X=-2 & \hbox{Lorentzian} & \hbox{null} 
\\
-2<X<X_-(\th) & \hbox{Lorentzian} & \hbox{spacelike} 
\\
X=X_-(\th) & \hbox{degenerate} & \hbox{--} 
\\
X_-(\th)<X<0 & \hbox{split }(2,2) & \hbox{--} 
\\
X=0 & \hbox{singular} & \hbox{--} 
\\
0<X<X_+(\th) & \hbox{Riemannian} & \hbox{--} 
\\
X=X_+(\th) & \hbox{degenerate} & \hbox{--} 
\\
X>X_+(\th) & \hbox{Lorentzian} & \hbox{timelike} 
\end{array}
\]
\end{proposition}

\begin{proof}
   The proof follows from Proposition \ref{Prop:general-signature-classification}. See also the Figure  \ref{fig:KKt}.
   \end{proof}

\begin{remark}
\lab{rem:conformalfactor}
Recall from \eqref{eq:Hbt-rescaled-inverse-matrix} that $\Hbt$ is the
matrix of components of the rescaled inverse metric $|\qt|^2\gt$.  In the region where $\det \Hbt\neq 0$, we choose the conformal factor as
\beq\lab{eq:conformal-factor-L2-branches}
|\qt|^2
:=\sin\tht\,\sqrt{|\det\Hbt|}
=\frac{|X|}{|S|}\sqrt{|P(X,\tht)|}.
\eeq
Since
\[
\det\Hbt
=|\qt|^8\det(\gt^{-1})
=\frac{|\qt|^8}{\det\gt},
\]
it follows that 
\begin{align*}
|\qt|^2\frac{1}{\sqrt{|\det \gt|}}\pr_{\at}(\sqrt{|\det\gt|}\gt^{\at\bt}\pr_{\bt}\psit)&=\frac{1}{\sin\tht}\pr_{\at}(\sin\tht|\qt|^2\gt^{\at\bt}\pr_{\bt}\psit)\\
&=\pr_{\at}\big(\hbt^{\at\bt}\pr_{\bt}\psit\big)+\frac{1}{\sin\tht}\pr_{\at}\big(\sin\tht O^{\at\bt}\pr_{\bt}\psit\big).
\end{align*}
This explains why we use the notation $|\qt|^2\square_{\gt}$ for the expression $\pr_{\at}\big(\hbt^{\at\bt}\pr_{\bt}\psit\big)+\frac{1}{\sin\tht}\pr_{\at}\big(\sin\tht O^{\at\bt}\pr_{\bt}\psit\big)$ in \eqref{eq:defofqtsquaregt}. We note that this expression in \eqref{eq:defofqtsquaregt} is defined everywhere in $\DDt$ while the Laplace-Beltrami operator of $\gt$ is only defined for $\det \Hbt\neq 0$.

We also compute
\[
|\qt|^{-2}dv_\gt
=\frac{|\qt|^2}{\sqrt{|\det\Hbt|}}\,
d\taut\,d\rt\,d\tht\,d\vphit.
\]
The choice \eqref{eq:conformal-factor-L2-branches} therefore gives
\beq
\lab{eq:flatvolume}
|\qt|^{-2}dv_\gt
=\sin\tht\,d\taut\,d\rt\,d\tht\,d\vphit
=d\taut\,d\rt\,d\omega.
\eeq
The identity is first obtained in the region where $\det \Hbt\neq 0$ and defines
the corresponding integral density. It extends to the region where $\det \Hbt=0$ by the right hand side of
\eqref{eq:flatvolume}; this defines $|\qt|^{-2}dv_\gt$ over
all of $\DDt$. 
\end{remark}

\subsubsection{BL--type  coordinates}  
\lab{subsection:BLcoords}


To   write the  metric $\gt$  in a more  symmetric form   we   make the change of coordinates
\beq
\lab{eq:BL-Whiting-general}
\tt=\vt-\rt^*(\rt), \qquad
\frac{d\rt^*}{d\rt}
=\frac{(R-m)S(\rt-\Rt)+2 R}
{\rt-\Rt+\frac{2}{S}} .
\eeq
designed to  imply  $\hbt^{\tt\rt}=0$. Indeed,
keeping the other coordinates  unchanged,
\begin{align*}
\hbt^{\tt\rt}
&=\hbt^{\vt\vt}\frac{\pr\tt}{\pr\vt}\frac{\pr\rt}{\pr\vt}
+\hbt^{\vt\rt}\frac{\pr\tt}{\pr\vt}\frac{\pr\rt}{\pr\rt}
+\hbt^{\vt\rt}\frac{\pr\rt}{\pr\vt}\frac{\pr\tt}{\pr\rt}
+\hbt^{\rt\rt}\frac{\pr\tt}{\pr\rt}\frac{\pr\rt}{\pr\rt}\\
&= \hbt^{\vt\rt}
+\hbt^{\rt\rt}
\bigg( -\frac{(R-m)S(\rt-\Rt)+ 2R}
{\rt-\Rt+\frac{2}{S}}\bigg)=0,\\
\hbt^{\tt \vphit}
&=\hbt^{\vt \vphit} \frac{\pr\tt}{\pr\vt}
\frac{\pr\vphit}{\pr\vphit}= \hbt^{\vt \vphit}.
\end{align*}
Also,
\begin{align*}
\hbt^{\tt\tt}
&=\hbt^{\vt\vt}\frac{\pr\tt}{\pr\vt}\frac{\pr\tt}{\pr\vt}
+2\hbt^{\vt\rt}\frac{\pr\tt}{\pr\vt}\frac{\pr\tt}{\pr\rt}
+\hbt^{\rt\rt}\frac{\pr\tt}{\pr\rt}\frac{\pr\tt}{\pr\rt}\\
&=\hbt^{\vt\vt}+\hbt^{\vt\rt}\frac{\pr\tt}{\pr\rt}
+\frac{\pr\tt}{\pr\rt}
\bigg(\hbt^{\vt\rt}+\hbt^{\rt\rt}\frac{\pr\tt}{\pr\rt}\bigg)\\
&=\hbt^{\vt\vt}+\hbt^{\vt\rt}\frac{\pr\tt}{\pr\rt}
=\hbt^{\vt\vt}+\hbt^{\vt\rt}
\bigg( -\frac{(R-m)S(\rt-\Rt)+2 R}
{\rt-\Rt+\frac{2}{S}}\bigg)
=\frac{f(\rt)}{\rt-\Rt+\frac{2}{S}},
\end{align*}
where
\beq
\lab{def-f(rt)-general}
\bsplit
f(\rt)&=(a^2-m^2)S^2(\rt-\Rt)^3
+4 a^2S(\rt-\Rt)^2+4a^2(\rt-\Rt).
\end{split}
\eeq
The non-vanishing components of $\hbt$, in BL type coordinates  are therefore
 \beq
 \lab{eq:BL-hbt.coordinates-general}
  \hbt^{\tt\tt}=\frac{f(\rt)}{\rt-\Rt+\frac{2}{S}}, \quad
  \hbt^{\tt\rt}=0, \quad
  \hbt^{\rt\rt}=(\rt-\Rt)^2+\frac{2}{S}(\rt-\Rt),\quad
 \hbt^{\tt\vphit}=\hbt^{\vphit\tt}= aS(\rt-\Rt).
\eeq
  The $O$ part of the metric remains unchanged, with $\tht=\th, \vphit=\vphi$,
  \[
  O^{\a\b}= \pr_{\tht}^\a \pr_{\tht}^\b +\frac{1}{\sin^2\tht} \pr_{\vphit}^\a \pr_{\vphit}^\b +  2 a \pr_\tt^{(\a} \pr_{\vphit}^{\b)}+ a^2\sin^2 \tht \pr_\tt^\a \pr_\tt^\b.
  \]

\subsubsection{Outgoing  coordinates.}  
 \lab{section:BL-Whiting-out}

 To describe the null infinities of $\gt$, we introduce the coordinates $(\ut,\rt, \tht, \vphit_-)$  as follows (the upper sign corresponds to
  $\rt>\Rt-\frac{2}{S}$ and the lower sign to $\rt<\Rt$).
   \beq\lab{eq:BL-Whiting-outXneg}
  \ut=\tt\mp\int \sqrt{-\frac{\hbt^{\tt\tt}}{\hbt^{\rt\rt}}},\qquad \vphit_-=\vphit\mp\int\frac{ a }{\sqrt{m^2-a^2}}\frac{1}{\rt-\Rt}.
  \eeq 
  where
  \begin{align*}
 \frac{\hbt^{\tt\tt}}{\hbt^{\rt\rt}}&=-\frac{(m^2-a^2)S^2(\rt-\Rt)^2
-4 a^2S(\rt-\Rt)-4a^2}{(\rt-\Rt+\frac{2}{S})^2}\\
&=-\frac{(m^2-a^2)S^2(\rt-\Rt+\frac{2}{S})^2
-2 (m^2+a^2)S(\rt-\Rt+\frac{2}{S})+4m^2}{(\rt-\Rt+\frac{2}{S})^2} 
\end{align*}

\begin{lemma}[Metric $\gt$  in outgoing coordinates]
\lab{lem:exact-outgoing-inverse-metric}
  In the coordinates $(\ut,\rt,\tht,\vphit_-)$ defined above, the
  nonvanishing inverse metric components are given by the following
  formulas (the upper sign corresponds to
  $\rt>\Rt-\frac{2}{S}$ and the lower sign to $\rt<\Rt$):
  \beq
  \lab{eq:metricoutgoing}
  \begin{aligned}
  |\qt|^2\gt^{\ut\ut}
  &=a^2\sin^2\tht,\\
  |\qt|^2\gt^{\ut\rt}
  &=\mp\sqrt{-\hbt^{\tt\tt}\hbt^{\rt\rt}}=\mp \sqrt{m^2-a^2}\,|S|(\rt-\Rt)^2
  +O(|\rt-\Rt|),\\
  |\qt|^2\gt^{\ut\vphit_-}
  &=a+aS(\rt-\Rt)
  +\frac{ a}{\sqrt{m^2-a^2}(\rt-\Rt)}
  \sqrt{-\hbt^{\tt\tt}\hbt^{\rt\rt}}=O(a),\\
  |\qt|^2\gt^{\rt\rt}
  &=(\rt-\Rt)\left(\rt-\Rt-\frac{2\varepsilon}{S}\right),\\
  |\qt|^2\gt^{\rt\vphit_-}
  &=\mp\frac{ a}{\sqrt{m^2-a^2}}
  \left(\rt-\Rt+\frac{2}{S}\right),\\
  |\qt|^2\gt^{\tht\tht}
  &=1,\\
  |\qt|^2\gt^{\vphit_-\vphit_-}
  &=\frac{1}{\sin^2\tht}
  +\frac{a^2}{m^2-a^2}
  \frac{\rt-\Rt+\frac{2}{S}}{\rt-\Rt}.
  \end{aligned}
    \eeq
\end{lemma}
\begin{proof}
Straightforward calculation.
\end{proof}

      \begin{remark} Note that  the  formulas   \eqref{eq:metricoutgoing}   are all  independent of $R$.
      \end{remark}
   
   \begin{proposition}
   \lab{prop:nullinfinities}
 In the outgoing coordinates $(\ut, \rt, \tht, \vphit_-)$ defined in \eqref{eq:BL-Whiting-outXneg}, the followings hold true.
     \begin{enumerate}
\item In the region $X=S(\rt-\Rt)\le -2$ where $\gt$ is Lorentzian,  the metric components are smooth  and  $\gt$ has a smooth past event horizon $\widetilde{\HH}^-=\{\ut, X=-2, \tht, \vphit_-\}$. Moreover, the metric $\gt$ is asymptotically flat and has a complete smooth future null infinity 
\beq
\lab{eq:negativebranchI.I^_+}
\widetilde{\II}^+=\{\ut, X=-\infty, \tht, \vphit_-\}.
\eeq
  \item In the region $X>X_+(\tht)=\frac{2|a\cos\tht|}{m-|a\cos\tht|}$ where $\gt$ is Lorentzian, $\gt$ is asymptotically flat and has a complete smooth future null infinity 
  \beq
  \lab{eq:positivebranchI.I^_+}
  \widetilde{\II}^+=\{\ut, X=\infty, \tht, \vphit_-\}. 
	  \eeq
	   \end{enumerate} 
\end{proposition}	   
\begin{proof}
By inspecting the metric components \eqref{eq:metricoutgoing} in the outgoing coordinates $(\ut, \rt, \tht, \vphit_-)$.
\end{proof}	   
	 

\subsection{The general  bilinear identity}
\lab{subsection:general-outgoing-bilinear-identity}

We recall that, with $X=S(\rt-\Rt)$,
\beq\lab{eq:general-outgoing-bilinear-transform}
\psit_i(\vt,\rt)
=\Wk[\psi_i](\vt, \rt)
:=\int_{r_+}^{\infty}
\psi_i\bigl(\vt+X(r-R),r\bigr)\,dr.
\eeq
The angular variables are suppressed and all identities remain valid after
integrating them with the same angular measure.  We use
\[
\widehat f(\xi)=\int_{\RRR}e^{-iu\xi}f(u)\,du,
\qquad
\widehat{H_uf}(\xi)=-i\operatorname{sgn}(\xi)\widehat f(\xi).
\]
we also define the fractional derivatives 
\beq
\lab{eq:fractionalT}
\widehat{|T|^{1/2}f}=|\xi|^{1/2}\widehat{f}.
\eeq
For simplicity we will often write $T^{1/2}$ instead of $ |T|^{1/2}$.

\begin{proposition}[General bilinear identity]
\lab{prop:bilinear-outgoing-general}
Assume that $\psi_i$ are sufficiently regular and decaying, and let
$I\subset\RRR$ be any interval in the $X$-line.  Then
\beq\lab{eq:bilinear-outgoing-general}
|S|\int_{\{X\in I\}}\int_{\RRR}
\pr_{\taut}\psit_1\,\overline{\psit_2}\,d\taut\,d\rt
=\frac1{2\pi}\int_{\RRR}\int_{r_+}^{\infty}\int_{r_+}^{\infty}
\widehat{\psi_1}(\xi,r)\overline{\widehat{\psi_2}(\xi,s)}
\left(\int_I i\xi e^{i\xi X(r-s)}\,dX\right)ds\,dr\,d\xi
\eeq
where $\taut$ is a regular time function in the region $\DDt\cap\{X\in I\}$.

In particular, if \(I=(\alpha,\beta)\) is finite, then
\beq\lab{eq:bilinear-outgoing-finite-interval}
\bsplit
&|S|\int_{X\in I}\int_{\RRR}
\pr_{\taut}\psit_1\,\overline{\psit_2}\,d\taut\,d\rt
\\
&\qquad
=\operatorname{p.v.}\int_{\RRR}
\int_{r_+}^{\infty}\int_{r_+}^{\infty}
\frac{
\psi_1\bigl(u+\beta r,r\bigr)
\overline{\psi_2\bigl(u+\beta s,s\bigr)}
}{r-s}\,ds\,dr\,du\\
&\qquad
-\operatorname{p.v.}\int_{\RRR}
\int_{r_+}^{\infty}\int_{r_+}^{\infty}
\frac{
\psi_1\bigl(u+\alpha r,r\bigr)
\overline{\psi_2\bigl(u+\alpha s,s\bigr)}
}{r-s}\,ds\,dr\,du .
\end{split}
\eeq
If one of the endpoints
is infinite, for example, for $I=(-\infty,-2)$,
\beq\lab{eq:bilinear-outgoing-half-line}
\bsplit
&|S|\int_{\{X<-2\}}\int_{\RRR}
\pr_{\taut}\psit_1\,\overline{\psit_2}\,d\taut\,d\rt
=\pi\int_{\RRR}\int_{r_+}^{\infty}
\psi_1(u,r)\,\overline{H_u\psi_2(u,r)}\,dr\,du\\
&+\operatorname{p.v.}\int_{\RRR}
\int_{r_+}^{\infty}\int_{r_+}^{\infty}
\psi_1\bigl(u-2(r-R),r\bigr) \times
\overline{\psi_2\bigl(u-2(s-R),s\bigr)}
\frac{ds\,dr\,du}{r-s}.
\end{split}
\eeq
where $H_u$ is the Hilbert transform  with respect to $u$.
Also, for  the entire interval  $I=\RRR$,
\beq\lab{eq:bilinear-outgoing-full-line}
|S|\int_{X\in\RRR}\int_{\RRR}
\pr_{\taut}\psit_1\,\overline{\psit_2}\,d\taut\,d\rt
=2\pi\int_{\RRR}\int_{r_+}^{\infty}
\psi_1(u,r)\,\overline{H_u\psi_2(u,r)}\,dr\,du .
\eeq
\end{proposition}

\begin{proof}
Fourier transformation of \eqref{eq:general-outgoing-bilinear-transform}
in $\taut=\vt+f(\rt)$ gives
\[
\widehat{\psit_i}(\xi,X)
=e^{-i\xi XR}e^{-i\xi f(\rt)}\int_{r_+}^{\infty}
e^{i\xi Xr}\widehat{\psi_i}(\xi,r)\,dr.
\]
Since $X=S(\rt-\Rt)$, integration over $I$ with respect to $|S|d\rt$ and applying Plancherel give
\eqref{eq:bilinear-outgoing-general}.  If \(I=(\alpha,\beta)\), then
\[
\int_\alpha^\beta i\xi e^{i\xi X\rho}\,dX
=\frac{e^{i\xi \beta\rho}-e^{i\xi \alpha\rho}}{\rho},
\qquad \rho=r-s,
\]
which gives \eqref{eq:bilinear-outgoing-finite-interval} after application of Plancherel, with the singularity at $r=s$ understood in the
principal value sense.

Finally, for the half-line $I=(-\infty,-2)$,  we have
\beq\lab{eq:bilinear-outgoing-half-line-distribution}
\int_{-\infty}^{-2} i\xi e^{i\xi X\rho}dX=\int_2^\infty i\xi e^{-i\xi\lambda\rho}\,d\lambda
=\pi i\operatorname{sgn}(\xi)\delta(\rho)
+e^{-2i\xi\rho}\operatorname{p.v.}\frac1\rho.
\eeq
The delta term, together with the Hilbert-transform convention above,
gives the first term on the right of
\eqref{eq:bilinear-outgoing-half-line}.  Plancherel converts the phase
$e^{-2i\xi(r-s)}$ into the two shifts
$u-2r$ and $u-2s$, proving the principal value term.

For the full interval \(I=\RRR\), we use instead
\[
\int_{-\infty}^{\infty} i\xi e^{i\xi X\rho}\,dX
=2\pi i\,\operatorname{sgn}(\xi)\delta(\rho).
\]
The $\delta(\rho)$ collapses the $s$-integral to $s=r$, and the
multiplier $i\operatorname{sgn}(\xi)$ is exactly the multiplier which gives
$2\pi\,\psi_1\overline{H_u\psi_2}$ after Plancherel with the convention
fixed above.  This proves \eqref{eq:bilinear-outgoing-full-line}.
\end{proof}

\begin{remark}
\lab{rem:bil.smallangle}
Note that the full-line bilinear estimate
\eqref{eq:bilinear-outgoing-full-line} does not contain the hard to handle
Hilbert transform with respect to $r$ on the right hand side. If we were able to control $\psit_1,\psit_2$ in the region $\{X>0\}\cup \{X\leq -\de\}$, then we have control on the whole transformed manifold except for the small region $-\de<X<0$.  In that case the corresponding bilinear identity has only a
small radial correction on the right hand side.
\beq\lab{eq:bilinear-small-aperture-volume}
|S|\int_{\KKt_\de}
\bigl(\pr_\vt\psit_1\bigr)\overline{\psit_2}\,
|\qt|^{-2}\,dV_\gt
=2\pi\int_{\DD}\psi_1\,\overline{H_u\psi_2}\,
|q|^{-2}\,dV_\g
-\int_{\DD}\psi_1\,\overline{\mathcal A_\de\psi_2}\,
|q|^{-2}\,dV_\g ,
\eeq
where $\KKt_\de:
=\KKt\setminus\{-\de\leq X\leq0\}$ and  the small radial correction is
\beq\lab{eq:def-small-aperture-operator}
(\mathcal A_\de f)(u,r,\omega)
:=\operatorname{p.v.}\int_{r_+}^{\infty}
\frac{
f(u,s,\omega)
-f\bigl(u+\de(r-s),s,\omega\bigr)
}{r-s}\,ds .
\eeq
\end{remark}

\begin{corollary}[$L^2$ boundedness]
\lab{cor:filtered-L2-outgoing-transform}
Set
\[
I_-:=(-\infty,0),\qquad I_+:=(0,\infty),
\]
and let $\DDt_\sigma:=\{X\in I_\la\}$, where
$\la\in\{-,+\}$.  Here $\DD=\{r>r_+\}$ denotes the domain of outer
communication of the original Kerr spacetime.  Then, for sufficiently regular
$f=f(u,r,\omega)$,
\bea
\lab{eq:MainL2estimate}
\int_{\DDt_\la}
\left|\Wk[T^{1/2}f]\right|^2
|\qt|^{-2}\,dv_\gt
\lesssim
\int_{\DD}|f|^2|q|^{-2}\,dv_\g .
\eea
If $f$ is supported away from $r=r_+$, then with $B(\rt)=S(\rt-\Rt)=X$
\bea
\lab{eq:e_4trick}
\int_{\DDt_\la}
\left|B(\rt)\partial_\ut\Wk[T^{1/2}f]\right|^2
|\qt|^{-2}\,dv_\gt
\lesssim
\int_{\DD}\left|e^{(out)}_4f\right|^2|q|^{-2}\,dv_\g .
\eea
The integral density $|\qt|^{-2}\,dv_\gt$ is defined as in \eqref{eq:flatvolume}.
\end{corollary}

\begin{proof}
We recall
\[
\Wk[f](\vt)=\int_{r_+}^{\infty}f(\vt+X(r-R),r)\,dr .
\]
If $E_0f$ denotes the zero extension of $f$ from $r>r_+$ to the full
radial line, define its two-dimensional Fourier transform by
\[
\widehat{E_0f}(\xi,\eta)
:=\int_{\RRR}\int_{\RRR}
e^{-i(u\xi+r\eta)}E_0f(u,r)\,du\,dr .
\]
For each fixed angular variable, the computation in the proof
of Proposition \ref{prop:bilinear-outgoing-general} gives
\[
\widehat{\Wk[T^{1/2}f]}(\xi)
=|\xi|^{1/2}e^{-i\xi XR}e^{-i\xi f(\rt)}\widehat{E_0f}(\xi,-\xi X),
\]
Hence, by Plancherel and the change of variables $\eta=-\xi X$,
\begin{align*}
\int_{I_\sigma}\int_{\RRR}
\left|\Wk^X[T^{1/2}f]\right|^2\,d\taut\,dX
&\lesssim
\int_{\RRR}\int_{I_\sigma}
|\xi|\,|\widehat{E_0f}(\xi,-\xi X)|^2\,dX\,d\xi\\
&\le
\int_{\RRR}\int_{\RRR}|\widehat{E_0f}(\xi,\eta)|^2\,d\eta\,d\xi\\
&\lesssim \int_{\RRR}\int_{r_+}^{\infty}|f(u,r)|^2\,dr\,du .
\end{align*}
This proves the first estimate.

For the second estimate, apply \eqref{eq:Rk-commutation} to $f$.  Since $f$
is supported away from $r=r_+$, the horizon trace term vanishes.  We therefore obtain
\[
B(\rt)\partial_\ut\Wk[T^{1/2}F]
=-\Wk[T^{1/2}\partial_rF].
\]
where $\pr_r$ is the radial coordinate derivative in the outgoing coordinates, that is, $\pr_r=e_4^{(out)}$. The already proved $L^2$ estimate, applied to $f=\partial_rF=e_4^{(out)}F$ gives the second estimate.
\end{proof}


\section{The  $X<-2$ branch}
\lab{sec:Xneg}
In this section we   analyze the  equation $\square_\gt \psit=\Nt $  in the domain  
$
\DDt
:= \{X<-2\} \subset \DDt_-.
$


\subsection{\(\T\)-energy identity}
\lab{sec:energyXleq-2}
As in Remark \ref{rem:conformalfactor}, we choose the conformal factor as 
\[
|\qt|^2=\sqrt{\sin^2\tht |\det\Hbt|}=\frac{|X|}{|S|}\sqrt{|P(X,\tht)|}\sim \rt^2\]
in the region $\DDt=\{X<-2\}$.  By Proposition \ref{Prop:outgoing-Snegative-regions}, we know that the Killing vectorfield
$\T$ is timelike throughout $\DDt=\{X<-2\}$ and the boundary $X=-2$
is its  bifurcate event horizon. By Proposition \ref{prop:nullinfinities}
, $\gt$ is asymptotically flat and has a complete smooth future null infinity
\[
\widetilde{\II}^+=\{\ut, X=-\infty, \tht, \vphit_-\}.
\]
Then
$(\DDt,\gt)$ can be regarded as the exterior of a stationary black
hole.   

For a complex
scalar field \(\phi\), set
\beq\lab{eq:Whiting-EF-norms}
\QQ_{\at\bt}[\psit]
:=\pr_{\at}\psit\,\pr_{\bt}\overline\psit
+\pr_{\bt}\psit\,\pr_{\at}\overline\psit
-\gt_{\at\bt}\gt^{\widetilde{\mu}\widetilde{\nu}}\pr_{\widetilde{\mu}}\psit\,\pr_{\widetilde{\nu}}\overline\psit .
\eeq
Since $\T$ is Killing,
\beq\lab{eq:negative-null-flux-current-divergence}
\D_{\at}\bigl(\QQ[\psit]^{\at}{}_{\bt}\T^{\bt}\bigr)
=2\Re\bigl((\T\psit)\overline{\square_\gt\psit}\bigr).
\eeq
We define
\begin{align*}
&\DDt(\taut_1, \taut_2):=\DDt\cap \{\taut_1\leq \taut\leq \taut_2\}\\
&\HHt^+(\taut_1, \taut_2):=\HHt^+\cap \{\taut_1\leq \taut\leq \taut_2\}\\
&\IIt^+(\taut_1, \taut_2):=\IIt^+\cap \{\taut_1\leq \taut\leq \taut_2\}\\
&\Sit(s):=\{\taut=s\}\cap\DDt
\end{align*}
where $\taut$ is a smooth time function\footnote{ Near the future event
horizon we set $\taut=\vt$, and smoothly interpolate between $\vt$ and
$\ut-m^2/\rt$ in a fixed intermediate $\rt$-region so that the level
hypersurfaces remain spacelike} transversal to both $\HHt^+$ and $\IIt^+$.

\begin{proposition}[$\T$-energy identity]
\lab{prop:energy-gt}
Let $\psit$ solve $|\qt|^2\square_\gt\psit=|\qt|^2\Nt$ in a region
$\DDt(\taut_1,\taut_2)\subset\DDt$ bounded by spacelike hypersurfaces
$\Sit(\taut_1),\Sit(\taut_2)$, future event horizon
$\HHt^+(\taut_1,\taut_2)$ and future null infinity $\IIt^+(\taut_1, \taut_2)$.  Then
\beq\lab{eq:T-energy-Xlessminus2}
\bsplit
&\int_{\Sit(\taut_2)}\QQ[\psit](\T,\nu)\,d\mu_{\Sit(\taut_2)}
+\int_{\HHt^+(\taut_1,\taut_2)}\QQ[\psit](\T,\nu_{\HHt^+})\,d\mu_{\HHt^+}+\int_{\IIt^+(\taut_1,\taut_2)}\QQ[\psit](\T,\nu_{\IIt^+})\,d\mu_{\IIt^+}\\
&\qquad
=\int_{\Sit_{\taut_1}}\QQ[\psit](\T,\nu)\,d\mu_{\Sit(\taut_1)}
+2\Re\int_{\DDt(\taut_1,\taut_2)}
(\T\psit)\c\overline{\Nt}\,dv_\gt .
\end{split}
\eeq
where $\nu_{\HHt^+}$ and $\nu_{\IIt^+}$ denote the null generators of $\HHt^+$ and $\IIt^+$. All four boundary terms are nonnegative.
\end{proposition}

\begin{proof}
Integrate \eqref{eq:negative-null-flux-current-divergence} over
$\DDt(\taut_1,\taut_2)$ and apply Divergence theorem.  The boundary terms are
nonnegative because $\T$ is timelike in $\DDt$ and null at 
$\HHt^+$.
\end{proof}

\begin{remark}
Since the null generator $\nu_{\IIt^+}$  behaves like $\nu_{\IIt^+}\sim \widetilde{e}_3$, we have $\QQ[\psit](\T,\nu_{\IIt^+})\sim |\widetilde{e}_3\psit|^2+|\widetilde{\nab}\psit|^2$. However, we are aiming for the control $|T\psit|^2$ at $\IIt^+$ while $\widetilde{e}_3$ differs from $T$ by $\widetilde{e}_4$. If we were able to prove that $\widetilde{e}_4\psit$ decays faster and thus is negligible at $\IIt^+$, then we are done. One way to prove that $\widetilde{e}_4\psit$ has faster decay is using $r^p$-weighted estimate. However, to close the $r^p$-weighted estimate, one needs the Morawetz estimate\footnote{We delay such estimates  in our paper \cite{He-K3}.} for $\square_{\gt}\psit=\Nt$. Instead, we use the conformal energy estimate to obtain the control of $|T\psit|^2$ at $\IIt^+$.
\end{remark}
\subsection{Conformal energy estimate}
\lab{sec:conformal-energy-Xlessminus2}

We now use the Penrose compactification of the asymptotic end in $\DDt$.  Set
\[
\rhot:=|X+1|^{-1},\qquad
\ghat:=|X+1|^{-2}\gt,\qquad
\psihat:=|X+1|\psit ,\qquad \widehat{\DDt}:=\DDt\cup\{\rhot=0\}.
\]
By Lemma \ref{lem:appendix-conformal-wave-equation},
\beq\lab{eq:conformal-wave-equation-Xminus}
\square_{\ghat}\psihat+V\psihat=|X+1|^3\Nt,
\qquad
V<0
\quad\hbox{in }\widehat{\DDt},
\eeq
The compactified metric extends smoothly to
\(\IIt^+=\{\rhot=0\}\).  The radiation field at $\IIt^+$ is defined as follows
\beq\lab{eq:null-flux-radiation-field}
\psit_\infty
:=\left.\psihat\right|_{\IIt^+}
=\lim_{X\to-\infty}|X+1|\psit .
\eeq



\begin{proposition}[Conformal Energy]
\lab{Prop:conf-en(sk-0)}
Let $\psi$ be an acceptable solution of $\square_{\g}\psi=N$. Then the following estimate holds true for $\psit$.
\beq
\lab{eq:confenergy-psit}
\int_{\HHt^+}|\T^2 T^{1/2}\psit|^2\,d\vt\,d\omega
+\int_{\IIt^+}|\T^2 T^{1/2}\psit_\infty|^2\,d\ut\,d\omega
\lesssim
\left|\int_{\DDt}
\Re{\big((\T^2T^{1/2}\psit)\c\overline{TT^{1/2}\Nt}\big)}\,dv_\gt\right|.
\eeq
\end{proposition}

\begin{proof}
For the conformal equation \eqref{eq:conformal-wave-equation-Xminus}, use the
potential stress-energy tensor
\[
\widehat{\QQ}^{\,V}_{\at\bt}[\psihat]
:=\widehat{\QQ}_{\at\bt}[\psihat]
+V|\psihat|^2\ghat_{\at\bt}.
\]
Since $V$ and $\ghat$ are stationary, the current $
\widehat J^{\at}_{\T}
:=(\widehat{\QQ}^{\,V})^{\at}{}_{\bt}\T^{\bt}$ satisfies
\beq\lab{eq:conformal-T-current-divergence}
\widehat\D_{\at}\widehat J^{\at}_{\T}
=2\Re\bigl((\T\psihat)\c\overline{|X+1|^3\Nt}\bigr).
\eeq
The sign $V<0$, together with the causality of $\T$, makes the energy
density of this current nonnegative on spacelike future boundaries.

For $\taut_1<\taut_2$, let
\[
\widehat{\DDt}(\taut_1,\taut_2)
:=\{\taut_1\leq\taut\leq\taut_2\}\cap \widehat{\DDt}\]
The boundary of this
compactified domain consists of
$\Sit_{\taut_1}, \Sit_{\taut_2}$, and the
future event horizon $\HHt^+(\taut_1,\taut_2)$ and future null infinity
$\IIt^+(\taut_1,\taut_2)$.  Integrating
\eqref{eq:conformal-T-current-divergence} over this domain gives
\begin{align*}
&\int_{\widetilde\Sigma_{\taut_2}}
\widehat{\QQ}^{\,V}[\psihat](\T,\nu)\,d\mu_{\widetilde\Sigma_{\taut_2}}
+\int_{\HHt^+(\taut_1,\taut_2)}
\widehat{\QQ}^{\,V}[\psihat](\T,\nu_{\HHt^+})\,d\mu_{\HHt^+}\\
&\qquad
+\int_{\IIt^+(\taut_1,\taut_2)}
\widehat{\QQ}^{\,V}[\psihat](\T,\nu_{\IIt^+})\,d\mu_{\IIt^+}
=\int_{\widetilde\Sigma_{\taut_1}}
\widehat{\QQ}^{\,V}[\psihat](\T,\nu)\,d\mu_{\widetilde\Sigma_{\taut_1}}+2\Re\int_{\widehat{\DDt}(\taut_1,\taut_2)}
(\T\psihat)\overline{|X+1|^3\Nt}\,dv_{\ghat}.
\end{align*}
On $\HHt^+$, the potential term has no flux because $\T$ is null. Therefore, we have
\[
\widehat{\QQ}^{\,V}[\psihat](\T, \nu_{\HHt^+})\,d\mu_{\HHt^+}
\sim |\T\psit|^2\,d\vt\,d\omega .
\]
On $\IIt^+$, the null generator of $\IIt^+$ is $
\widehat{\D}\rhot\sim\pr_\ut=\T$. Then we have
\[
\widehat{\QQ}^{\,V}[\psihat](\T,\nu_{\IIt^+})
\sim |\pr_\ut\psihat|^2.
\]
Moreover, we note that $d\mu_{\IIt^+}\simeq d\ut\,d\omega$ and $\psihat|_{\IIt^+}=\psit_\infty$, and thus we obtain
\[
\widehat{\QQ}^{\,V}[\psihat](\T,\nu_{\IIt^+})\,d\mu_{\IIt^+}
\sim |\T\psit_\infty|^2\,d\ut\,d\omega .
\]
Since $\T\psihat=|X+1|\,\T\psit$ and $dv_{\ghat}=|X+1|^{-4}dv_\gt$, it follows that 
\[
\int_{\widehat{\DDt}}
(\T\psihat)\overline{|X+1|^3\Nt}\,dv_{\ghat}
=\int_{\DDt}
(\T\psit)\overline{\Nt}\,dv_\gt .
\]
All the above analysis also apply to $TT^{1/2}\psit$.

Finally, by Lemma \ref{lem:acceptablesol}, the energy fluxes at $\taut_1, \taut_2$ vanish as $\taut_1\to-\infty$ and $\taut_2\to\infty$.\footnote{     The statement here should be  interpreted in terms of  sequences  of $\taut_1, \taut_2$.  Note also that  the flux through
$\widetilde\Sigma_{\taut_2}$ is nonnegative.} Therefore, letting
$\taut_1\to-\infty$ and $\taut_2\to\infty$ gives \eqref{eq:confenergy-psit}.
\end{proof}


\subsection{Control of horizon flux of $\psi$}
\lab{sec:Fourier-trace-horizon-flux}

The goal of this section is to prove a relation relating the flux through the future event
horizon of a function $\psi$, defined on the domain of outer communication of Kerr spacetimes, to the
flux at future null infinity of the metric $\gt$ of $\psit$  in $ \DDt$. 

For a function $f(\ut, \rt, \tht, \vphit_-)$ on $\DDt$, we decompose it in Fourier series as follows:
\[
f(\ut,\rt,\tht,\vphit_-)=\sum_n e^{in\vphit_-}f_n(\ut,\rt,\tht),
\qquad
f_n(\ut,\rt,\tht)
=\frac1{2\pi}\int_0^{2\pi}
f(\ut,\rt,\tht,\vphit_-)e^{-in\vphit_-}\,d\vphit_- .
\]
Denote the Fourier transform of $f_n$ in $\ut$ by
\[
\FF_{\ut}[f_n](\xi,\rt, \tht)
:=\int_{\RRR}e^{-i\ut\xi}f_n(\ut,\rt, \tht)\,d\ut.
\]
Similarly, for a function $f(v, r, \th, \vphi_+)$ on Kerr space time, we have the Fourier series decomposition
\[
f(v ,r, \th, \vphi_+)=\sum_n e^{in\vphi_+}f_n(v, r,\th),
\qquad
f_n(v,r, \th)
=\frac1{2\pi}\int_0^{2\pi}
f(v,r,\th, \vphi_+)e^{-in\vphi_+}\,d\vphi_+
\]
and define the Fourier transform of $f_n(v, r, \th)$ in $v$
\[
\FF_{v}[f_n](\xi, r, \th)=\int_{\RRR}e^{-i\xi v}f_n(v, r,\th)\, dv.
\]
For each Fourier mode $n$ and each time frequency $\xi$, we define with $\la_+=\frac{4mr_+}{r_+-r_-}, \la_0=\frac{2a}{r_+-r_-}$
\beq
\lab{eq:defofbeta_+}
\beta_+:=\la_+\xi+\la_0 n=\la_+(\xi+\omega_+n),\qquad \omega_+=\frac{a}{2mr_+}.
\eeq
Recall the  definition of  the  fractional multipliers, see \eqref{eq:fractionalT},
\[
\FF[T^{1/2}f_n]
=|\xi|^{1/2}\FF[f_n],
\]
We also define
\beq
\FF[T_+^{1/2}f_n]
=|\xi+\omega_+n|^{1/2}\FF[f_n].
\eeq
where $T_+=T+\omega_+Z$ is the Hawking vectorfield.

For $M>0$, define the controlled frequency region by
\beq
\lab{eq:controlled-frequency-region}
\mathcal A_M
:=\{\beta_+\xi\ge0\}\cup\{|\beta_+|\le M\}.
\eeq
For each $n\in\ZZZ$, denote its $\xi$-slice by
\[
\mathcal A_M(n)
:=\{\xi\in\RRR:(\xi,n)\in\mathcal A_M\}.
\]
Here $M$ denotes the size of the bounded frequency window in the variable
$\beta_+$.  We denote by $\Pi_{\mathcal A_M}$ the corresponding frequency
projection.

\begin{proposition}
\lab{prop:psi-psit-flux-strong-Xminus}
Let $\psi$ be an acceptable solution to $\square_{\g}\psi=N$. 
 Let\footnote{Assume that the radiation field $\psit$ has  enough regularity.}
$\psit=\Wk[\psi]$.  The following estimate holds true.
\beq
\lab{eq:psi-psit-flux-strong-Xminus}
\int_{\HH^+}
\big|\Pi_{\mathcal A_M}\T_+^{1/2}\psi\big|^2+\big|\Pi_{\mathcal A_M}\psi\big|^2
\;dv\,d\omega
\lesssim_M
\int_{\IIt^+}\big|\Pi_{\mathcal A_M}\T\psit_\infty\big|^2\,d\ut\,d\omega .
\eeq
The conclusion remains true after applying  $T, \Z$ of  appropriate  $1/2$ versions of them.
\end{proposition}


\begin{proof}
We recall that 
\[
\psit(\vt,\rt,\vphit)
=\int_{r_+}^\infty
\psi\big(\vt+B(\rt)(r-R),r,\vphit\big)\,dr .
\]
Here $\psi$ is initially written in the outgoing coordinates
$(u,r,\vphi=\phi_-)$, with $\vphit$ identified with $\vphi$, and we
suppress the coordinates $\th,\tht$.  Since the
target flux is on the future event horizon $\HH^+$, we pass to the ingoing
coordinates
\[
v=u+2r_*(r),\qquad \phip=\phi_-+2\phi_*(r).
\]
Thus
\[
\psit(\vt,\rt,\vphit)
=\int_{r_+}^\infty
\psi\big(\vt-B_+(\rt)(r-R)+2r_*(r),r,
\vphit+2\phi_*(r)\big)\,dr .
\]
Next we pass to the  outgoing coordinates
$(\ut,\rt,\tht,\vphit_-)$ of \eqref{eq:BL-Whiting-outXneg}.  We recall that 
\[
\vt=\ut+f(\rt),\qquad \vphit=\vphit_-+g(\rt).
\]
Then
\[
\psit(\ut,\rt,\vphit_-)
=\int_{r_+}^\infty
\psi\big(\ut+B(\rt)(r-R)+f(\rt)+2r_*(r),r,
\vphit_-+g(\rt)+2\phi_*(r)\big)\,dr .
\]
We decompose $\psit(\ut,\rt,\vphit_-)$ in Fourier series in $\vphit_-$ and the corresponding $n$-mode is given by
\begin{align*}
\psit_n(\ut,\rt)
&=\frac1{2\pi}\int_0^{2\pi}
\psit(\ut,\rt,\vphit_-)e^{-in\vphit_-}\,d\vphit_-\\
&=e^{ing(\rt)}
\int_{r_+}^\infty e^{2in\phi_*(r)}
\psi_n\big(\ut+B(\rt)(r-R)+2r_*(r)+f(\rt),r\big)\,dr
\end{align*}
Taking the Fourier transform of $\psit_n$ in $\ut$, we derive
\begin{align*}
\FF_{\ut}[\psit_n](\xi,\rt)
&=\int_{-\infty}^\infty e^{-i\ut\xi}\psit_n(\ut,\rt)\,d\ut\\
&=e^{ing(\rt)}
\int_{r_+}^\infty e^{2in\phi_*(r)}\,dr
\int_{-\infty}^\infty e^{-i\ut\xi}
\psi_n\big(\ut+B(\rt)(r-R)+2r_*(r)+f(\rt),r\big)\,d\ut\\
&=e^{ing(\rt)}
\int_{r_+}^\infty e^{2in\phi_*(r)}
e^{-i\xi\big(-B(\rt)(r-R)-2r_*(r)-f(\rt)\big)}
\FF_v[\psi_n](\xi,r)\,dr\\
&=e^{ing(\rt)}e^{i\xi f(\rt)}e^{-i\xi B(\rt)R}
\int_{r_+}^\infty e^{2in\phi_*(r)}
e^{-i\xi\big(-B(\rt)r-2r_*(r)\big)}
\FF_v[\psi_n](\xi,r)\,dr
\\
&=e^{ing(\rt)}e^{i\xi\big(f(\rt)-B(\rt)R\big)}I_n(\xi,\rt),
\end{align*}
where 
\beq
\lab{eq:formular-Fourier-n-Xminus}
\bsplit
I_n(\xi,\rt)&= \int_{r_+}^\infty e^{-if_n(r,\rt,\xi)}
\FF_v[\psi_n](\xi,r)\,dr,\\
f_n(r,\rt,\xi)&=\big(-B(\rt)r-2r_*(r)\big)\xi-2\phi_*(r)n .
\end{split}
\eeq
Note that 
\begin{align*}
\frac{r^2+a^2}{\De}&=1+\frac{2mr}{(r- r_-) (r- r_+) }=1+\frac{2m}{r_+- r_-} \Big(\frac{r_+}{r- r_+}- \frac{r_-}{ r- r_-} \Big)\\
\frac{a}{\De}&= \frac{a}{r_+- r_-} \Big(\frac{1}{r- r_+}- \frac{1}{ r- r_-} \Big)
\end{align*}
Hence, modulo a constant
\beq
\lab{eq:rstar-phistar-canonical-Xminus}
\bsplit
r_*(r)&= r +\frac 12 \la_+ \ln|r-r_+| - \frac 1 2  \la_- \ln|r-r_-|,  \qquad \la_\pm=\frac{4mr_\pm}{r_+- r_-}\\
\phi_*(r)&=\frac 1 2 \la_0\big(\ln |r-r_+|-  \ln |r- r_-| \big), \qquad \qquad \,\,\la_0= \frac{2a}{r_+- r_-}.
\end{split}
\eeq
and with 
\beq
\lab{eq:C-rt-Xminus}
C(\rt)=-B(\rt)-2=-X-2,
\eeq
we write
\begin{align*}
 f_n(r,\rt,\xi) &= \big(-B(\rt)r-2r_*(r)\big)\xi-2\phi_*(r)n\\
 &= \big( C(\rt)r - \la_+ \ln|r-r_+|+ \la_-\ln|r-r_-|\big)\xi-\big(\la_0 \ln |r-r_+|-  \la_0  \ln |r- r_-| \big)n.
\end{align*}
Thus, changing variables yields,
\begin{align*}
I_n(\xi, \rt)&=\int_{r_+}^\infty e^{-i \xi  r C(\rt)} \big(r-r_+\big) ^{i \big(\la_+\xi+\la_0 n\big)}  \big(r-r_-\big) ^{-i\big(\la_-\xi+\la_0 n\big)} \FF_v[\psi_n](\xi,  r) dr
\\
 &= e^{-i \xi r_+ C(\rt)}  \int_{0}^\infty e^{-i \xi  r  C(\rt)}r^ {i\big( \la_+\xi +\la_0 n \big)} \big(r+ r_+-r_-\big) ^{-i \big(\la_-\xi+\la_0 n\big)} \FF_v[\psi_n](\xi,  r+ r_+) dr.
\end{align*}
We rewrite this as
\begin{align*}
I_n(\xi, \rt)&= e^{-i \xi r_+ C(\rt)}   J_n(\xi, \rt)\\
 J_n(\xi, \rt)&= \int_{0}^\infty e^{-i \xi  r  C(\rt)}r^ {i  \beta_+}
 \big(r+ r_+-r_-\big) ^{-i \beta_-} \FF_v[\psi_n](\xi,  r+ r_+) dr.
\end{align*}
where $\beta_\pm=\la_\pm\xi+\la_0n$.
We deduce 
\beq
\lab{eq:expressionI-Xminus}
\bsplit
\FF[\psit_n](\xi, \rt)&= e^{in g(\rt)}e^{i\xi\big(f(\rt)-B(\rt)R-r_+C(\rt)\big)}J_n(\xi,\rt) \\
J_n(\xi, \rt)&=  \int_{0}^\infty e^{-i \xi  r  C(\rt)}r^ {i  \beta_+ }
	 \big(r+ r_+-r_-\big) ^{-i \beta_-} \FF_v[\psi_n](\xi,  r+ r_+) dr.
\end{split}
\eeq
\begin{remark}
Note that, setting   $\xi=-\om $  and $n=\ell_z$,
\[
  \beta_+\xi=\la_+ \xi^2 +\la_0n\xi  =\frac{4  m r_+}{\big(r_+- r_-\big)}\om^2 \big(1-\frac{a}{2mr_+} \frac{\ell_z}{\om}  \big)
\]
Recall\footnote{See, for example, Definition 3.14 in \cite{He-K1}.} that the
frequencies $(\ell_z,\om)$ are called superradiant if
$\frac{a\ell_z}{2mr_+\om}>1$.  Thus $\beta_+(\xi, n)\xi\leq 0$ corresponds precisely to the
superradiant frequencies.
\end{remark} 

We now appeal to the following well-known lemma.\footnote{See Appendix
\ref{section:Erdelyi} for a short proof.}

\begin{lemma} 
\lab{Le:Erdelyi-Xminus}
Consider the integral  $J(\nu)= \int_{0}^\infty e^{- i \nu r}  r ^{ib} h(r) dr $, with $h$   a sufficiently smooth function.
\begin{enumerate}
\item
 As $\nu \to \infty$ we have 
\[
J(\nu)=  h(0)  \frac{\Ga(1+ib)} { (\nu)^{1+ib} } e^{-i\pi/2(1+ib) } +O(\nu^{-2} ).
\]
\item
As $\nu \to -\infty$ we have
\[
J(\nu)= h(0)  \frac{\Ga(1+ib)} { (-\nu)^{1+ib} } e^{i\pi/2(1+ib) } +O(\nu^{-2} ).
\]
\end{enumerate}
\end{lemma}
We write the integral
\[
J_n(\xi,\rt)=\int_0^\infty e^{-i\xi rC(\rt)}
r^{i\beta_+}\big(r+r_+-r_-\big)^{-i\beta_-}
\FF_v[\psi_n](\xi,r+r_+)\,dr
\]
in \eqref{eq:expressionI-Xminus} in the form
\[
J_n=\int_0^\infty e^{-ir\nu}r^{i\beta_+}h(r)\,dr,
\qquad
\nu=C(\rt)\xi,
\]
where
\[
h(r)=\big(r+r_+-r_-\big)^{-i\beta_-}\FF_v[\psi_n](\xi,r+r_+),
\qquad
\beta_\pm=\la_\pm\xi+\la_0n .
\]
Since $C(\rt)=-B(\rt)-2=-X-2>0$ in $\DDt$, we have
$\nu>0$ for $\xi>0$ and $\nu<0$ for $\xi<0$. For  fixed  $\xi>0$,  as $\rt \to \infty$  we have   $\nu =C(\rt) \xi\to \infty$, in view of the first part of Lemma \ref{Le:Erdelyi-Xminus},
\begin{align*}
J_n(\xi,  \rt)  &= \big( r_+-r_-\big) ^{-i \beta_-}\FF_v[\psi_n](\xi,   r_+)  \frac{\Ga(1+i\beta_+)} { (\nu)^{1+i\beta_+} } e^{-i\pi/2(1+i\beta_+) } + O(\nu^{-2})\\
&= \big( r_+-r_-\big) ^{-i \beta_-}\FF_v[\psi_n](\xi,   r_+) \frac{\Ga(1+i\beta_+)} { (C(\rt) \xi)^{1+i \beta_+ } } e^{-i\pi/2(1+i\beta_+)}+ O( \rt^{-2} |\xi|^{-2} )
\end{align*}
Thus, for fixed $\xi>0$, recalling the classical formula
$\lvert\Ga(1+i\beta_+)\rvert^2=\pi \beta_+/\sinh(\pi \beta_+)$, we obtain
 \begin{align*}
    \lim_{\rt\to\infty}|C(\rt)|^2 |\xi|^2 |J_n(\xi, \rt) |^2 &= |\FF_v[\psi_n](\xi, r_+)|^2\big| \Ga(1+i\b_+)\big|^2  e^{\pi  \b_+ }\\
  &=| \FF_v[\psi_n](\xi,   r_+)|^2  \frac{\pi \beta_+ }{\sinh \pi   \beta_+} e^{\pi  \beta_+} .
\end{align*}
Since $C(\rt)=-B(\rt)-2=-S(\rt-\Rt)-2$, we have for $\xi>0$
\beq
 \lab{eq:Flux-interm-Xminusxipos}
 \lim_{\rt\to\infty}\rt^2|\xi|^2 |J_n(\xi,\rt)|^2
=S^{-2}|\FF_v[\psi_n](\xi,r_+)|^2
\frac{\pi |\beta_+|}{\sinh\pi |\beta_+|}e^{\pi \beta_+}.
\eeq
Similarly,  applying the second part of Lemma \ref{Le:Erdelyi-Xminus},  for $\xi <0$ and  $\rt \to +\infty$,
\beq
 \lab{eq:Flux-interm-Xminusxineg}
 \lim_{\rt\to\infty}\rt^2|\xi|^2 |J_n(\xi,\rt)|^2
=S^{-2}|\FF_v[\psi_n](\xi,r_+)|^2
\frac{\pi |\beta_+|}{\sinh\pi |\beta_+|}e^{-\pi \beta_+}.
\eeq
\begin{itemize}
\item{\bf The case of non-superradiant frequencies.}
In the non-superradiant region $\beta_+\xi\ge0$, we rewrite 
	 \beq
	 \lab{eq:Flux-interm-Xminus}
\lim_{\rt\to\infty}\rt^2|\xi|^2 |J_n(\xi,\rt)|^2
=S^{-2}|\FF_v[\psi_n](\xi,r_+)|^2
\frac{\pi |\beta_+|}{\sinh\pi |\beta_+|}e^{\pi |\beta_+|}.
\eeq
and thus\footnote{Because
\(\frac{x}{\sinh x}e^x\ge 1+x\), for \(x\ge0\).}
\beq
\lab{eq:Flux-b_+>0-Xminus}
\big(1+\pi |\beta_+|\big)|\FF_v[\psi_n](\xi,r_+)|^2
\le S^2\lim_{\rt\to\infty}\rt^2|\xi|^2|J_n(\xi,\rt)|^2 .
	 \eeq
	 
	 \item{\bf The case of bounded frequencies.}
In the bounded frequency region $|\beta_+|\le M$, we obtain
 \[
\lim_{\rt\to\infty}\rt^2|\xi|^2 |J_n(\xi,\rt)|^2
\geq S^{-2}|\FF_v[\psi_n](\xi,r_+)|^2
\frac{\pi |\beta_+|}{\sinh\pi |\beta_+|}e^{-\pi |\beta_+|}.
\]
and thus
\beq
\lab{eq:low-fr-Xminus}
\bsplit
(1+\pi|\beta_+|)\,|\FF_v[\psi_n](\xi,r_+)|^2
&\le S^2e^{2\pi|\beta_+|}\lim_{\rt\to\infty}\rt^2|\xi|^2 |J_n(\xi,\rt)|^2\\
&\le S^2e^{2\pi M}\lim_{\rt\to\infty}\rt^2|\xi|^2 |J_n(\xi,\rt)|^2.
\end{split}
\eeq
\end{itemize}

Recall the controlled frequency region $\mathcal A_M$ defined in
\eqref{eq:controlled-frequency-region}.
According to \eqref{eq:Flux-b_+>0-Xminus} and \eqref{eq:low-fr-Xminus},
we obtain, throughout $\mathcal A_M$,
\beq\lab{eq:controlled-frequency-horizon-trace}
(1+|\beta_+|)\,|\FF_v[\psi_n](\xi,r_+)|^2
\le C_M
\lim_{\rt\to\infty}\rt^2|\xi|^2|J_n(\xi,\rt)|^2 .
\eeq
Returning to \eqref{eq:expressionI-Xminus}, the prefactor relating
\(\FF[\psit_n](\xi,\rt)\) and \(J_n(\xi,\rt)\) has modulus one. Therefore
\beq\lab{eq:controlled-frequency-trace-with-T}
(1+|\beta_+|)|\FF_v[\psi_n](\xi,r_+)|^2
\le C_M
\lim_{\rt\to\infty}\rt^2|\xi|^2|\FF[\psit_n](\xi,\rt)|^2 .
\eeq
We now justify the passage from this fixed-frequency $\rt$-limit, in \eqref{eq:controlled-frequency-trace-with-T},  to the  desired integrated
estimate.  Integrating over the slices
$\mathcal A_M(n)$, summing over $n$ and using Fatou's lemma gives
\begin{align*}
&\sum_n\int_{\mathcal A_M(n)}
(1+|\beta_+|)|\FF_v[\psi_n](\xi,r_+)|^2\,d\xi\\
&\qquad\leq C_M\liminf_{\rt\to\infty}
\sum_n\int_{\mathcal A_M(n)}
\rt^2|\xi|^2|\FF[\psit_n](\xi,\rt)|^2\,d\xi\\
&\qquad\leq C_M\liminf_{\rt\to\infty}
\sum_n\int_{\mathcal A_M(n)}
\rt^2|\xi|^2|\FF[\psit_n](\xi,\rt)|^2\,d\xi .
\end{align*}
Applying Plancherel
gives
\[
\lim_{\rt\to\infty}
\sum_n\int_{\mathcal A_M(n)}
\rt^2|\xi|^2|\FF[\psit_n](\xi,\rt)|^2\,d\xi
\simeq
\int_{\IIt^+}|\T\psit_\infty|^2\,d\ut\,d\omega .
\]
Consequently,
\beq\lab{eq:controlled-frequency-HH-flux}
\int_{\HH^+}
\big|\Pi_{\mathcal A_M}\T_+^{1/2}\psi\big|^2+\big|\Pi_{\mathcal A_M}\psi\big|^2
\;dv\,d\omega
\lesssim_M
\int_{\IIt^+}\big|\Pi_{\mathcal A_M}\T\psit_\infty\big|^2\,d\ut\,d\omega .
\eeq
This proves \eqref{eq:psi-psit-flux-strong-Xminus}.  The complementary region 
\beq
\lab{define:BB_M}
\BB_M=\{\beta_+\xi<0\}\cap \{|\beta_+|>M\}
\eeq
 is the large superradiant region  not covered by the argument above.
\end{proof}


\section{The  $X>0$ branch}
\lab{sec:Xpos}

We now turn to the positive branch $\DDt_+:=\{X>0\}$.  Recall that the rescaled metric $|\qt|^2\gt$ is Lorentzian in the region $X>X_+(\tht)$ and
Riemannian in $0<X<X_+(\tht)$.  Moreover, $\T$ is timelike in the
Lorentzian region $X>X_+(\tht)$. We now treat the large superradiant frequency region
$\mathcal B_M=\{\beta_+\xi\leq0,\ |\beta_+|>M\}$ by woking in the positive
 branch.  There are two specific difficulties which do not occur
in the subdomain $\DDt$ of the negative branch $\DDt_-$. 
\begin{enumerate}
\item  First,  when  we  apply    the $T$-energy estimate to $|\qt|^2\square_{\gt}\psit=|\qt|^2\Nt$  we generate a boundary  term  at  $X=0$, rather than the  $X=-2$ null boundary   in the previous case. Remarkably, because of the form of the metric $\gt$,   this boundary term    actually  vanishes. 
  
\item The above procedure   leads to difficulties, near $X=0$,  due  to the presence of the nontrivial inhomogenuous term $|\qt|^2 \Nt$ generated by the  $\chi(\tauh)$ cut-off in the  original  wave equation. 
To fix this issue, we  replace $\T$ with a  multiplier
of the form $f(X)\T$, with $f(X)=0$ near $X=0$ and $f(X)=1$ when $X$ is large.

\item The multiplier $f(X)\T$  produces an additional bulk term.  In order to control this bulk term, we turn to the localized elliptic estimate in
a small   annular  region near $X=0$. 
\end{enumerate}

\subsection{Modified energy estimate}
We implement the $f(X)\T$-multiplier strategy mentioned above.  Fix
$\de>0$ sufficiently small and choose a smooth non-negative function
$f(X)$ such that
\beq
\lab{eq:defoff(X)}
f(X)=0\quad\hbox{for }0<X<2\de,\qquad
f(X)=1\quad\hbox{for }X>3\de,
\eeq
with $0\leq f\leq1$ and then
\[
|f'(X)|\leq C\de^{-1}\mathbf 1_{\{2\de<X<3\de\}}.
\]
In particular,
\beq\lab{eq:fX-comparison}
f(X)\leq \de^{-1}X=\de^{-1}B(\rt)
\qquad\hbox{on }\DDt_+.
\eeq
By abuse of notation, we denote
\begin{align*}
&\DDt_+(\taut_1, \taut_2):=\DDt_+\cap \{\taut_1\leq \taut\leq \taut_2\}\\
&\IIt:=\{\ut, X=\infty, \tht, \vphit_-\}\\
&\IIt^+(\taut_1, \taut_2):=\IIt^+\cap \{\taut_1\leq \taut\leq \taut_2\}\\
&\Sit(s):=\{\taut=s\}\cap\DDt_+
\end{align*}
where $\taut$ is a smooth time function\footnote{ Near $X=0$ we set $\taut=\tt$, and smoothly interpolate between $\tt$ and
$\ut-m^2/\rt$ in a fixed intermediate $X$-region so that the level
hypersurfaces remain spacelike for large $X$.} transversal to $\IIt^+$.

We now use the Penrose compactification of the asymptotic end in $\DDt_+$.  Set
\[
\rhot:=(X+1)^{-1},\qquad
\ghat:=(X+1)^{-2}\gt,\qquad
\psihat:=(X+1)\psit ,\qquad \widehat{\DDt_+}:=\DDt_+\cup\{\rhot=0\}.
\]
By Lemma \ref{lem:appendix-conformal-wave-equation},
\beq\lab{eq:conformal-wave-equation-Xplus}
|\qt|^2\square_{\gt}\psihat+V\psihat=(X+1)|\qt|^2\Nt,
\qquad
V<0
\quad\hbox{in }\widehat{\DDt_+} 
\eeq
where
\beq
|\qt|^2\square_{\gt}=\pr_{\a'}\hbt^{\a'\b'}\pr_{\b'}+\frac{1}{\sin\tht}\pr_{\a'}\sin\tht O^{\a'\b'}\pr_{\b'}.
\eeq
is expressed in the compactified outgoing coordinates $(\ut, \rhot=(	X+1)^{-1},\tht, \vphit_-)$.

The compactified metric extends smoothly to
$\IIt^+=\{\rhot=0\}$.  The radiation field at $\IIt^+$ is defined as follows
\beq\lab{eq:null-flux-radiation-field-pos}
\psit_\infty
:=\left.\psihat\right|_{\IIt^+}
=\lim_{X\to\infty}(X+1)\psit .
\eeq

\begin{proposition}[Weighted Conformal Energy]
\lab{Prop:conf-en(sk-0)Xpos}
Let $\psi$ be an acceptable solution of $\square_{\g}\psi=N$. Then the following estimate holds true for  $\psit$.
\beq
\lab{eq:confenergy-psitXpos}
\bsplit
&\int_{\IIt^+}|\T^2T^{1/2}\psit_\infty|^2\,d\ut\,d\omega
\lesssim
\left|\int_{\DDt_+}
\Re{\big((f(X)\T^2T^{1/2}\psit)\c\overline{TT^{1/2}(|\qt|^2\Nt})\big)}\,|\qt|^{-2}dv_\gt\right|\\
&+\int_{\DDt_+}\mathbf1_{\{2\de<X<3\de\}}\big(|T^2T^{1/2}\psit|^2+|ZTT^{1/2}\psit|^2+|\pr_{\rt}TT^{1/2}\psit|^2+|TT^{1/2}\psit|^2\big)\,|\qt|^{-2}dv_{\gt}.\end{split}
\eeq
\end{proposition}

\begin{proof}
Let
\begin{align*}
\EE^T[\psihat]&:=-(\hbt^{\taut\taut}+O^{\taut\taut})|T\psihat|^2+\hbt^{\rhot\rhot}|\pr_{\rhot}\psihat|^2+2\hbt^{\rhot\vphit_-}\Re(\pr_{\rhot}\psihat\overline{ Z\psihat})\\
&\qquad+O^{\tht\tht}|\pr_{\tht}\psihat|^2+(\hbt^{\vphit_-\vphit_-}+O^{\vphit_-\vphit_-})|Z\psihat|^2+V|\psihat|^2.
\end{align*}
We multiply the equation \eqref{eq:conformal-wave-equation-Xplus} 
\[
|\qt|^2\square_{\gt}\psihat+V\psihat=(X+1)|\qt|^2\Nt
\]
 by $f(X)T\psihat$ and integrate over the region $\widehat{\DDt_+}(\taut_1,\taut_2)$ with respect to $d\taut d\rhot d\omega$ to derive
\begin{align*}
&\int_{\widetilde\Sigma_{s}}-\frac12f(X)\EE^T[\psihat]\,d\rhot d\omega\Big|_{s=\taut_1}^{s=\taut_2}
+\int_{\IIt^+(\taut_1,\taut_2)}-\hbt^{\taut\rhot}|T\psihat|^2d\taut d\omega\\
&\qquad=\Re\int_{\widehat{\DDt_+}(\taut_1,\taut_2)}\pr_{\rhot}f(X)\hbt^{\rhot\b}\pr_\be\psihat\overline{T\psihat}\,d\taut d\rhot d\omega+\Re\int_{\widehat{\DDt_+}(\taut_1,\taut_2)}
(f(X)\T\psihat)\overline{(X+1)|\qt|^2\Nt}\,d\taut d\rhot d\omega.
\end{align*}
First, since $\psihat|_{\IIt^+}=\psit_\infty$ and $\hbt^{\taut\rhot}\sim 1$ and $\taut=\ut-m^2/\rt$ at $\IIt^+$, it follows that 
\[
\int_{\IIt^+(\taut_1,\taut_2)}\hbt^{\taut\rhot}|T\psihat|^2d\taut d\omega \sim \int_{\IIt^+(\taut_1,\taut_2)}|\T\psit_\infty|^2\,d\ut\,d\omega.
\]
Second, since $\T\psihat=(X+1)\,\T\psit$ and $d\taut d\rhot d\omega\sim |X+1|^{-2}d\taut d\rt d\omega=|X+1|^{-2}|\qt|^{-2}dv_{\gt}$, it follows that 
\[
\int_{\widehat{\DDt_+}}
(f(X)\T\psihat)\overline{(X+1)|\qt|^2\Nt}\,d\taut d\rhot d\omega \sim \int_{\DDt_+}
(f(X)\T\psit)\c\overline{|\qt|^2\Nt}\,|\qt|^{-2}dv_\gt.
\]
Third, since $
|\pr_{\rhot}f(X)|\les \de^{-1}\mathbf 1_{\{2\de<X<3\de\}}$ and in the region $2\de\leq X\leq 3\de$
\[
|\hbt^{\rhot\rhot}|\sim X, \quad  \hbt^{\rhot\taut}=\hbt^{\rhot \vphit}=0
\]
it follows that
\[
|\pr_{\rhot}f(X)\hbt^{\rhot\b}\pr_\be\psihat\overline{T\psihat}|\les \mathbf 1_{\{2\de<X<3\de\}}\Big(|T\psit|^2+|Z\psit|^2+|\pr_{\rt}\psit|^2+|\psit|^2\Big)\]
and thus
\[
\left|\Re\int_{\widehat{\DDt_+}(\taut_1,\taut_2)}\pr_{\rhot}f(X)\hbt^{\rhot\b}\pr_\be\psihat\overline{T\psihat}\,d\taut d\rhot d\omega\right|\les \int_{\DDt_+(\taut_1,\taut_2)}\mathbf1_{\{2\de<X<3\de\}}\big(|T\psit|^2+|Z\psit|^2+|\pr_{\rt}\psit|^2+|\psit|^2\big)\,|\qt|^{-2}dv_{\gt}.
\]
All the above analysis also apply to $TT^{1/2}\psit$.

Finally, by Lemma \ref{lem:acceptablesol}, the energy fluxes terms associated to $TT^{1/2}\psit$ at $\taut_1, \taut_2$ vanish as $\taut_1\to-\infty$ and $\tau_2\to\infty$. Therefore, letting
$\taut_1\to-\infty$ and $\taut_2\to\infty$ gives \eqref{eq:confenergy-psitXpos}.
\end{proof}

\subsection{Localized elliptic estimate}
In order to control the second term on the right side of the estimate \eqref{eq:confenergy-psitXpos}, we will make use of the following frequency localized   coercivity lemma. 

\begin{lemma}[Coercivity in the superradiant region]
\lab{lem:positive-superradiant-tangential-coercivity}
Let $\Phi$ have Fourier support in
$\mathcal B_M=\{\beta_+\xi<0,\ |\beta_+|>M\}$.  For \(\de>0\) sufficiently
small and $\de<X<4\de$, we have
\beq\lab{eq:positive-superradiant-tangential-coercivity}
\bsplit
&\int|\qt|^2\gt^{\taut\taut}|T\Phi|^2
+2|\qt|^2\gt^{\taut\vphit}
\Re(T\Phi\,\overline{Z\Phi})
+|\qt|^2\gt^{\vphit\vphit}|Z\Phi|^2\,d\taut d\vphit\gtrsim \int |T\Phi|^2+\frac{1}{\sin^2\tht}|Z\Phi|^2 \,d\taut d\vphit.
\end{split}
\eeq
\end{lemma}
\begin{proof}
Assume in what follows that $a \ge 0$.\footnote{Because we have the symmetry $a\to -a$ and $\vphit\to -\vphit$.}
According to the definition of $\taut$, we have $\taut=\tt$ on $\de<X<4\de$.  Hence \eqref{eq:BL-hbt.coordinates-general} gives
\[
  |\qt|^2\gt^{\tt\tt}=a^2\sin^2\tht+O(X),  \quad
 |\qt|^2\gt^{\tt\vphit}= a+aX, \quad |\qt|^2\gt^{\vphit\vphit}=\frac{1}{\sin^2\tht}.
 \]
By Plancherel, the term on the
left side of \eqref{eq:positive-superradiant-tangential-coercivity} can be written as 
\beq\lab{eq:positive-tangential-form-mode}
\sum_{n}\int \mathfrak q_X(\xi,n)|\FF_{\taut}[\Phi](\xi)|^2\,d\xi 
\eeq
where
\beq\lab{eq:positive-tangential-symbol-mode}
\mathfrak q_X(\xi,n)
=\left(a\sin\tht\,\xi+\frac{n}{\sin\tht}\right)^2
+O(X)\left(\xi^2+2a|\xi| |n|\right).
\eeq
Since
\[
\beta_+\xi=\la_+\xi(\xi+\omega_+n)=\la_+\xi(\xi+\frac{a}{2mr_+}n)\leq 0,
\]
we obtain\footnote{We can safely assume $n\neq 0$.} with $y=\xi/n$
\[
-\frac{a}{2mr_+}\leq y \leq 0.
\]
Then we have
\[
|ay\sin^2\tht +1|
\geq |1-\frac{a^2}{2mr_+}\sin^2\tht|\\
\geq |1-\frac{a^2}{2mr_+}|\geq \frac12
\]
and thus 
\[
\mathfrak q_{X=0}(\xi,n)
=\frac{n^2}{\sin^2\tht}|ay\sin^2\tht +1|^2\geq \frac14\frac{n^2}{\sin^2\tht}\gtrsim \xi^2+\frac{n^2}{\sin^2\tht}.
\]
Choosing $\de>0$ sufficiently small,  we then have 
\[
\mathfrak q_X(\xi,n)\gtrsim \xi^2+\frac{n^2}{\sin^2\tht},
\qquad \de\leq X\leq 4\de.
\]
Applying Plancherel again proves
\eqref{eq:positive-superradiant-tangential-coercivity}.
\end{proof}

Now we are ready to prove the following localized elliptic estiamte.
We use the notation
\beq\lab{eq:positive-elliptic-density}
\mathcal E_{\mathrm{ell}}[\Phi]
:=|\T\Phi|^2+\frac{1}{\sin^2\tht}|\Z\Phi|^2+|\pr_{\tht}\Phi|^2
+X(X+2)|\pr_{\rt}\Phi|^2+|\Phi|^2
\eeq

\begin{proposition}[Localized elliptic estimate]
\lab{prop:positive-localized-elliptic-estimate}
Let $\psi$ be an acceptable solution of $\square_{\g}\psi=N$. With
$\mathcal E_{\mathrm{ell}}[\Pi_{\mathcal{B}_M}\psit]$ defined in
\eqref{eq:positive-elliptic-density}, we have
\beq\lab{eq:positive-localized-elliptic-estimate}
\bsplit
\int_{2\de\leq X \leq 3\de}\mathcal E_{\mathrm{ell}}[\Pi_{\mathcal{B}_M}T T^{1/2}\psit]\,|\qt|^{-2}dv_\gt
&\lesssim \left|\Re\int_{\de\leq X\leq 4\de}
\zeta(X)\Pi_{\mathcal{B}_M}TT^{1/2}\psit\c\overline{\Pi_{\mathcal{B}_M}TT^{1/2}(|\qt|^2\Nt)}\,|\qt|^{-2}dv_{\gt}\right|\\
&\quad+\int_{\de\leq X\leq 4\de}\de^{-1}|\Pi_{\mathcal{B}_M}TT^{1/2}\psit|^2\,|\qt|^{-2}dv_{\gt}.
\end{split}
\eeq
The conclusion remains true after applying stationary  $T, \Z$ of  appropriate  $1/2$ versions of them.
\end{proposition}

\begin{proof}
We let
\[
\Phi=\Pi_{\mathcal{B}_M}TT^{1/2}\psit.
\]
Let $\zeta=\zeta(X)$ be a non-negative cutoff satisfying
\[
\zeta=1\quad\hbox{for  }2\de\leq X\leq 3\de,\qquad
\zeta=0\quad\hbox{for  }X\in(-\infty, \de]\cup [4\de,\infty).
\]
Multiplying
the equation $|\qt|^2\square_{\gt}\Phi=\Pi_{\mathcal{B}_M}TT^{1/2}(|\qt|^2\Nt)$ for $\Phi$ by $\zeta\overline\Phi$, integrating over $\DDt_+$ with respect to $d\taut d\rt d\omega$, and taking the real part gives
\beq\lab{eq:ellpticest2}
\bsplit
\Re\int_{\DDt_+}
\zeta(X)|\qt|^2\gt^{\at\bt}
\pr_{\at}\Phi\,\pr_{\bt}\overline\Phi
\,d\taut\,d\rt\,d\omega
&+\Re\int_{\{\de<X<4\de\}}
\zeta'(X)\,S\hbt^{\rt\at}
\pr_{\at}\Phi\,\overline\Phi
\,d\taut\,d\rt\,d\omega
=S_\ze,\\
S_\ze
&:=-\Re\int_{\DDt_+}
\zeta(X)
\bigl(\Pi_{\mathcal{B}_M}T T^{1/2}(|\qt|^2\Nt)\bigr)\c\overline\Phi
\,d\taut\,d\rt\,d\omega.
\end{split}
\eeq
We note that there is no boundary term at $X=0$, since
$\hbt^{\rt\at}=O(X)$ there, and the time-boundary terms vanish by Lemma \ref{lem:acceptablesol}.

By Lemma \ref{lem:positive-superradiant-tangential-coercivity}, the
$(T,Z)$-part of the first term on the left side of
\eqref{eq:ellpticest2} is coercive.  Together with
\[
\hbt^{\rt\rt}=S^{-2}X(X+2)>0,
\qquad O^{\tht\tht}=1,
\]
we conclude that the first term is nonnegative and controls $\mathcal E_{\mathrm{ell}}[\Phi]$ on $2\de\leq X\leq 3\de$. 

 The
second term on the left side of \eqref{eq:ellpticest2}, containing
$\zeta'$, is supported in $\de<X<4\de$.  On $0<X<4\de$ we have
$\taut=\tt$, and thus the only nonzero component of $|\qt|^2\gt^{\rt\at}$ is $\hbt^{\rt\rt}$. Integration by parts in $\rt$ gives
\[
\Re\int_{\de\leq X\le 4\de}
\zeta'(X)S\hbt^{\rt\rt}
(\pr_\rt\Phi)\overline\Phi\,
d\taut\,d\rt\,d\omega\\
=-\frac12\int_{\de\leq X\leq 4\de}
\pr_\rt\bigl(\zeta'(X)S\hbt^{\rt\rt}\bigr)
|\Phi|^2\,d\taut\,d\rt\,d\omega .
\]
On the fixed annulus $\de\leq X\leq 4\de$, we have
\[
\left|
\pr_\rt\bigl(\zeta'(X)S\hbt^{\rt\rt}\bigr)
\right|\les \de^{-1}.
\]
Therefore
\[
\left|\Re\int_{\de\leq X\le 4\de}
\zeta'(X)S\hbt^{\rt\rt}
(\pr_\rt\Phi)\overline\Phi\,
d\taut\,d\rt\,d\omega\right|\les \int_{\de\leq X\leq 4\de}\de^{-1}|\Phi|^2\,|\qt|^{-2}dv_{\gt}.
\]
This finishes the proof of \eqref{eq:positive-localized-elliptic-estimate}.
\end{proof}
Combining the localized elliptic estimate in Proposition \ref{prop:positive-localized-elliptic-estimate} with the weighted conformal energy estimate (applied to $\Pi_{\mathcal{B}_M}\psit$) in Proposition \ref{Prop:conf-en(sk-0)Xpos} yields the following.
\begin{proposition}
\lab{Prop:finalenergyXpos}
Let $\psi$ be an acceptable solution of $\square_{\g}\psi=N$. Then 
\beq
\lab{eq:finalenergy-psitXpos}
\bsplit
\int_{\IIt^+}|T^2T^{1/2}\Pi_{\mathcal{B}_M}\psit_\infty|^2\,d\ut\,d\omega
&\lesssim
\left|\int_{\DDt_+}
(f(X)T^2T^{1/2}\Pi_{\mathcal{B}_M}\psit)\c\overline{TT^{1/2}\Pi_{\mathcal{B}_M}(|\qt|^2\Nt)}\,|\qt|^{-2}dv_\gt\right|\\
&\quad+\left|\int_{\de\leq X\leq 4\de}
\zeta(X)TT^{1/2}\Pi_{\mathcal{B}_M}\psit\c\overline{TT^{1/2}\Pi_{\mathcal{B}_M}(|\qt|^2\Nt)}\,|\qt|^{-2}dv_{\gt}\right|\\
&\quad+\int_{\de\leq X\leq 4\de}\de^{-1}|TT^{1/2}\Pi_{\mathcal{B}_M}\psit|^2\,|\qt|^{-2}dv_{\gt}
\end{split}.
\eeq
The conclusion remains true after applying stationary  $T, \Z$ of  appropriate  $1/2$ versions of them.
\end{proposition}



\subsection{Control of horizon flux of $\psi$}
\lab{sec:positive-superradiant-horizon-flux}

\begin{proposition}[Large superradiant horizon flux]
\lab{prop:positive-branch-large-superradiant-flux}
Let\footnote{Assume that the radiation field $\psit$ has the desired regularity.}
 $\psi$ be an acceptable solution to $\square_{\g}\psi=N$.  Let
$\psit=\Wk[\psi]$  Then
\beq
\lab{eq:psi-psit-flux-strong-Xplus}
\int_{\HH^+}
\big|\Pi_{\mathcal B_M}\T_+^{1/2}\psi\big|^2+\big|\Pi_{\mathcal B_M}\psi\big|^2
\;dv\,d\omega
\lesssim
\int_{\IIt^+}\big|\Pi_{\mathcal B_M}\T\psit_\infty\big|^2\,d\ut\,d\omega .
\eeq
The conclusion remains true after applying stationary  $T, \Z$ of  appropriate  $1/2$ versions of them.
\end{proposition}
\begin{proof}
The proof is the same as that of Proposition \ref{prop:psi-psit-flux-strong-Xminus}, except for the change of sign of $C(\rt)=-S(\rt-\Rt)-2=-X-2<0$ as $\rt\to -\infty$.
 \end{proof}

\section{The proof of  the  main theorem}


\subsection{Statement of the main theorem}
We recall below the $B$ and $P$ norms   uses in \cite{He-K1} (see Section~2.4 of
\cite{He-K1}).
 With
$\DD_{\ntrap}$ denoting the nontrapping region, set
\beq
\lab{eq:defofnorms}
\bsplit
B[\psi](\widehat\tau_1,\widehat\tau_2)
&:=\int_{\DD(\widehat\tau_1,\widehat\tau_2)}
\left(r^{-2}|\Rhat\psi|^2+r^{-4}|\psi|^2\right)\,dv_\g\\
&\quad+\int_{\DD_{\ntrap}(\widehat\tau_1,\widehat\tau_2)}
\left(r^{-2}|e_{3,4}\psi|^2+r^{-1}|\nab\psi|^2\right)\,dv_\g,\\
P[\psi](\tauh_1, \tauh_2)&:=\int_{\DD_{trap}(\tauh_1, \tauh_2) }\left|\mathfrak{T}[\psi]\right|^2, \qquad \mathfrak{T}[\psi]:=\big(\TT T-a(r-m)\Z\big)\psi,
\end{split}
\eeq
where $\TT= r^3 -3 mr^2+ a^2r + a^2 m$. Also
\beq
\bsplit
B_p[\psi](\tauh_1, \tauh_2) &=B[\psi](\tauh_1, \tauh_2)+ \int_{\DD(\tauh_1, \tauh_2)\cap\{r\ge r_0\}} r^{p-3}\big( |\dk \psi|^2 +|\psi|^2\big)\\
\NN_p^s[ N](\tauh_1, \tauh_2) &:=\sum_{k\leq s}\int_{\DD(\tauh_1, \tauh_2)}r^{p+1}|\dk^kN|^2,\\
\NN_p^s[\psi,  N](\tauh_1, \tauh_2) &=\NN_p^s[N](\tauh_1,\tauh_2)+\sum_{k\leq s}\left|\int_{\DD_{trap}{(\tauh_1, \tauh_2)}} (\T, \Z)\dk^k\psi \c\dk^kN\right|.
\end{split}
\eeq
We  also defined  slightly modified versions of the $E$ and $E_p$ norms
\bea
\bsplit
E[\psi](\widehat s)
&:=\int_{\Si(\widehat s)}
\left(|e_4\psi|^2+r^{-1}|e_3\psi|^2
+|\nab\psi|^2+r^{-2}|\psi|^2\right),\\
E_p[\psi](\widehat s)
&:=E[\psi](\widehat s)
+\int_{\Si(\widehat s)\cap\{r\geq r_0\}}r^p
\left(r^{-2}|\nab_4(r\psi)|^2
+r^{-1}|\nab\psi|^2+r^{-2}|\psi|^2\right).
\end{split}
\eea
We   also make use of the  compact energy  $E_{\leq r_0}$ which  is the restriction of $E$ to
$\Si(\widehat s)\cap\{r\leq r_0\}$.

For $Q=B,E,E_{\leq r_0},E_p$, we use the higher-order  derivative conventions
\[
Q^k[\psi]:=\sum_{j=0}^kQ[\dk^j\psi],\qquad \dk\psi=\big(r(e_4,\nab)\psi,e_3\psi\big).
\]
We  recall  below the statement of the main theorem  \ref{thm:mainthm}.
	       		
\begin{theorem}
\lab{thm:cutoff-horizon-flux-estimate}
Let $\psi$ be a solution of $\square_\g\psi=N$. For every $\epsilon>0$, we have for    $p>0$\footnote{See  \eqref{eq:defofnorms} for  the definition of the norms used here. Note that the $E_p$ norms  differs slightly from  the  corresponding norms in \cite{He-K1}. }  and     $r_0>0$   sufficiently large.
\beq\lab{eq:Thm-cutoff-horizon-flux-estimate}
\bsplit
\int_{\HH^+(0,\tauh)}
\big|T_+T^2\psi\big|^2\,dv\,d\omega
&\leq \epsilon \BPE_p^2[\psi](0,\tauh)+C(\epsilon)\int_{\widehat{\tau}}^{\widehat{\tau}+1}
E_{\leq r_0}^2[\psi](\widehat{s})\,d\widehat{s}\\
&\quad+C(\epsilon)\Big(E_p^2[\psi](0)+\NN_p^2[\psi, N](0,\tauh)\Big).
\end{split}
\eeq
\end{theorem}	
For simplicity we ignore\footnote{The additional terms due to $N$ in  the expansion \eqref{eq:mainestiamteinlaststep}  can be easily treated  by Cauchy-Schwarz.  } below the  contribution of the inhomogeneous   $N$ in the proof, i.e. we assume $\square \psi=0$.
In what follows   we treat  acceptable solutions  of the form $\square \psi'=N'$, where    $N'$  is  generated by the cut-off of the original solution  of  the homogeneous equation.   By abuse of language we  set   $\psi'=\psi$ and  $N=N'$ in  the proof below.

\subsubsection{Main constants in the proof}

The following constant are involved in the proof of the main theorem
\[
\de, K, \sigma, M, \varep, r_0, \ep.
\]
The constant $\de$ is the size of the annular region where localized elliptic estimate is applied.  The constants $\sigma, K$ are used to describe the superradiant region in a quantitative way.  The constant $M$ is tied to the size of the bounded frequency region. The constant $\varep$ is a small constant appearing in the main Cauchy-Schwarz inequality \eqref{eq:CS-main}.  The constant $r_0$ is a large constant tied to  the  definitions of the  compact  energy norm. Finally, $\epsilon$ is the (small) constant  which appears in the statement of our  main theorem. In the proof, we first fix $\de, \sigma, K$, which  are independent of each other.  Next we choose $M$ sufficiently large. Then we choose $\varep$ sufficiently small. Finally, we  choose $r_0$ sufficiently large.

\subsection{$\Nt$-conditional  horizon flux estimate}
\lab{sec:combined-horizon-flux}

We now combine the results from sections \ref{sec:Xneg} and \ref{sec:Xpos}.  We recall, see \eqref{eq:controlled-frequency-region} and  \eqref{define:BB_M},
\[
\mathcal A_{M}
=\{\beta_+\xi\ge0\}\cup\{|\beta_+|\le M\},
\qquad
\mathcal B_{M}
=\{\beta_+\xi<0,\ |\beta_+|>M\}.
\]
We further define subsets of $\mathcal{A}_M$ as follows
\beq
\mathcal A_{M,\sigma}
:=\{0\leq \beta_+\xi\leq  \sigma\beta_+^2\}\cup\{|\beta_+|\le M\},\qquad \mathcal A^c_{M,\sigma}:=\{\beta_+\xi> \sigma\beta_+^2, |\beta_+|>M\}\eeq
We note that $\mathcal A_{M,\sigma}, \mathcal A^c_{M,\sigma}, \mathcal B_{M}$ are disjoint and their union is the full
frequency space.  
\begin{proposition}[Full horizon flux estimate]
\lab{prop:combined-horizon-flux}
Let $\psi$ be an acceptable solution of $\square_{\g}\psi=N$. Then for any $M>0$ and $\sigma>0$, we have
\beq\lab{eq:combined-horizon-flux}
\bsplit
&\int_{\HH^+}
|T_+T^2\psi|^2+|T_+^{1/2}T^{5/2}\psi|^2
\;dv\,d\omega\\
&\quad\les C_1(M,\sigma)\left|\int_{\DDt}
(\T^3(T_+^{1/2}, T^{1/2})\Pi_{\mathcal A_{M,\sigma}}\psit)\,
\overline{T^2(T_+^{1/2}, T^{1/2})\Pi_{\mathcal A_{M,\sigma}}(|\qt|^2\Nt)}\,|\qt|^{-2}dv_\gt\right|\\
&\qquad+C_2(\sigma)\left|\int_{\DDt_+}
(f(X)T^3(T_+^{1/2}, T^{1/2})\Pi_{\mathcal{B}_M}\psit)\c\overline{T^2(T_+^{1/2}, T^{1/2})\Pi_{\mathcal{B}_M}(|\qt|^2\Nt)}\,|\qt|^{-2}dv_\gt\right|\\
&\qquad+C_2(\sigma)\left|\int_{\de\leq X\leq 4\de}
(\zeta(X)T^2(T_+^{1/2}, T^{1/2})\Pi_{\mathcal{B}_M}\psit)\c\overline{T^2(T_+^{1/2}, T^{1/2})\Pi_{\mathcal{B}_M}(|\qt|^2\Nt)}\,|\qt|^{-2}dv_{\gt}\right|\\
&\qquad+C_2(\sigma)\int_{\de\leq X\leq 4\de}\de^{-1}|T^2(T_+^{1/2}, T^{1/2})\Pi_{\mathcal{B}_M}\psit|^2\,|\qt|^{-2}dv_{\gt}\\
&\qquad+C\sigma^{-1}\int_{\HH^+}
\Re\big(T_+T^2\psi\c\overline{TT^2\psi}\big)\,dvd\omega.
\end{split}
\eeq
The first three terms on the right side are generated from the inhomogeneous term $|\qt|^2\Nt$ and the fourth term is lower order.
\end{proposition}

\begin{proof}
Decompose the horizon trace into the three frequency pieces
\[
1=\Pi_{\mathcal A_{M,\si}}+\Pi_{\mathcal B_{M}}+\Pi_{\mathcal A^c_{M,\si}}
\]
Since the three frequency regions are disjoint, Plancherel gives
\begin{align*}
\int_{\HH^+}
|\T_+T^2\psi|^2+|\T_+^{1/2}T^{5/2}\psi|^2\,dv\,d\omega
&=
\int_{\HH^+}
\big|\Pi_{\mathcal A_{M,\sigma}}\T_+T^2\psi\big|^2+\big|\Pi_{\mathcal A_{M,\sigma}}\T_+^{1/2}T^{5/2}\psi\big|^2\,dv\,d\omega\\
&+\int_{\HH^+}
\big|\Pi_{\mathcal B_{M}}\T_+T^2\psi\big|^2+\big|\Pi_{\mathcal B_{M}}\T_+^{1/2}T^{5/2}\psi\big|^2\,dv\,d\omega\\
&+\int_{\HH^+}
\big|\Pi_{\mathcal A^c_{M,\sigma}}\T_+T^2\psi\big|^2+\big|\Pi_{\mathcal A^c_{M,\sigma}}\T_+^{1/2}T^{5/2}\psi\big|^2\,dv\,d\omega.
\end{align*}
In the frequency region $\mathcal A^c_{M,\sigma}$, we have
\[
\xi(\xi+\frac{a}{2mr_+}n)=\frac{1}{\la_+}\xi\beta_+>\frac{\sigma}{\la_+}\beta_+^2=\sigma\la_+(\xi+\frac{a}{2mr_+}n)^2,\qquad \xi(\xi+\frac{a}{2mr_+}n)=|\xi||\xi+\frac{a}{2mr_+}n|
\]
and thus
\begin{align*}
&\int_{\HH^+}
\big|\Pi_{\mathcal A^c_{M,\sigma}}\T_+T^2\psi\big|^2+\big|\Pi_{\mathcal A^c_{M,\sigma}}\T_+^{1/2}T^{5/2}\psi\big|^2\,dv\,d\omega\\
&\quad \les \sigma^{-1} \int_{\HH^+}
\Re\big(\Pi_{\mathcal A^c_{M,\sigma}}T_+T^2\psi\c\overline{\Pi_{\mathcal A^c_{M,\sigma}}TT^2\psi}\big)\,dvd\omega.
\end{align*}
We further estimate 
\begin{align*}
&\int_{\HH^+}
\Re\big(\Pi_{\mathcal A^c_{M,\sigma}}T_+T^2\psi\c\overline{\Pi_{\mathcal A^c_{M,\sigma}}TT^2\psi}\big)\,dvd\omega\\
&\quad=\int_{\HH^+}
\Re\big(T_+T^2\psi\c\overline{TT^2\psi}\big)\,dvd\omega-\int_{\HH^+}
\Re\big((\Pi_{\mathcal A_{M,\sigma}}+\Pi_{\mathcal B_{M}})T_+T^2\psi\c\overline{(\Pi_{\mathcal A_{M,\sigma}}+\Pi_{\mathcal B_{M}})TT^2\psi}\big)\,dvd\omega\\
&\quad\leq \int_{\HH^+}
\Re\big(T_+T^2\psi\c\overline{TT^2\psi}\big)\,dvd\omega+\int_{\HH^+}
|\Pi_{\mathcal A_{M,\sigma}}T_+^{1/2}T^{5/2}\psi|^2\,dvd\omega+\int_{\HH^+}
|\Pi_{\mathcal B_{M}}T_+^{1/2}T^{5/2}\psi|^2\,dvd\omega.
\end{align*}
Therefore, 
\begin{align*}
\int_{\HH^+}
|\T_+T^2\psi|^2+|\T_+^{1/2}T^{5/2}\psi|^2\,dv\,d\omega
&\leq
C(1+\sigma^{-1})\int_{\HH^+}
\big|\Pi_{\mathcal A_{M,\sigma}}\T_+T^2\psi\big|^2+\big|\Pi_{\mathcal A_{M,\sigma}}\T_+^{1/2}T^{5/2}\psi\big|^2\,dv\,d\omega\\
&+C(1+\sigma^{-1})\int_{\HH^+}
\big|\Pi_{\mathcal B_{M}}\T_+T^2\psi\big|^2+\big|\Pi_{\mathcal B_{M}}\T_+^{1/2}T^{5/2}\psi\big|^2\,dv\,d\omega\\
&+C\sigma^{-1}\int_{\HH^+}
\Re\big(T_+T^2\psi\c\overline{TT^2\psi}\big)\,dvd\omega.
\end{align*}
The $\mathcal A_{M,\sigma}$ contribution is controlled by Proposition
\ref{prop:psi-psit-flux-strong-Xminus}, using the subdomain $\DDt$ in negative the branch
$\DDt_-$.  The $\mathcal B_{M}$ contribution is controlled by Proposition
\ref{prop:positive-branch-large-superradiant-flux}, using the positive branch
$\DDt_+$. 
Therefore, we have 
\beq
\lab{eq:fullflux}
\bsplit
&\int_{\HH^+}
|\T_+T^2\psi|^2+|\T_+^{1/2}T^{5/2}\psi|^2\,dv\,d\omega\\
&\qquad\leq
C_1(M,\sigma)\int_{\IIt^+(X<-2)}
\big|\Pi_{\mathcal A_{M,\sigma}}\T^3(T_+^{1/2},T^{1/2})\psit_{\infty}\big|^2\,d\ut\,d\omega\\
&\qquad+
C_2(\sigma)\int_{\IIt^+(X>0)}
\big|\Pi_{\mathcal B_{M}}\T^3(T_+^{1/2}, T^{1/2})\psit_{\infty}\big|^2\,d\ut\,d\omega\\
&\qquad+C\sigma^{-1}\int_{\HH^+}
\Re\big(T_+T^2\psi\c\overline{TT^2\psi}\big)\,dvd\omega.
\end{split}
\eeq
Apply Proposition \ref{Prop:conf-en(sk-0)} to $\Pi_{\mathcal A_{M,\sigma}}(T^{1/2}T_+^{1/2}, T)\psit$ yields
\beq
\lab{eq:negnullinfinity}
\bsplit
&\int_{\IIt^+(X<-2)}
\big|\Pi_{\mathcal A_{M,\sigma}}\T^3T_+^{1/2}\psit_{\infty}\big|^2\,d\ut\,d\omega\\
&\qquad\les\left|\int_{\DDt}
(\T^3(T_+^{1/2}, T^{1/2})\Pi_{\mathcal A_{M,\sigma}}\psit)\,
\overline{T^2(T_+^{1/2}, T^{1/2})\Pi_{\mathcal A_{M,\sigma}}(|\qt|^2\Nt)}\,|\qt|^{-2}dv_\gt\right|.
\end{split}
\eeq
Applying Proposition \ref{Prop:finalenergyXpos} to  $(T^{1/2}T_+^{1/2}, T)\psit$, we have
\beq
\lab{eq:posnullinfinity}
\bsplit
&\int_{\IIt^+}|\Pi_{\mathcal{B}_M}T^3T_+^{1/2}\psit_\infty|^2\,d\ut\,d\omega\\
&\qquad\lesssim
\left|\int_{\DDt_+}
(f(X)\T^3(T_+^{1/2}, T^{1/2})\Pi_{\mathcal{B}_M}\psit)\c\overline{T^2(T_+^{1/2}, T^{1/2})\Pi_{\mathcal{B}_M}(|\qt|^2\Nt)}\,|\qt|^{-2}dv_\gt\right|\\
&\qquad+\left|\int_{\de\leq X\leq 4\de}
(\zeta(X)T^2(T_+^{1/2}, T^{1/2})\Pi_{\mathcal{B}_M}\psit)\c\overline{T^2(T_+^{1/2}, T^{1/2})\Pi_{\mathcal{B}_M}(|\qt|^2\Nt)}\,|\qt|^{-2}dv_{\gt}\right|\\
&\qquad+\int_{\de\leq X\leq 4\de}\de^{-1}|T^2(T_+^{1/2}, T^{1/2})\Pi_{\mathcal{B}_M}\psit|^2\,|\qt|^{-2}dv_{\gt}.
\end{split}
\eeq
Putting \eqref{eq:fullflux}, \eqref{eq:negnullinfinity} and \eqref{eq:posnullinfinity} together proves \eqref{eq:combined-horizon-flux}.
\end{proof}

\subsection{ Estimates  for  $\Nt$  and end of the proof}

We apply Proposition \ref{prop:combined-horizon-flux} to $\chi\psi$ to obtain
\beq
\lab{eq:mainestiamteinlaststep}
\bsplit
&\int_{\HH^+}
|T_+T^2(\chi\psi)|^2+|T_+^{1/2}T^{5/2}(\chi\psi)|^2
\;dv\,d\omega\\
&\quad\les C_1(M,\sigma)\left|\int_{\DDt}
(\T^3(T_+^{1/2}, T^{1/2})\Pi_{\mathcal A_{M,\sigma}}\psit)\,
\overline{T^2(T_+^{1/2}, T^{1/2})\Pi_{\mathcal A_{M,\sigma}}(|\qt|^2\Nt)}\,|\qt|^{-2}dv_\gt\right|\\
&\qquad+C_2(\sigma)\left|\int_{\DDt_+}
(f(X)T^3(T_+^{1/2}, T^{1/2})\Pi_{\mathcal{B}_M}\psit)\c\overline{T^2(T_+^{1/2}, T^{1/2})\Pi_{\mathcal{B}_M}(|\qt|^2\Nt)}\,|\qt|^{-2}dv_\gt\right|\\
&\qquad+C_2(\sigma)\left|\int_{\de\leq X\leq 4\de}
(\zeta(X)T^2(T_+^{1/2}, T^{1/2})\Pi_{\mathcal{B}_M}\psit)\c\overline{T^2(T_+^{1/2}, T^{1/2})\Pi_{\mathcal{B}_M}(|\qt|^2\Nt)}\,|\qt|^{-2}dv_{\gt}\right|\\
&\qquad+C_2(\sigma)\int_{\de\leq X\leq 4\de}\de^{-1}|T^2(T_+^{1/2}, T^{1/2})\Pi_{\mathcal{B}_M}\psit|^2\,|\qt|^{-2}dv_{\gt}\\
&\qquad+C\sigma^{-1}\int_{\HH^+}
\Re\big(T_+T^2(\chi\psi)\c\overline{TT^2(\chi\psi)}\big)\,dvd\omega.
\end{split}
\eeq
Let $h(r)$ be a smooth cutoff such that
\[
h(r)=1\quad\hbox{for }r\leq \frac{2r_++\rhat_1}{3},
\qquad
h(r)=0\quad\hbox{for }r\geq\frac{r_++\rhat_1}{2}.
\]
We write
\beq
\psi=\psi_1+\psi_2,
\qquad
\psi_1:=h(r)\psi,
\qquad
\psi_2:=(1-h(r))\psi.
\eeq
Thus $\psi_1$ is supported near the event horizon and outside
the trapping region, whereas $\psi_2$ is supported away from the event horizon.  On the  Kerr side we have $\square_{\g}(\chi\psi)=N$ where
\[
|q|^2N=\big[|q|^2\square_{\g}, \chi(\widehat\tau)\big]\psi
=\big[|q|^2\square_{\g}, \chi(\widehat\tau)\big]\psi_1
+\big[|q|^2\square_{\g}, \chi(\widehat\tau)\big]\psi_2
:=|q|^2N_1+|q|^2N_2.
\]
Correspondingly, we define
\[
   |\qt|^2\Nt =\Wk[|q|^2N_\chi], \qquad |\qt|^2\Nt_{1}:=\Wk[|q|^2N_1],
\qquad
|\qt|^2\Nt_{2}:=\Wk[|q|^2N_2].
\]
We shall analyze the right side of \eqref{eq:mainestiamteinlaststep} term by term.

{\bf The first term.} We first analyze the term in \eqref{eq:mainestiamteinlaststep}
\[
\left|\int_{\DDt}
(\T^3(T_+^{1/2}, T^{1/2})\Pi_{\mathcal A_{M,\sigma}}\psit)\,
\overline{T^2(T_+^{1/2}, T^{1/2})\Pi_{\mathcal A_{M,\sigma}}(|\qt|^2\Nt)}\,|\qt|^{-2}dv_\gt\right|.
\]
It suffices to handle the term 
\beq
\left|\int_{\DDt}
(\T^3T_+^{1/2}\Pi_{\mathcal A_{M,\sigma}}\psit)\,
\overline{T^2T_+^{1/2}\Pi_{\mathcal A_{M,\sigma}}(|\qt|^2\Nt)}\,|\qt|^{-2}dv_\gt\right|
\eeq
as the case of $T^{1/2}$ can be analyzed in the same way.

Applying Cauchy-Schwarz yields
\beq
\lab{eq:CS-main}
\left|\int_{\DDt}
(\T^3T_+^{1/2}\Pi_{\mathcal A_{M,\sigma}}\psit)\,
\overline{T^2T_+^{1/2}\Pi_{\mathcal A_{M,\sigma}}(|\qt|^2\Nt)}\,|\qt|^{-2}dv_\gt\right|\leq 
\varep\Romanupper{1}_1+C(\varep)\Romanupper{2}_1
\eeq
where
\beq
\Romanupper{1}_1=\int_{\DDt}
\big|T^3T_+^{1/2}\Pi_{\mathcal A_{M,\sigma}}\psit\big|^2\,|\qt|^{-2}dv_\gt,\qquad
\Romanupper{2}_1=\int_{\DDt}
\big|T^2T_+^{1/2}(|\qt|^2\Nt)\big|^2\,|\qt|^{-2}dv_\gt.
\eeq
 
 We now analyze the term $\Romanupper{2}_1$. Since $\DDt=\{X\leq -2\}$, we have $|B(\rt)|=|X|\geq 2$. This lower bound of $|B(\rt)|$ is the  reason that the following
insertion of $B(\rt)$ is legitimate:
\beq
\bsplit
\Romanupper{2}_1&\lesssim\int_{\DDt}\left|T^{2}T_+^{1/2}(|\qt|^2\Nt_1)\right|^2|\qt|^{-2}dv_{\gt}+\int_{\DDt}\left|T^{2}T_+^{1/2}|\qt|^2\Nt_2\right|^2|\qt|^{-2}dv_{\gt}\\
&\les \int_{\DDt}\left|T^{2}T_+^{1/2}|\qt|^2\Nt_1\right|^2|\qt|^{-2}dv_{\gt}+\int_{\DDt}\left|B(\rt)T^{2}T_+^{1/2}|\qt|^2\Nt_2\right|^2|\qt|^{-2}dv_{\gt}.
\end{split}
\eeq
We justify the second term in the last line.  Since
$|\qt|^2\Nt_2=\Wk[|q|^2N_2]$ and the $(T,Z)$-derivatives commute with
$\Wk$,
\[
B(\rt)T^{2}T_+^{1/2}(|\qt|^2\Nt_2)
=B(\rt)\pr_\ut\Wk[TT_+^{1/2}(|q|^2N_2)].
\]
Since the source $N_2$ is supported away from the event horizon, applying the second estimate of Corollary
\ref{cor:filtered-L2-outgoing-transform} on $\DDt\subset\DDt_-$ yields
\begin{align*}
\int_{\DDt}\left|B(\rt)\pr_\ut\Wk[TT_+^{1/2}(|q|^2N_2)]\right|^2|\qt|^{-2}dv_{\gt}
&\les
\int_{\DD}\left|e_4^{(out)}T^{1/2}T_+^{1/2}(|q|^2N_2)\right|^2|q|^{-2}dv_\g\\
&\les
\int_{\DD}\left|e_4(T,Z)(|q|^2N_2)\right|^2|q|^{-2}dv_\g.
\end{align*}
Hence, applying the first estimate of Corollary
\ref{cor:filtered-L2-outgoing-transform} on $\DDt\subset\DDt_-$ yields
\beq
\bsplit
 \Romanupper{2}_1&\les
 \int_{\DDt}\left|T^{2}T_+^{1/2}|\qt|^2\Nt_1\right|^2|\qt|^{-2}dv_{\gt}+\int_{\DD}\left|e_4(T,Z)(|q|^2N_2)\right|^2|q|^{-2}dv_\g\\
&\les \int_0^1 +\int_{\widehat{\tau}^*-1}^{\widehat{\tau}^*} E_{\leq r_0}[T(T,Z)\psi](\widehat{s})\,d\widehat{s}+\int_{\DD}\left|e_4(T,Z)(|q|^2N_2)\right|^2|q|^{-2}dv_\g.
\end{split}
\eeq
where $r_0\geq \frac{r_++\rhat_1}{2}$. Therefore, it remains to analyze $|q|^2N_2$ (more precisely, $e_4(|q|^2N_2)$). By Lemma \ref{Le:comm-tauh}, after commuting with
$(T,Z)$, we have for $r$ sufficiently large
\[
e_4(T,Z)(|q|^2N_2)
=O(1)\sum_{j=1}^{4}|\chi^{(j)}|
\left(\dk^{\leq1}(T,Z)^{\leq 1}\psi
+re_4\dk^{\leq1}(T,Z)^{\leq 1}\psi\right).
\]
This implies that 
\begin{align*}
\int_{\DD}\left|e_4(T,Z)(|q|^2N_2)\right|^2|q|^{-2}dv_\g
&\les \int_0^1 +\int_{\widehat{\tau}^*-1}^{\widehat{\tau}^*} E^1_{\leq r_0}[(T,Z)^{\leq1}\psi](\widehat{s})\,d\widehat{s}\\
&+r_0^{-p}\bigg(
\int_0^1  +\int_{\widehat{\tau}^*-1}^{\widehat{\tau}^*}E^1_p[(T,Z)^{\leq1}\psi](\widehat{s})\,d\widehat{s}
\bigg).
\end{align*}
and thus
\beq\lab{eq:term2_1}
\bsplit
\Romanupper{2}_1\les\int_0^1 +\int_{\widehat{\tau}^*-1}^{\widehat{\tau}^*} E^1_{\leq r_0}[(T,Z)^{\leq1}\psi](\widehat{s})\,d\widehat{s}+r_0^{-p}\bigg(
\int_0^1  +\int_{\widehat{\tau}^*-1}^{\widehat{\tau}^*}E^1_p[(T,Z)^{\leq1}\psi](\widehat{s})\,d\widehat{s}
\bigg)
\end{split}
\eeq
where $r_0$ is sufficiently large. This finishes the analysis of the term $\Romanupper{2}_1$.

We now turn to the term $\Romanupper{1}_1$. We define
\beq
\psit_1:=\Wk[\chi\psi_1],\quad \psit_2:=\Wk[\chi\psi_2]
\eeq
and estimate
\beq
\lab{eq:controlofterm1_1}
\bsplit
\Romanupper{1}_1&\lesssim\int_{\DDt}\left|T^{3}T_+^{1/2}\Pi_{\mathcal A_{M,\sigma}}\psit_1\right|^2|\qt|^{-2}dv_\gt+\int_{\DDt}\left|T^3T_+^{1/2}\Pi_{\mathcal A_{M,\sigma}}\psit_2\right|^2|\qt|^{-2}dv_\gt\\
&\les \int_{\DDt}\left|T^3T_+^{1/2}\psit_1\right|^2|\qt|^{-2}dv_\gt+\int_{\DDt}\left|B(\rt)T^3T_+^{1/2}\Pi_{\mathcal A_{M,\sigma}}\psit_2\right|^2|\qt|^{-2}dv_\gt\\
\end{split}
\eeq
 where we use again that
$|B(\rt)|\ge2$ on $\DDt$.  Hence, using Corollary \ref{cor:filtered-L2-outgoing-transform}  again,
\beq
\lab{eq:treatmentofI_}
\bsplit
\Romanupper{1}_1&\les \int_{\DDt}\left|\Wk[T^3T_+^{1/2}(\chi\psi_1)]\right|^2|\qt|^{-2}dv_\gt+\int_{\DDt}\left|\Wk[e^{(out)}_4T^2T_+^{1/2}\Pi_{\mathcal A_{M,\sigma}}(\chi\psi_2)]\right|^2|\qt|^{-2}dv_\gt\\
&\les \int_{\DD}\left|T^2(T,Z)(\chi\psi_1)\right|^2|q|^{-2}dv_\g+ \int_{\DD}\left|e_4T(T,Z)\Pi_{\mathcal A_{M,\sigma}}(\chi\psi_2)\right|^2|q|^{-2}dv_\g\\
&\les B^1[(T,Z)^{\leq1}\psi](0,\tauh^*)+ \int_{\DD}\left|e_4T(T,Z)\Pi_{\mathcal A_{M,\sigma}}(\chi\psi_2)\right|^2|q|^{-2}dv_\g.
\end{split}
\eeq
We further decompose $
\mathcal{A}_{M,\sigma}=\mathcal{A}^{\mathrm{bd}}_{M,\sigma}\cup \mathcal{A}^{\mathrm{nt}}_{M,\sigma}$ where\footnote{ Note that the analysis in the  first region works for any $K$.  We actually fix   $K$ sufficiently large, later,  in the analysis  of  $\mathcal{A}^{\mathrm{nt}}_{M,\sigma}$.}
\[
\mathcal{A}^{\mathrm{bd}}_{M,\sigma}:=\{|\beta_+|\leq M, \sigma\beta_+^2< \beta_+\xi< K\beta_+^2\},\qquad \mathcal{A}^{\mathrm{nt}}_{M,\sigma}:=\mathcal{A}_{M,\sigma}\setminus \mathcal{A}^{\mathrm{bd}}_{M,\sigma}.
\]
Hence
\beq
\lab{eq:furtherdecom}
\bsplit
\int_{\DD}\left|e_4T(T,Z)\Pi_{\mathcal A_{M,\sigma}}(\chi\psi_2)\right|^2|q|^{-2}dv_\g&\leq \int_{\DD}\left|e_4T(T,Z)\Pi_{\mathcal A^{\mathrm{bd}}_{M,\sigma}}(\chi\psi_2)\right|^2|q|^{-2}dv_\g\\
&+\int_{\DD}\left|e_4T(T,Z)\Pi_{\mathcal A^{\mathrm{nt}}_{M,\sigma}}(\chi\psi_2)\right|^2|q|^{-2}dv_\g.
\end{split}
\eeq
A straightforward calculation implies
\[
\mathcal{A}^{\mathrm{bd}}_{M,\sigma}\subset\{|\xi|\leq KM\},
\]
and thus
\beq\lab{eq:boundedfre}
\bsplit
\int_{\DD}\left|
e_4T(T,Z)\Pi_{\mathcal A_{M,\sigma}^{\mathrm{bd}}}(\chi\psi_2)\right|^2|q|^{-2}\,dv_\g
&\leq C(M,K)\int_{\DD}\left|
e_4(T,Z)\Pi_{\mathcal A_{M,\sigma}^{\mathrm{bd}}}(\chi\psi_2)\right|^2|q|^{-2}\,dv_\g\\
&\leq C(M,K) B^1[(T,Z)^{\leq1}\psi](0,\widehat\tau^*).
\end{split}
\eeq
As for the contribution from the term $\Pi_{\mathcal A^{\mathrm{nt}}_{M,\sigma}}(\chi\psi_2)$, we write
\begin{align*}
&\int_{\DD}\left|e_4T(T,Z)\Pi_{\mathcal A^{\mathrm{nt}}_{M,\sigma}}(\chi\psi_2)\right|^2|q|^{-2}dv_\g\\
&\quad=\int_{\DDntrap}\left|e_4T(T,Z)\Pi_{\mathcal A^{\mathrm{nt}}_{M,\sigma}}(\chi\psi_2)\right|^2|q|^{-2}dv_\g+\int_{\DDtrap}\left|e_4T(T,Z)\Pi_{\mathcal A^{\mathrm{nt}}_{M,\sigma}}(\chi\psi_2)\right|^2|q|^{-2}dv_\g\\
&\quad \les B^1[(T,Z)^{\leq1}\psi](0,\widehat{\tau}^*)+\int_{\DDtrap}\left|e_4T(T,Z)\Pi_{\mathcal A^{\mathrm{nt}}_{M,\sigma}}(\chi\psi_2)\right|^2|q|^{-2}dv_\g.
\end{align*}
In the trapping region $\DDtrap=[\rhat_1-\de_{trap}, \rhat_2+\de_{trap}]$, we further analyze the trapping term $P$  
\[
P[\psi](\tau_1, \tau_2):=\int_{\DD_{trap}(\tau_1, \tau_2) }\left|\mathfrak{T}[\psi]\right|^2, \qquad \mathfrak{T}[\psi]:=\big(\TT(r) T-a(r-m)\Z\big)\psi.
\]
We now use Lemma~3.16 from section~3.5.2 of
\cite{He-K1}.  It states, in the strictly subextremal case, that
\[
-2mr_+<\frac{\TT(\rhat_1)}{\rhat_1-m}.
\]
Since $\TT(r)/(r-m)$ is increasing, we may choose $\de_{trap}>0$
sufficiently small that, for some $c_0>0$,
\[
-2mr_+(1-2c_0)\leq \frac{\TT(r)}{r-m}\leq \frac{\TT(\rhat_2+\de_{trap})}{\rhat_2+\de_{trap}-m}\quad\hbox{on}\quad
[\rhat_1-\de_{trap},\rhat_2+\de_{trap}].
\]
In the region 
\[
\{\beta_+\xi\leq \sigma\beta_+^2\}\cup\{\beta_+\xi\geq K\beta_+^2\},
\]
 we have
\[
-\frac{K\la_+}{K\la_+-1}\frac{a}{2mr_+}\leq \frac{\xi}{n}\leq \frac{\sigma\la_+}{1-\sigma\la_+}\frac{a}{2mr_+}.
\]
Then we may choose $\sigma>0$ sufficiently small and $K>0$ sufficiently large such that 
\[
-(1+c_0)\frac{a}{2mr_+}\leq \frac{\xi}{n}\leq c_0\frac{a}{2mr_+} \quad \text{for}\quad (\xi,n)\in\{\beta_+\xi\leq \sigma\beta_+^2\}\cup\{\beta_+\xi\geq K\beta_+^2\}
\]
Consequently, for $r\in[\rhat_1-\de_{trap},\rhat_2+\de_{trap}]$ and $(\xi,n)\in  \{\beta_+\xi\leq \sigma\beta_+^2\}\cup\{\beta_+\xi\geq K\beta_+^2\}  $
\beq
\lab{eq:choiceofsigmaK}
\left|\TT(r)\xi-a(r-m)n\right|
=(r-m)|n|
\left|\frac{\TT(r)}{r-m}\frac{\xi}{n}-a\right|
\gtrsim a|n|\gtrsim|\xi|.
\eeq
Since  $  \mathcal{A}^{\mathrm{nt}}_{M,\sigma}\subset    \{\beta_+\xi\leq \sigma\beta_+^2\}\cup\{\beta_+\xi\geq K\beta_+^2\} $  we conclude that \eqref{eq:choiceofsigmaK} holds in 
 $  \mathcal{A}^{\mathrm{nt}}_{M,\sigma} $.
Therefore
\begin{align*}
&\int_{\DDtrap}\left|e_4T(T,Z)\Pi_{\mathcal A^{\mathrm{nt}}_{M,\sigma}}(\chi\psi_2)\right|^2|q|^{-2}dv_\g\\
&\quad\les \int_{\DDtrap}\left|\big(\TT(r)T-a(r-m)Z\big)e_4(T,Z)\Pi_{\mathcal A^{\mathrm{nt}}_{M,\sigma}}(\chi\psi_2)\right|^2|q|^{-2}dv_\g \les \BP^1[(T,Z)^{\leq 1}\psi](0,\tauh^*).
\end{align*}
and then
\beq
\lab{eq:nt}
\int\left|e_4T(T,Z)\Pi_{\mathcal A^{\mathrm{nt}}_{M,\sigma}}(\chi\psi_2)\right|^2|q|^{-2}dv_\g\les \BP^1[(T,Z)^{\leq1}\psi](0,\tauh^*).
\eeq
Putting \eqref{eq:treatmentofI_}, \eqref{eq:furtherdecom}, \eqref{eq:boundedfre} and \eqref{eq:nt} gives
\beq
\lab{eq:term1_1}
\Romanupper{1}_1\leq C(M,K)\BP^1[(T,Z)^{\le 1}\psi](0,\tauh^*).
\eeq
This finishes the analysis of the term $\Romanupper{1}_1$.

Finally, combining \eqref{eq:term2_1} and \eqref{eq:term1_1}, we obtain
\beq
\lab{eq:firstterm}
\bsplit
&\left|\int_{\DDt}
(\T^3(T_+^{1/2}, T^{1/2})\Pi_{\mathcal A_{M,\sigma}}\psit)\,
\overline{T^2(T_+^{1/2}, T^{1/2})\Pi_{\mathcal A_{M,\sigma}}(|\qt|^2\Nt)}\,|\qt|^{-2}dv_\gt\right|\\
&\quad\leq \varep C(M,K)\BP^1[(T,Z)^{\leq1}\psi](0,\tauh^*)\\
&\quad+C(\varep)\Big(\int_0^1 +\int_{\widehat{\tau}^*-1}^{\widehat{\tau}^*}
E^1_{\leq r_0}[(T,Z)^{\leq1}\psi](\widehat{s})\,d\widehat{s}\Big)\\
&\quad+C(\varep)r_0^{-p}\Big(
\int_0^1 +\int_{\widehat{\tau}^*-1}^{\widehat{\tau}^*}
E^1_p[(T,Z)^{\leq1}\psi](\widehat{s})\,d\widehat{s}\Big)\end{split}.
\eeq

{\bf The second term.}  We recall the second term in \eqref{eq:mainestiamteinlaststep}
\[
\left|\int_{\DDt_+}
(f(X)T^3(T_+^{1/2}, T^{1/2})\Pi_{\mathcal{B}_M}\psit)\c\overline{T^2(T_+^{1/2}, T^{1/2})\Pi_{\mathcal{B}_M}(|\qt|^2\Nt)}\,|\qt|^{-2}dv_\gt\right|.
\]
The proof uses two important ingredients.
\begin{enumerate}
\item  Since $f(X)$ satistifies $|f(X)|\leq 1$ and $f(X)=0$ for $0\leq X\leq \de$, $B(\rt)$ may be
inserted before using the commutation relation
\eqref{eq:Rk-commutation} and Corollary
\ref{cor:filtered-L2-outgoing-transform}. 
\item Since $\mathcal{B}_{M}\subset\{\beta_+\xi\leq 0\}$, the result in \eqref{eq:choiceofsigmaK} is also true for the second term. That is, there is no trapping here, in the sense that $(T,Z)$ derivatives can be controlled by $P$.
\end{enumerate}
Therefore, the estimate of the second term is the same as that of the first term, with the resulting constants now
depending on $\de$. Therefore,
\beq
\lab{eq:secondterm}
\bsplit
&\left|\int_{\DDt_+}
(f(X)T^3(T_+^{1/2}, T^{1/2})\Pi_{\mathcal{B}_M}\psit)\c\overline{T^2(T_+^{1/2}, T^{1/2})\Pi_{\mathcal{B}_M}(|\qt|^2\Nt)}\,|\qt|^{-2}dv_\gt\right|\\
&\quad\leq \varep C(\de)\BP^1[(T,Z)^{\leq 1}\psi](0,\tauh^*)\\
&\quad+C(\varep,\de)\Big(\int_0^1 +\int_{\widehat{\tau}^*-1}^{\widehat{\tau}^*}
E^1_{\leq r_0}[(T,Z)^{\leq1}\psi](\widehat{s})\,d\widehat{s}\Big)\\
&\quad+C(\varep,\de)r_0^{-p}\Big(
\int_0^1 
+\int_{\widehat{\tau}^*-1}^{\widehat{\tau}^*}
E^1_p[(T,Z)^{\leq1}\psi](\widehat{s})\,d\widehat{s}\Big)\end{split}.
\eeq

{\bf The third term.} The analysis of the third term in \eqref{eq:mainestiamteinlaststep}
\[
\left|\int_{\de\leq X\leq 4\de}
(\zeta(X)T^2(T_+^{1/2}, T^{1/2})\Pi_{\mathcal{B}_M}\psit)\c\overline{T^2(T_+^{1/2}, T^{1/2})\Pi_{\mathcal{B}_M}(|\qt|^2\Nt)}\,|\qt|^{-2}dv_{\gt}\right|
\]
is the same (in fact simpler because we do not need to make use of $P$) as that of the second term. Therefore
\beq
\lab{eq:thirdterm}
\bsplit
&\left|\int_{\de\leq X\leq 4\de}
(\zeta(X)T^2(T_+^{1/2}, T^{1/2})\Pi_{\mathcal{B}_M}\psit)\c\overline{T^2(T_+^{1/2}, T^{1/2})\Pi_{\mathcal{B}_M}(|\qt|^2\Nt)}\,|\qt|^{-2}dv_{\gt}\right|
\\
&\quad\leq \varep C(\de)B^1[(T,Z)^{\leq 1}\psi](0,\tauh^*)\\
&\quad+C(\varep,\de)\Big(\int_0^1 +\int_{\widehat{\tau}^*-1}^{\widehat{\tau}^*}
E^1_{\leq r_0}[(T,Z)^{\leq1}\psi](\widehat{s})\,d\widehat{s}\Big)\\
&\quad+C(\varep,\de)r_0^{-p}\Big(
\int_0^1 
+\int_{\widehat{\tau}^*-1}^{\widehat{\tau}^*}
E^1_p[(T,Z)^{\leq1}\psi](\widehat{s})\,d\widehat{s}\Big)
\end{split}
\eeq

{\bf The fourth term.} As for the fourth term  in \eqref{eq:mainestiamteinlaststep}
\[
\int_{\de\leq X\leq 4\de}\de^{-1}|T^2(T_+^{1/2}, T^{1/2})\Pi_{\mathcal{B}_M}\psit|^2\,|\qt|^{-2}dv_{\gt},
\]
using $\mathcal{B}_{M}=\{\beta_+\xi\leq 0, |\beta_+|>M\}$, we derive
\beq
\lab{eq:fourthterm}
\bsplit
&\int_{\de\leq X\leq 4\de}\de^{-1}|T^2(T_+^{1/2}, T^{1/2})\Pi_{\mathcal{B}_M}\psit|^2\,|\qt|^{-2}dv_{\gt}\\
&\quad\leq \frac{C(\de)}{M^2}\int_{\de\leq X\leq 4\de}|T^2T_+(T_+^{1/2}, T^{1/2})\Pi_{\mathcal{B}_M}\psit|^2\,|\qt|^{-2}dv_{\gt}\leq \frac{C(\de)}{M^2}\BP^1[(T,Z)^{\le1}\psi](0,\tauh^*).
\end{split}
\eeq

 {\bf The last term.} To control the last term in \eqref{eq:mainestiamteinlaststep}
 \[
 \Re\big(T_+T^2(\chi\psi)\c\overline{TT^2(\chi\psi)}\big)\,dvd\omega,
  \]
we apply the $T$-energy identity to the wave equation of $T^2(\chi\psi)$. More specifically, we compute
\[
 \D^\b \big( \T^\a \QQ_{\a\b }[T^2(\chi \psi)]\big)= \D^\b \big( \T^\a \QQ_{\a\b }[\chi T^2 \psi]\big)+\D^\b\big(T^\a\QQ_{\a\b}[T^2(\chi\psi)]-T^\a\QQ_{\a\b}[\chi T^2\psi]\big).
  \] 
Using \eqref{eq:cutoffdivergence}, we further write
\begin{align*}
 \D^\b \big( \T^\a \QQ_{\a\b }[T^2(\chi \psi)]\big)&= 2\chi\chi'\QQ[\T^2\psi](T, \D\tau)
 +\D^\b \Big((\T^2\psi)^2\T^\a  \QQ_{\a\b }[\chi]+2\chi T^2\psi\, \T^\a \QQ_{\a\b}[\chi,T^2\psi]\Big)\\
 &+\D^\b\big(T^\a\QQ_{\a\b}[T^2(\chi\psi)]-T^\a\QQ_{\a\b}[\chi T^2\psi]\big).
\end{align*}
We note that    the terms under the divergence  sign on the  right side vanish  in regions where  $\chi'=0$. Integrating the above identity over the region $\DD(-1, \tauh^*+1)$ gives
\beq
\lab{eq:lastterm}
\bsplit
&\int_{\HH^+} \Re\big(T_+T^2(\chi\psi)\c\overline{TT^2(\chi\psi)}\big)\,dvd\omega\\
&\quad\leq \int_{\DD}\big(-2\chi\chi'\QQ[\T^2\psi](T, \D\tau)\big)\,dv_{\g}+ 
(\int_0^1+\int_{\widehat{\tau}^*-1}^{\widehat{\tau}^*})
E_{\leq r_0}[T^{\leq2}\psi](\widehat{s})\,d\widehat{s}\\
&\quad \les 
\int_0^1E[T^{\leq2}\psi](\widehat{s})\,d\widehat{s}+\int_{\widehat{\tau}^*-1}^{\widehat{\tau}^*}
E_{\leq r_0}[T^{\leq2}\psi](\widehat{s})\,d\widehat{s}
 \end{split}
 \eeq

Putting \eqref{eq:mainestiamteinlaststep}, \eqref{eq:firstterm}, \eqref{eq:secondterm}, \eqref{eq:thirdterm} ,\eqref{eq:fourthterm} and \eqref{eq:lastterm} together yields
\begin{align*}
\int_{\HH^+}
\big|T_+T^2(\chi\psi)\big|^2\,dv\,d\omega
&\leq \big(\varep C_1(M,K,\sigma,\de)+\frac{C(\de)}{M^2})\BP^1[(T,Z)^{\leq 1}\psi](0,\tauh^*)\\
&+C_2(\varep,M,\sigma,\de)\Big(\int_0^1E[T^{\leq2}\psi](\widehat{s})\,d\widehat{s}+\int_{\widehat{\tau}^*-1}^{\widehat{\tau}^*}
E_{\leq r_0}^1[(T,Z)^{\leq 1}\psi](\widehat{s})\,d\widehat{s}\Big)\\
&+C_2(\varep, M,\sigma,\de)r_0^{-p}\Big(
\int_0^1 +\int_{\widehat{\tau}^*-1}^{\widehat{\tau}^*}
E^1_p[(T,Z)^{\leq1}\psi](\widehat{s})\,d\widehat{s}\Big)
\end{align*}
For any $\epsilon>0$, we first fix $\de>0$.  At the same time, $K,\sigma$ are fixed such that  \eqref{eq:choiceofsigmaK} holds true. Next
$M$ is chosen sufficiently large such that $\frac{C(\de)}{M^2}\leq \frac{\epsilon}{2}$.  Then $\varep$ is chosen sufficiently small such that $\varep C_1(M,K,\sigma,\de)\leq \frac{\epsilon}{2}$. Finally, choose $r_0$ sufficiently large that
$C_2(\varep, M, \sigma,\de)r_0^{-p}\leq \epsilon$.  This is the reason for the large-$r_0$ requirement
in the statement of the theorem. Therefore, we have
\begin{align*}
\int_{\HH^+}
\big|T_+T^2(\chi\psi)\big|^2\,dv\,d\omega
&\leq \epsilon\Big(\BP^2[\psi](0,\tauh^*)+\int_{\widehat{\tau}^*-1}^{\widehat{\tau}^*}
E^2_p[\psi](\widehat{s})\,d\widehat{s}\Big)\\
&+C_2(\epsilon)\Big(\int_0^1
E^2_p[\psi](\widehat{s})\,d\widehat{s}+\int_{\widehat{\tau}^*-1}^{\widehat{\tau}^*}
E_{\leq r_0}^2[\psi](\widehat{s})\,d\widehat{s}\Big).
\end{align*}
Applying finite-in-time energy estimate on $\tauh\in[0,1]$ and $\tauh\in[\tauh^*-1, \tauh^*]$, we obtain 
\[
\int_{\HH^+(0,\tauh^*-1)}
\big|T_+T^2\psi\big|^2\,dv\,d\omega
\leq \varepsilon \BPE_p^2[\psi](0,\tauh^*-1)+C(\varepsilon)\Big(E_p^2[\psi](0)+\int_{\widehat{\tau}^*-1}^{\widehat{\tau}^*}
E_{\leq r_0}^2[\psi](\widehat{s})\,d\widehat{s}\Big)
\]
and this finishes the proof of 
Theorem~\ref{thm:cutoff-horizon-flux-estimate}.



  \section{The $\Wk$-transform for  non  zero spin  wave equations}
  
  \lab{section-spins.waves}


\subsection{Spin-$\sk$ wave equations}

We consider the  spin-$\sk$  wave equation 
 \beq
 \lab{eq:Teuk1}
 \Lks_{a,m}\as=N
 \eeq
where  $\Lks=\Lks_{a,m} $  denotes  the general Teukolsky  wave operator. 
In  BL coordinates\footnote{Here $\Box_{a,m}$ is the standard wave operator in $\KK(a,m)$.}, 
\beq
 \lab{eq:Teuk2}
|q|^2\Lks=|q|^2 \Box_{a,m} +2s(r-m)\partial_r +\As(r,\th)\pr_t +\Bs(r,\th)\pr_\phi +\Vs(r,\th)
\eeq
 with  the complex  scalar  coefficients
\beq
\lab{eq:Teuk3}
\bsplit
\As(r,\th)&= 2s\Big(\frac{m(r^2-a^2)}{\Delta}- q\Big)= 2s \Big(\frac{(r-m)(r^2+a^2)}{\Delta}-2r-ia\cos\theta\Big)\\
\Bs(r,\th)&= 2s \Big(\frac{a(r-m)}{\Delta}+\frac{i \cos\theta}{\sin^2\theta}\Big)\\
\Vs(r,\th)&= \Big(s-s^2\cot^2\theta\Big)
\end{split}
\eeq
\begin{remark}
One often expresses  equation \eqref{eq:Teuk1}  in the renormalized form   
\beq
 \De^{-\frac s 2 }\Lkshat  \big(\De^{\frac s 2 } \as \big)=  N  
\eeq
where 
\[
|q|^2\Lkshat_{a,m}=|q|^2\square_{a,m} + \As\pr_t+\Bs\pr_\phi+ \widehat{V}^{[s]}, \quad \widehat{V}^{[s]}= \Vs - s \big(1+ s\frac{(r-m)^2}{\De}\big).
\]
\end{remark}
As in the case of the standard wave equation  we can decompose  the operator  $\Lks$
 as follows,
 \beq
 \bsplit
 |q|^2 \Lks&=\RR^{[s]} +\OO^{[s]}\\
 \RR^{[s]}&= \RR+ 2 s(r-m)\partial_r +\Big( \Re(\As)\pr_t  + \Re(\Bs)\pr_\phi \Big)\\
 \OO^{[s]}&=\OO + \Big(i\Im(\As)\pr_t  +i \Im(\Bs)\pr_\phi + \Vs \Big).
 \end{split}
 \eeq
 where $\RR$ and $\OO$ are  those of \eqref{eq:RR-OO}. 
 Clearly $\big[\RR^{[s]}, \OO^{[s]} \big]=0$.

\subsubsection{Spin-$\sk$ wave operators  in outgoing coordinates} 
In the outgoing EF coordinates $(u=t-r_*, r, \th, \vphi= \phi-\phi_*)$  the BL coordinate  derivative  $\pr_r$  transforms to 
$-\frac{r^2+a^2}{\De} \pr_u +\pr_r -\frac{a}{\De} \pr_\vphi$. 
Thus $\Lks$ takes the form
\begin{align*}
|q|^2\Lks&= |q|^2 \Box_{\g} +2s(r-m)\partial_r+\Big( \As-2s(r-m)\frac{r^2+a^2}{\De} \Big)\pr_u + \Big(\Bs- 2 s(r-m)\frac{a}{\De} \Big)\pr_\vphi\\
&+\Vs
\end{align*}
and therefore
\beq
\lab{eq:Teuk-outgoing}
\bsplit
  |q|^2\Lks&=|q|^2  \Box_{\g} +2s(r-m)\partial_r+ 2 s \Big(-2r-ia\cos\theta\Big)\partial_u + 2i s \frac{\cos\th}{\sin^2\th}\pr_\vphi\\
  &+\big(s-s^2\cot^2\theta\big).
  \end{split}
\eeq

  
  \subsection{Transformed Teukolsky equations}
  \lab{sect:TransformedTeuk}

   We write the spin-$\sk$ Teukolsky wave operators in the  outgoing  coordinates, see  \eqref{eq:Teuk-outgoing},
  \begin{align*}
  |q|^2\Lks_{a,m}&=|q|^2  \Box_{a,m} +2s(r-m)\partial_r+ 2 s \Big(-2r-ia\cos\theta\Big)\partial_u + 2i s \frac{\cos\th}{\sin^2\th}\pr_\vphi\\
  &+\big(s-s^2\cot^2\theta\big).
\end{align*}

\begin{lemma}
\lab{Le:Lkst}
Under the same assumptions as in Lemma \ref{Lemma:Whiting-metric-general} the following transformation formula holds true
\[
|\qt|^2 \Lkst[\psit]= \Wk\big[|q|^2\Lks\psi\big]
\]
where, with metric $\gt$ as in Lemma \ref{Lemma:Whiting-metric-general},
\beq
\lab{eq:Lkst-psit}
\bsplit
 |\qt|^2\Lkst\psit&=|\qt|^2\square_\gt\psit
-\frac{2s}{S}(X+2)\pr_{\rt}\psit
-2s\bigl((R-m)X+2R+ia\cos\tht\bigr)\pr_{\vt}\psit \\
&+ 2is  \frac{\cos\tht}{\sin^2\tht} \pr_\vphit \psit-\big( s+ s^2 \cot^2 \tht  \big) \psit.
\end{split}
\eeq
Moreover, denoting
\beq
\lab{eq:def-psits}
\psits= (\rt-\Rt)^{-s}\psit,
\eeq
and introducing the operator
\beq
\lab{eq:def-Pgt-spin}
\mathfrak{P}_{\gt}^{[s]}
:=|\qt|^2\square_\gt
-2ias\cos\tht\pr_{\vt}
+\frac{2is\cos\tht}{\sin^2\tht}\pr_{\vphit}-s^2\left(1+\frac2X+\cot^2\tht\right)
\eeq
we derive
\beq
\lab{eq:Lkst-psits}
\mathfrak{P}_{\gt}^{[s]}\psits= (\rt-\Rt)^{-s} \Wk\big[|q|^2\Lks\psi\big].
\eeq
\end{lemma}
\begin{proof}
In view of Lemma \ref{Lemma:Whiting-metric-general}   and  the commutation properties in Lemma \ref{lem:Rk-commutation}
\begin{align*}
\Wk[|q|^2\Lks\psi]&= |\qt|^2\square_\gt \psit+ 2s\Wk[ (r-m)\pr_r \psi]- 4 s \Wk[r \pr_u \psi]\\
&- 2ia s\Wk[\cos\tht   \pr_u \psi] + 2is  \Wk[ \frac{\cos\tht}{\sin^2\tht}\pr_\vphi \psi] +\Wk[\big(s-s^2\cot^2\tht\big)\psi]\\
&=  |\qt|^2\square_\gt \psit+ 2s\Wk[ (r-m)\pr_r \psi]- 4 s \Wk[r \pr_u \psi]\\
&- 2ia s \cos\tht \pr_\vt \psit  + 2is  \frac{\cos\tht}{\sin^2\tht} \pr_\vphit \psit +\big(s-s^2\cot^2\tht\big)\psit
\end{align*}
The commutation identities
\eqref{eq:Rk-commutation} give
\beq
\lab{eq:general-spin-commutations}
\bsplit
\Wk[(r-m)\pr_r\psi]
&=-X\pr_{\vt}\Wk[(r-m)\psi]-\psit\\
&=-\frac{X}{S}\pr_{\rt}\psit-X(R-m)\pr_{\vt}\psit-\psit,\\
\Wk[r\pr_u\psi]
&=\pr_{\vt}\Wk[r\psi]
=\frac1S\pr_{\rt}\psit+R\pr_{\vt}\psit.
\end{split}
\eeq
Consequently,
\beq\lab{eq:general-spin-real-first-order-transform}
\bsplit
2s\Wk[(r-m)\pr_r\psi]-4s\Wk[r\pr_u\psi]
&=-\frac{2s}{S}(X+2)\pr_{\rt}\psit\\
&\quad-2s\bigl((R-m)X+2R\bigr)\pr_{\vt}\psit-2s\psit.
\end{split}
\eeq
  Hence
  \begin{align*}
  \Wk\big[|q|^2\Lks\psi\big]&=|\qt|^2\square_\gt \psit
 -\frac{2s}{S}(X+2)\pr_{\rt}\psit-2s\bigl((R-m)X+2R\bigr)\pr_{\vt}\psit-2s\psit\\
&- 2ia s \cos\tht \pr_\vt \psit  + 2is  \frac{\cos\tht}{\sin^2\tht} \pr_\vphit \psit +\big(s-s^2\cot^2\tht\big)\psit\\
&=|\qt|^2\square_\gt\psit
-\frac{2s}{S}(X+2)\pr_{\rt}\psit
-2s\bigl((R-m)X+2R+ia\cos\tht\bigr)\pr_{\vt}\psit + 2is  \frac{\cos\tht}{\sin^2\tht} \pr_\vphit \psit\\
&-\big( s+ s^2 \cot^2 \tht  \big) \psit
  \end{align*}
  from which \eqref{eq:Lkst-psit} follows.
  
 To remove  the first order derivatives with real coefficients in \eqref{eq:Lkst-psit}, we set  
 \[
  \psits= (\rt-\Rt)^{-s}\psit
  \]
and    derive
\begin{align*}
|\qt|^2\square_{\gt}\psits&=(\rt-\Rt)^{-s}|\qt|^2\square_{\gt}\psit+2\hbt^{\rt\rt}\pr_\rt(\rt-\Rt)^{-s}\c\pr_\rt\psit\\
&+2\hbt^{\rt\vt}\pr_\rt(\rt-\Rt)^{-s}\c\pr_\vt\psit+(|\qt|^2\square_{\gt}(\rt-\Rt)^{-s})\c\psit\\
&=(\rt-\Rt)^{-s}|\qt|^2\square_{\gt}\psit-2s\frac{X+2}{S}(\rt-\Rt)^{-s}\pr_\rt\psit\\
&-2s((R-m)X+2R)(\rt-\Rt)^{-s}\pr_\vt\psit+\Big(s(s-1)+\frac{2s^2}{X}\Big)\psits.
\end{align*}
Combining this with \eqref{eq:Lkst-psit} we finally obtain
\[
\mathfrak{P}_{\gt}^{[s]}\psits= (\rt-\Rt)^{-s} \Wk\big[|q|^2\Lks\psi\big]
\]
as stated in \eqref{eq:Lkst-psits}.
	 \end{proof}

	  \begin{remark}[Comparison with Ma's equations]\lab{rem:Ma-comparison}
	  The transformed spin $\sk=\pm2$ equations \eqref{eq:Lkst-psits}
	  should be compared with the physical-space Chandrasekhar hierarchy
	  used by Ma in the slowly rotating Kerr case
	  \cite[(1.23)--(1.27)]{MaLG}.  There, after applying suitable
	  differential transformations to the Teukolsky variables
	  \cite[(1.23a),(1.23b)]{MaLG}, one obtains the coupled systems
	  \cite[(1.24),(1.25)]{MaLG} of spin-weighted Regge--Wheeler type
	  equations \cite[(1.27a),(1.27b)]{MaLG}.  In the present setting the
	  transformed equations appear at least as favorable, in fact better:
	  the operator $\mathfrak{P}_{\gt}^{[s]}$, defined in \eqref{eq:def-Pgt-spin}, is a scalar wave operator for
	  $\gt$ plus explicit lower order spin terms in the transformed
	  variables, while Ma's original equations contain additional terms
	  coupled with the original Teukolsky variables. 
	  \end{remark}


   \subsection{  Energy estimate  for $\mathfrak{P}_{\gt}^{[s]}\psits=|\qt|^2\Nt$}
   \lab{section:energy-spin-sk}
   
   In this section we show how to extend  the result of Proposition
   \ref{Prop:conf-en(sk-0)} to  the higher spin wave equations.
    Recall, see \eqref{eq:Lkst-psits},
   that 
   \[
\mathfrak{P}_{\gt}^{[s]}\psits= (\rt-\Rt)^{-s} \Wk\big[|q|^2\Lks\psi\big]
\]
and
\[
\mathfrak{P}_{\gt}^{[s]}
=|\qt|^2\square_\gt
-2ias\cos\tht\pr_{\vt}
+\frac{2is\cos\tht}{\sin^2\tht}\pr_{\vphit}-s^2\left(1+\frac2X+\cot^2\tht\right).
\]
  %

\begin{proposition}
\lab{Prop:conf-en(sk)} 
The result of Proposition \ref{Prop:conf-en(sk-0)} extends to  the spin $\sk $-wave  equations 
$\mathfrak{P}_{\gt}^{[s]}\psits=|\qt|^2\Nt$.
\end{proposition}
  In section \ref{section;Prop:conf-en(sk)}   we provide the main  new  ingredients 
  of the proof of the Proposition.


   \subsubsection{Proof of Proposition \ref{Prop:conf-en(sk)}  }
   \lab{section;Prop:conf-en(sk)}
   
   We consider below a slightly more general  class  of wave equations   of the form  
\beq
\lab{eq:ReducedgRW-|q|^2}
\square_\g  \Psi+\frac{1}{|q|^2}  \Big[i \big(  f  \T \Psi  + h  \Z\Psi \big) -v\Psi\Big]=N
\eeq
where
\begin{enumerate}
\item  The metric  $\g$  is  a  stationary, axially symmetric,  asymptotically flat     Lorentz metric, with  the corresponding  Killing vectorfields  $\T, Z$.
\item The functions  $f, h, v$ are real and   $\Psi=\psi+i \psi^*$.
\item The energy momentum  $\QQ_{\mu\nu}=\Re\Big[\pr_\mu \Psi\ov{\pr_\nu \Psi}- \frac 1 2 \g_{\mu\nu} \big(\g^{\a\b} \pr_\a\Psi\ov{ \pr_\b \Psi })\Big]$ verifies
\[
\D^\nu \QQ_{\mu\nu} =\Re\big[\pr_\mu \Psi\c \ov{\square_\g \Psi}\big].
\]
\end{enumerate} 
   We define
   \beq
 P^{\mu}= \QQ^\mu_\nu \T^\nu -  |q|^{-2}h \psi \Big( Z(\psi^*) \T^\mu- \T(\psi^*) \Z^\mu\Big) - \frac 1 2|q|^{-2} v  |\Psi|^2 \T^\mu.
       \eeq
       A straightforward calculation implies that 
       \beq
        \Div P=\Re\big( \ov{ T \Psi} \c N\big).
               \eeq
  Since all the the future normalized normals $\nu$ to the boundaries  are orthogonal to $Z$, i.e. $Z\c\nu=0$, it follows that 
\beq
P\c \nu= \QQ(\T, \nu)-|q|^{-2} h\psi \Z(\psi^*) \T\c \nu  -\frac12|q|^{-2} v \big(|\psi|^2 +|\psi^*|^2 \big)\T\c \nu .
\eeq
 In the conformal energy estimate, since $T\c\nu=0$ at $\HHt^+\cup \IIt^+$, it follows that the energy flux at $\HHt^+\cup \IIt^+$ is the same as in the spin-$0$, scalar wave equation, case.




\appendix

\section{The Schur complement criterion}
\lab{appendix:Schur-complement}

We record the elementary linear algebra fact used in Section
\ref{sec:GenWhitingTr}.

\begin{lemma}[Schur complement positivity criterion]
\lab{lemma:Schur-complement-positivity}
Let
\[
M=
\begin{pmatrix}
A&B\\
B^T&D
\end{pmatrix}
\]
be a real symmetric matrix, where $D$ is symmetric and positive definite.
Then $M$ is positive definite if and only if the Schur complement
\[
S_D:=A-BD^{-1}B^T
\]
is positive definite.
\end{lemma}

\begin{proof}
Since $D$ is positive definite, $D$ is invertible.  We factor $M$ by
elementary block Gaussian elimination:
\[
\begin{pmatrix}
I&BD^{-1}\\
0&I
\end{pmatrix}
\begin{pmatrix}
A-BD^{-1}B^T&0\\
0&D
\end{pmatrix}
\begin{pmatrix}
I&0\\
D^{-1}B^T&I
\end{pmatrix}
=
\begin{pmatrix}
A&B\\
B^T&D
\end{pmatrix}
=M.
\]
Equivalently, for vectors $x,y$ of the corresponding sizes,
\[
\begin{pmatrix}x\\ y\end{pmatrix}^T
M
\begin{pmatrix}x\\ y\end{pmatrix}
=
x^T(A-BD^{-1}B^T)x
+(y+D^{-1}B^Tx)^TD(y+D^{-1}B^Tx).
\]
If $S_D=A-BD^{-1}B^T$ is positive definite, then the right hand side is
strictly positive for every non-zero pair $(x,y)$: if $x\ne0$ the first
term is positive, while if $x=0$ and $y\ne0$ the second term is positive
because $D>0$.  Thus $M>0$.

Conversely, assume $M>0$.  Fix any non-zero $x$ and choose
\[
y=-D^{-1}B^Tx.
\]
Then
\[
0<
\begin{pmatrix}x\\ -D^{-1}B^Tx\end{pmatrix}^T
M
\begin{pmatrix}x\\ -D^{-1}B^Tx\end{pmatrix}
=x^T(A-BD^{-1}B^T)x.
\]
Therefore $S_D$ is positive definite.  This proves the equivalence.
\end{proof}

 \section{ Proof of  the  Erdelyi Lemma} 
 \lab{section:Erdelyi}
  We give below a short informal proof of Lemma \ref{Le:Erdelyi-Xminus}.
To start with we     can assume that $h=h\chi $  where  $\chi $ is supported in a neighborhood  of $r=0$ and $\chi(0)=1$. 
Expanding $h(r)  $   near $r=0$  we  can further reduce  $J(\nu) $ to the integral
\[
  J(\nu) = \int_{0}^\infty e^{- i \nu r}    \chi(r)r ^{ib}  dr
\]
Rescaling  $r=s/\nu$
\[
  J(\nu) = \int_{0}^\infty e^{- i \nu r}    \chi(r)r ^{ib}  dr=\nu^{-1-ib} \int_0^\infty e^{-i  s} s^{ib}\chi(s/\nu) ds 
\]
Thus as $\nu\to \infty$
\[
J(\nu)\to   \nu^{-1-ib}  h(0)  \int_0^\infty e^{-i  s} s^{ib}  ds   
\]
We are thus left with the integral
\[
 \int_0^\infty e^{-i  s} s^{ib}  ds 
\]
We  rotate the contour of  integration  in $\CCC$   by setting 
$s = e^{-i\pi/2} t = -it$, such that
\[
ds = -i\,dt, \quad e^{-is} = e^{-t}, \quad s^{ib}= e^{-i\frac{\pi}{2}  (ib)} =e^{\frac{\pi}{2} b} t^{ib}.
\]
and 
\[
\int_0^\infty e^{-is} s^{ib}  ds=  e^{-i\frac{\pi}{2}}  e^{\frac{\pi}{2} b} \int_0^\infty e^{-t} t^{ib} dt = 
  e^{-i\frac{\pi}{2}}  e^{\frac{\pi}{2} b}\Ga(ib+1).
\]
Hence
\begin{align*}
J(\nu)&\to  \nu^{-1-ib}  h(0)\Big( e^{-i\frac{\pi}{2}}  e^{\frac{\pi}{2} b}\Ga(ib+1)+O(\nu^{-1} ) \Big)\\
&=h(0) \Big( \frac{\Ga(1+ib)} { (\nu)^{1+ib} } e^{-i\pi/2(1+ib) } +O(\nu^{-2} )\Big)
\end{align*}
as stated.
To prove the second part of the Lemma, for $\nu \to -\infty$, we set $\nu =-\mu$
\[
J(\nu)= \int_{0}^\infty e^{ i \mu r}  r ^{ib} h(r) dr  :=J(\mu)
\]
Proceeding as before  we have
\[
J(\mu)=  \mu^{-1-ib}  h(0)  \int_0^\infty e^{i  s} s^{ib}  ds
\]
We then rotate  the contour of  integration  in $\CCC$   by setting 
$s = e^{i\pi/2} t = it$, such that
\[
ds = i\,dt, \quad e^{is} = e^{-t}, \quad s^{ib}= e^{i\frac{\pi}{2}  (ib)} =e^{-\frac{\pi}{2} b} t^{ib}.
\]
and 
\[
\int_0^\infty e^{is} s^{ib}  ds=  e^{i\frac{\pi}{2}}  e^{-\frac{\pi}{2} b} \int_0^\infty e^{-t} t^{ib} dt = 
  e^{i\frac{\pi}{2}}  e^{-\frac{\pi}{2} b}\Ga(ib+1).
\]
Therefore, as $\mu \to \infty$
\begin{align*}
J(\mu)&\to  \mu^{-1-ib}  h(0) \Big(   e^{i\frac{\pi}{2}}  e^{-\frac{\pi}{2} b}\Ga(ib+1) +O(\mu^{-1}) \Big)\\
&= h(0) \Big( \frac{\Ga(1+ib)} { (-\nu)^{1+ib} } e^{i\pi/2(1+ib) } +O(\nu^{-2} )\Big)
\end{align*}
as stated.

\section{The conformal wave equation}
\lab{appendix:conformal-wave-equation}

\begin{lemma}[Conformal wave equation]
\lab{lem:appendix-conformal-wave-equation}
Let $\psit$ solve
$|\qt|^2\square_\gt\psit=|\qt|^2\Nt$. We work in the (compactified) outgoing coordinates $
(\ut,\rhot=|X+1|^{-1},\tht,\vphit_-)$.  
\begin{enumerate}
\item We set
\[
\psihat:=|X+1|\psit,\qquad \ghat:=(X+1)^{-2}\gt .
\]
Then
\beq\lab{eq:appendix-conformal-wave-equationXpos}
|\qt|^2\square_{\gt}\psihat+V(\rt,\tht)\psihat=|X+1||\qt|^2\Nt,
\qquad
V(\rt,\tht)=-\frac{2}{(X+1)^2}.
\eeq
\item In particular, in the region $X\leq -2$, we choose the conformal factor as
\[
|\qt|^2=\sqrt{\sin^2\tht |\det\Hbt|}=\frac{|X|}{|S|}\sqrt{|P(X,\tht)|}
\]
and set
\[
\psihat:=|X+1|\psit,\qquad \ghat:=|X+1|^{-2}\gt .
\]
Then
\beq\lab{eq:appendix-conformal-wave-equation}
\square_{\ghat}\psihat+V(\rt,\tht)\psihat=|X+1|^3\Nt,
\qquad
V(\rt,\tht)=-\frac{2}{|\qt|^2}.
\eeq
\end{enumerate}
\end{lemma}

\begin{proof}
We recall that in outgoing coordinates $(\ut, \rt, \tht, \vphit_-)$,
\[
|\qt|^2\square_{\gt}\psit=\pr_{\at}(\hbt^{\at\bt}\pr_{\bt}\psit)+\frac{1}{\sin\tht}\pr_{\at}(\sin\tht O^{\at\bt}\pr_{\bt}\psit)=|\qt|^2\Nt.
\]
Let $\hbt^{\a'\b'}, O^{\a'\b'}$ be the metric components in the compactified outgoing coordiantes $(\ut, \rhot=|X+1|^{-1},\tht, \vphit_-)$. Then it follows that
\begin{align*}
\pr_{\a'}\hbt^{\a'\b'}\pr_{\b'}+\frac{1}{\sin\tht}\pr_{\a'}\sin\tht O^{\a'\b'}\pr_{\b'}&=\pr_{\at}\hbt^{\at\bt}\pr_{\bt}+\frac{1}{\sin\tht}\pr_{\at}\sin\tht O^{\at\bt}\pr_{\bt}+\pr_{\rhot}(\frac{\pr\rhot}{\pr\rt})\hbt^{\rt\bt}\pr_{\bt}\\
&=\pr_{\at}\hbt^{\at\bt}\pr_{\bt}+\frac{1}{\sin\tht}\pr_{\at}\sin\tht O^{\at\bt}\pr_{\bt}+2\rhot^{-1}\frac{\pr\rhot}{\pr\rt}\hbt^{\rt\bt}\pr_{\bt}.
\end{align*}
Therefore, with $|\qt|^2\square_{\gt}=\pr_{\a'}\hbt^{\a'\b'}\pr_{\b'}+\frac{1}{\sin\tht}\pr_{\a'}\sin\tht O^{\a'\b'}\pr_{\b'}$, we have
\begin{align*}
|\qt|^2\square_{\gt}\psihat&=|\qt|^2\square_{\gt}(|X+1|\psit)=|X+1||\qt|^2\square_{\gt}\psit+2\hbt^{\a'\b'}\pr_{\a'}|X+1|\pr_{\be'}\psit+(|\qt|^2\square_{\gt}|X+1|)\c\psit\\
&=|X+1||\qt|^2\Nt+2|X+1|\rhot^{-1}\frac{\pr\rhot}{\pr\rt}\hbt^{\rt\bt}\pr_{\bt}\psit-2\rhot^{-2}\hbt^{\rt\bt}\pr_{\rt}\rhot\pr_{\bt}\psit+\frac{S}{X+1}(\pr_{\rt}\hbt^{\rt\rt}-2\frac{S}{X+1}\hbt^{\rt\rt})\c\psihat\\
&=|X+1||\qt|^2\Nt+\frac{S}{X+1}(\pr_{\rt}\hbt^{\rt\rt}-2\frac{S}{X+1}\hbt^{\rt\rt})\c\psihat\\
&=|X+1||\qt|^2\Nt+\frac{2}{(X+1)^2}\psihat.
\end{align*}
This proves \eqref{eq:appendix-conformal-wave-equationXpos}.

In the case of $X\leq -2$. Since $\ghat_{\at\bt}=(X+1)^{-2}\gt_{\at\bt}$
\[
\ghat^{\at\bt}=(X+1)^2\gt^{\at\bt},
\qquad
\sqrt{|\det\ghat|}=(X+1)^{-4}\sqrt{|\det\gt|}.
\]
we have
\begin{align*}
\square_{\ghat}\psihat
&=\frac{(X+1)^4}{\sqrt{|\det\gt|}}
\pr_{\at}\Big((X+1)^{-2}\sqrt{|\det\gt|}
\gt^{\at\bt}\pr_{\bt}\psihat\Big)  \\
&=\frac{(X+1)^4}{|\qt|^2\sin\tht}
\pr_{\at}\Big((X+1)^{-2}\sin\tht |\qt|^2\gt^{\at\bt}\pr_{\bt}\psihat\Big)\\
&=\frac{(X+1)^2}{|\qt|^2\sin\tht}
\pr_{\at}\Big(\sin\tht |\qt|^2\gt^{\at\bt}\pr_{\bt}\psihat\Big)+2\frac{(X+1)^2}{|\qt|^2}\rhot^{-1}\frac{\pr\rhot}{\pr\rt}\hbt^{\rt\bt}\pr_{\bt}\psihat\\
&=\frac{|X+1|^2}{|\qt|^2}\frac{1}{\sin\tht}\pr_{\at'}\Big(\sin\tht |\qt|^2\gt^{\at'\bt'}\pr_{\bt'}\psihat\Big)=\frac{|X+1|^2}{|\qt|^2}|\qt|^2\square_{\gt}\psihat.
\end{align*}
This proves \eqref{eq:appendix-conformal-wave-equation}.
\end{proof}

\begin{remark}[Scalar curvature of the conformal metric]
With our convention for \(\square\), the conformal wave operator in dimension
four is $
\square_{\ghat}-\frac16R_{\ghat}.$
Thus the zeroth-order coefficient in Lemma
\ref{lem:appendix-conformal-wave-equation} is  the same as  $V(\rt,\tht)=-\frac16R_{\ghat}$.
\end{remark}




\end{document}